\documentclass[11pt,reqno,a4paper]{article}
\usepackage{amsmath}
\usepackage{amsbsy}
\usepackage{amsthm}
\usepackage{amssymb}
\usepackage{amscd}
\usepackage{accents}
\usepackage{float}
\usepackage{url}
\usepackage{url}
\usepackage{tikz}
\usetikzlibrary{matrix,arrows,decorations.pathmorphing,calc, backgrounds,fit,positioning,shapes}
\usepackage{relsize}
\usepackage[normalem]{ulem}

\usepackage{comment}

\usepackage{enumerate}

\usepackage{multicol}
\usepackage{longtable}

\usepackage{hyperref}
\hypersetup{
pdftitle={},%
pdfauthor={},%
pdfsubject={},%
pdfkeywords={},%
colorlinks=true,%
linkcolor={black},%
linktoc={},
linktocpage={false},%
pageanchor={},
citecolor={black},
}

\usepackage[english]{babel}
\usepackage[utf8]{inputenc}
\usepackage[margin=2.41cm]{geometry}
\usepackage{mathrsfs}
\usepackage[all]{xy}
\usepackage{hyperref}
\usepackage{url}

\usepackage{cite}

\usepackage{titlesec}
\titleformat{\section}{\large\bfseries\filcenter}{\thesection}{1em}{}
\titleformat{\subsection}{\bfseries}{\thesubsection}{1em}{}

\usepackage{etoolbox}

\makeatletter
\patchcmd{\ttlh@hang}{\parindent\z@}{\parindent\z@\leavevmode}{}{}
\patchcmd{\ttlh@hang}{\noindent}{}{}{}
\makeatother

\newtheorem{thm}{Theorem}[section]
\newtheorem{cor}[thm]{Corollary}
\newtheorem{lemma}[thm]{Lemma}

\newtheorem{prop}[thm]{Proposition}
\newtheorem{ass}[thm]{Assumption}

\newtheorem{conj}[thm]{Conjecture}

\theoremstyle{definition}
\newtheorem{defn}[thm]{Definition}

\theoremstyle{remark}

\theoremstyle{definition}

\newtheorem{rmk}[thm]{Remark}
\newtheorem{remark}[thm]{Remark}

\theoremstyle{example}
\newtheorem{exa}[thm]{Example}

\numberwithin{equation}{section}
\allowdisplaybreaks[1]

\catcode`@=12

\def\beq{\begin{equation}}
\def\eeq{\end{equation}}

\def\beqn{\begin{equation*}}
\def\eeqn{\end{equation*}}

\def\ben{\begin{enumerate}}
\def\een{\end{enumerate}}

\usepackage{mathtools}
\usepackage{tikz-cd}

\renewcommand\thanks[1]{%
  \begingroup
  \renewcommand\thefootnote{}\footnote{#1}%
  \addtocounter{footnote}{-1}%
  \endgroup
}

\renewcommand{\tilde}{\widetilde}
\renewcommand{\epsilon}{{\varepsilon}}

\def\0{{(0,0)}}

\def\C{{\mathbb C}}

\def\L{{\mathbb L}}
\def\Q{{\mathbb Q}}

\def\N{{\mathbb N}}
\renewcommand\P{{\mathbb P}}
\def\R{{\mathbb R}}

\def\Z{{\mathbb Z}}
\def\bS{\mathbb{S}}

\def\im{{\rm Im\,}}
\def\ker{{\rm Ker\,}}
\def\coker{{\rm Coker\,}}

\def\cf{\emph{cf.}~}
\def\ie{\emph{i.e.}~}

\def\eg{\emph{e.g.}~}

\def\cO{{\mathcal O}}

\def\cU{{\mathcal U}}

\def\sI{{\mathscr I}}

\def\sC{{\mathscr C}}
\def\sD{{\mathscr D}}

\def\sF{{\mathscr F}}

\def\sI{{\mathscr I}}

\def\sP{{\mathscr P}}

\def\bA{{\mathbf A}}

\def\bC{{\mathbf C}}
\def\bD{{\mathbf D}}

\def\bG{{\mathbf G}}
\def\bH{{\mathbf H}}

\def\bN{{\mathbf N}}

\def\bP{{\mathbf P}}
\def\bQ{{\mathbf Q}}

\def\bS{{\mathbf S}}

\def\bZ{{\mathbf Z}}

\def\a{\alpha}

\def\be{\beta}

\def\ol{\overline}
\def\Id{{\rm Id\,}}

\def\Hom{{\rm Hom\,}}
\def\Coeq{{\rm Coeq\,}}
\def\eq{{\rm eq\,}}

\def\Spec{{\rm Spec}}
\def\Zar{{\rm Zar\,}}

\def\op{{\rm op \,}}
\def\disc{{\rm disc\,}}
\def\ab{{\rm ab\,}}
\def\fin{{\rm rel\,}}
\def\fine{{\rm fine\,}}

\def\cofib{{\rm cofib\,}}

\def\id{{\rm id\,}}
\def\op{{\rm op\,}}

\def\lt{\langle}
\def\gt{\rangle}
\def\ul{\underline}
\def\ol{\overline}

\newcommand{\minus}{\scalebox{0.55}[1.0]{$-$}}

\def\Gelf{\mathscr{G}^{\mathtt {NC}}}

\def\bMod{{\mathbf{Mod}}}
\def\bSMod{{\mathbf{SMod}}}
\def\bBiMod{{\mathbf{BiMod}}}
\def\bSites{{\mathbf{Sites}}}
\def\bZar{{\mathbf{Zar}}}

\def\bOuv{{\mathbf{Ouv}}}

\def\bTop{{\mathbf{Top}}}

\def\bCoh{{\mathbf{Coh}}}
\def\bCo{{\mathbf{Co}}}

\def\bSets{{\mathbf{Sets}}}
\def\bAb{{\mathbf{Ab}}}

\def\bRings{{\mathbf{Rings}}}

\def\bDGA{{\mathbf{DGA}}}
\def\bCh{{\mathbf{Ch}}}

\def\bSRings{{\mathbf{SRings}}}

\def\bHo{{\mathbf{Ho}}}
\def\bAff{{\mathbf{Aff}}}
\def\bdAff{{\mathbf{dAff}}}
\def\bLoc{{\mathbf{Loc}}}
\def\bAlg{{\mathbf{Alg}}}
\def\bCAlg{{\mathbf{CAlg}}}

\def\bFunc{{\mathbf{Func}}}
\def\bPsh{{\mathbf{Psh}}}
\def\bSh{{\mathbf{Sh}}}
\def\bPoly{{\mathbf{Poly}}}
\def\bFree{{\mathbf{Free}}}

\def\bSma{{\mathbf{Sma}}}
\def\bThick{{\mathbf{Thick}}}

\def\SpecNC{{\mathrm{Spec}^{\hspace{-0.4mm}\tiny{\mathtt{NC}}}\hspace{-0.5mm}}}
\def\SpecG{{\mathrm{Spec}^{\hspace{-0.45mm}\tiny{\mathtt{G}}}\hspace{-0.5mm}}}

\def\void{{\rm \varnothing}}
\def\then{\Rightarrow}

\def\limpro{\mathop{\lim\limits_{\displaystyle\leftarrow}}}

\def\limind{\mathop{\lim\limits_{\displaystyle\rightarrow}}}

\def\ootimes{{\ol{\otimes}}}

\newcommand{\bigast}{\mathop{\Huge \mathlarger{\mathlarger{\ast}}}}

\begin{document}

\begin{flushright}

\baselineskip=4pt

\end{flushright}

\begin{center}
\vspace{5mm}

{\LARGE\bf {Noncommutative Gelfand Duality: \\[3mm]
the algebraic case}}
\medskip

%\thanks{F.B.\ is supported the DFG research grant BA 6560/1-1 ``Derived geometry and arithmetic''. M.C. gratefully acknowledges th Department of Mathematics of University College London, where part of this work was conducted.
%S.M. acknowledges the support
%by the research grant “Geometric boundary value problems for the Dirac operator”
%of the Juniorprofessurenprogramm Baden-W\"rttemberg as well as that by the DFG
%research training groups GRK
%1821 ``Cohomological Methods in Geometry'' during the initial stages of this project.}
% 

\vspace{5mm}

{\bf by}

\vspace{5mm}

 { \bf Federico Bambozzi$^1$, Matteo Capoferri$^{2,3}$, Simone Murro$^4$ }\\[1mm]
\noindent  {\it $^1$ Dipartimento di Matematica, Università di Padova \& GNSAGA-INdAM,  Italy}\\[1mm]
\noindent  {\it $^2$ Maxwell Institute for Mathematical Sciences \& Department of Mathematics, Heriot-Watt University, Edinburgh EH14~4AS, UK}\\[1mm]
\noindent  {\it $^3$ Dipartimento di Matematica ``Federigo Enriques'', Universit\`a degli Studi di Milano, Italy}\\[1mm]
\noindent  {\it $^4$Dipartimento di Matematica, Universit\`a di Genova \& GNFM-INdAM \& INFN, Italy}\\[1mm]
E-mail: {\tt federico.bambozzi@unipd.it, matteo.capoferri@unimi.it, simone.murro@unige.it}
\\[10mm]
\end{center}

\begin{abstract}

The goal of this paper is to define a notion of \emph{non-commutative Gelfand duality}.
Using techniques from derived algebraic geometry, we show that the category of rings is anti-equivalent to a subcategory of pre-ringed sites, inspired by Grothendieck's work on commutative rings. Our notion of spectrum, although formally reminiscent of the Grothendieck spectrum, is \emph{new}.
%; in particular, for commutative rings it does not always agree with the Grothendieck spectrum, but it always projects onto it. 
Remarkably, an appropriately refined relative version of our spectrum agrees with the Grothendieck spectrum for finitely generated commutative $\mathbb{C}$-algebras, among others. This work aims to represent the starting point for a rigorous study of geometric properties of quantum spacetimes. 
\end{abstract}

\paragraph*{Keywords:} derived geometry, homotopical epimorphism, noncommutative Gelfand duality, noncommutative space, noncommutative spectrum.
\paragraph*{MSC 2020:} Primary: 14A22; Secondary: 14A30, 16E35, 18F20.
\\[0.5mm]
\tableofcontents

\renewcommand{\thefootnote}{\arabic{footnote}}
\setcounter{footnote}{0}

\section{Introduction and main result}

The notion of duality lies at the heart of many of the paradigm shifts that took place in mathematics and physics in the last couple of centuries. On the one hand, switching perspective brings a breath of fresh air and new insight in situations where progress seemed precluded.
On the other hand, changing point of view reveals new aspects that remained otherwise hidden. In physics this is the case, for example, for Einstein's theory of General Relativity: describing the gravitational interactions in purely geometric terms unveiled striking and unexpected physical phenomena, such as black holes, gravitational deflection of light, and gravitational waves, which could not be seen through the prism of Newton's classical theory of gravitation. In mathematics, to name but one example, we can mention Grothendieck's theory of schemes, which has been used to define \'etale cohomology leading to the solution of the Riemann hypothesis for varieties over finite fields.

In this paper, we are interested  in the duality between algebraic structures and geometric spaces. 
One of the major contributions in this area is due to Gelfand. In his seminal paper~\cite{gelfand}, he showed that a compact Hausdorff topological space can be functorially reconstructed from its Banach algebra of continuous functions. Conversely, the \emph{Gelfand spectrum} of the commutative $C^*$-algebra of continuous functions is homeomorphic to the underlying compact Hausdorff topological space.  
Loosely speaking, Gelfand promoted the idea that any abstract commutative complex Banach algebra can be represented as the ring of functions on a topological space whose points are in correspondence with the maximal ideals of the algebra. It is no overstatement to say that this deep intuition had a transformative effect across all areas of pure mathematics, from functional analysis to --- crucially ---  modern algebraic geometry, and continues to inform and inspire cutting-edge research to this day.

\paragraph{Algebraic geometry and the spectrum of a commutative ring.}

The branch of mathematics where Gelfand's ideas have had the strongest impact is algebraic geometry. The modern foundations of algebraic geometry were laid in the second half of the twentieth century by the French school led by Grothendieck, based on a variant of Gelfand duality for the category of commutative rings. In this case, to a commutative ring one associates a topological space --- called \emph{spectrum} --- whose points correspond to the prime ideals of the ring, and whose topology is defined in the style of the topology defined by Zariski in his study of algebraic varieties over $\C$. Elements of the ring are interpreted as functions over the spectrum, so that a perfect duality between spaces and commutative rings is obtained via the theory of locally ringed spaces, briefly recalled in Section~\ref{sec:gelfand}. In this work, we focus on generalising to noncommutative rings the duality between commutative rings and affine schemes pertinent to algebraic geometry. However, the essential ideas behind of our methods apply to the $C^*$-algebra setting as well, or even to more general situations. To obtain the latter (analytic) versions of our noncommutative Gelfand duality, one is required to work in a different framework, one where the methods from homological algebra can be applied effectively to the theory of Banach spaces, or more general functional-analytic settings. This is the case, for example, for the bornological HAG context developed in the series of papers \cite{bambi3,bambi4,bambi2,bambi,BamMi,BamMi2} (see also the recent \cite{BKK} for a more comprehensive treatment of the subject). Alternatively, one could also work in the framework of condensed mathematics set out in \cite{Scholze}. In the present work, for the sake of clarity, we avoid these further complications and stick with the algebraic setting where the methods of homological and homotopical algebra are applied in the classical setting of the categories of modules over a ring. Derived categories and homological algebra are the key tools we will use to move from the commutative to the noncommutative setting.

\paragraph{Motivation and heuristics for noncommutative dualities.} Despite being of interest in its own right, an important motivation for seeking a noncommutive analogue of the Gelfand duality originates from mathematical physics, in that such a duality could lead to the rigorous formulation of a \emph{quantum} theory of gravitation. Indeed, although a full quantisation scheme for Einstein's equation is still beyond reach, it is clear from a physical perspective that spacetime should eventually become ``noncommutative''~\cite{Froeb}.
The general idea is to construct a noncommutative algebra which possesses a ring epimorphism to the ring of smooth functions over a manifolds, the latter being the  ring of functions of the spacetime at hand via the smooth version of the Gelfand duality~\cite{smoothGelfand}. 
This procedure is conceptually analogous to the quantisation of the classical phase space in the quantum-mechanical setting.
This physical principle has been realised in a variety of concrete approaches to the problem, including  Connes's noncommutative spectral
triples~\cite{connes1,connes2}, Lorentzian spectral triples~\cite{Franco,verch}, and deformation quantisation~\cite{Buchholz,waldmann,Grosse}.
A very popular strategy is to promote Cartesian coordinates describing the flat Minkowski
spacetime to operators, and postulate commutation relations in terms of a constant skew-symmetric matrix whose entries are of the order of the Planck length, resulting in the so-called DFR spacetime~\cite{DFR}. The spacetime is then interpreted as a noncommutative Frechét algebra which becomes commutative in the classical limit.
The main shortcoming in all these approaches is that the definition of ``noncommutative space'' is realised only at a formal level, so that the geometric properties of such noncommutative space are difficult to capture; indeed, the existing strategies do not even lead to a satisfactory definition of what a point of a noncommutative space would be.
Remarkably, there has been evidence that noncommutative geometry manifests itself ``in real life"; for instance, 
%\sout{it makes its appearance in the mathematical description of aperiodic crystals~\cite{Bellissard1,Bellissard2} or in dirty superconductors~\cite{DeNittis} } 
aperiodic crystals~\cite{Bellissard1,Bellissard2} and dirty superconductors~\cite{DeNittis} are mathematically modelled by noncommutative tori. Therefore, it would be rather interesting to study the spectrum of a noncommutative torus or more generally of a symplectic twisted group algebras, see e.g.~\cite{NCtorus,NCtori}.
\color{black}

\paragraph{Reyes's no-go theorem.} 
In the literature, there have been many attempts aimed at extending the Gelfand spectrum to the
\emph{noncommutive} $C^*$-algebra. In parallel, there have been efforts to extend the definition of the \emph{spectrum} functor from the category of commutative rings to the category of rings. 
A robust notion of noncommutative spectrum should possess, at the very minimum, the following properties.
\begin{itemize}
\item[(A)] The spectrum of every
nonzero ring is nonempty.
\item[(B)] The spectrum, viewed in the appropriate framework, is a contravariant functor from the category of rings to the category
of sets.
\end{itemize} 

These two properties are usually required to hold in conjunction with the extension property

\begin{itemize}
\item[(*)] The notion of spectrum defined on the category of all rings must agree with the Grothendieck spectrum of prime ideals when restricted to commutative rings.
\end{itemize} 

Goldman’s prime torsion theories~\cite{Goldman} and the ``left spectrum'' of Rosenberg~\cite{rosenberg} are examples of extensions of the spectrum of prime ideals for which only property (A) is satisfied. Examples of approaches that satisfy property (B) are the spectrum of the ``abelianization'' $R \mapsto \Spec(R/[R, R])$, the set of completely prime ideals, and the ``field
spectrum'' of Cohn~\cite{Cohn}.
In fact, it turns out that any extension of the Grothendieck spectrum of primes ideals leads to a notion of noncommutative spectrum that can only satisfy either property (A) or property (B) but not both.
This was shown quite recently by Reyes in~\cite{nogo}, rigorously demonstrating that a functor satisfying the properties (A), (B) and (*) cannot exist. Therefore, if one wants to successfully extend the notion of spectrum to noncommutative rings, a different and, in a way, more sophisticated approach is needed. In \cite{nogo}, Reyes has proved a similar no-go theorem for $C^*$-algebras and the Gelfand spectrum. In the present paper, we focus on the algebraic setting, leaving for future work a similar study for $C^*$-algebras (and more general algebras) in the aforementioned bornological HAG framework of \cite{bambi2} and \cite{BKK}.

\medskip

\paragraph{Main result.}
The goal of this paper is to define a sufficiently robust notion of \emph{noncommutative} spectrum that allows one to implement a \emph{noncommutative Gelfand duality}. 

We should immediately emphasise that the most appropriate framework in which to develop a theory of noncommutative schemes would be that of $\infty$-categories. However, in this paper we use derived and homotopy categories, instead. The main underlying reason for this choice is that we would like to make our paper accessible to a rather diverse readership, a large portion of which may not be familiar with $\infty$-categories\footnote{Readers with a background in the theory of $\infty$-categories will have no problem in translating our results in that framework.}.
Indeed, the duality we will be developing sits at the intersection of many different areas of mathematics, ranging from the very pure to the more applied, and different communities may find our results useful. As a final disclaimer, in this manuscript we only consider the categories of rings $\bRings_\Z$ and homotopical rings $\bH \bRings_\Z$, postponing until a future publication the more technically demanding case of Banach algebras. 

\

The following represents the central result of the paper, see Corollary~\ref{cor:Gelfand_duality}.

\begin{thm}[Classical form of noncommutative Gelfand duality]\label{thm:main}
The category $\bRings_\Z$ is anti-equivalent to a subcategory of \emph{pre-ringed sites} $\mathbf{PreRingSites}$ (as per  Definition~\ref{defn:pre-ringed_space}).
\end{thm}

In the remainder of the introduction we shall explain the meaning of the above result and outline how it is obtained. The first thing to notice is that, although the statement of Theorem~\ref{thm:main} is phrased in classical terms, it appears to be unreachable working within such framework. The main reason why one is forced to go beyond the category of rings $\bRings_\Z$ is that the spectrum of noncommutative rings have open subsets whose natural algebra of functions is a \emph{differential graded algebra} (dg-algebra), unlike the commutative case where such open subsets do not exist. A similar phenomenon happens in analytic geometry, and its study recently led to the resolution to the so-called \emph{sheafyness problem} for Huber spaces \cite{bambi} (a similar solution has been give also by Clausen and Scholze). If we denote by $\bH\bRings_\Z := \bHo(\bD\bG\bA_\Z^{\le 0})$ the homotopy category of \emph{connective dg-algebras over $\Z$}, our main result shows Gelfand duality for $\bH\bRings_\Z$; the classical formulation of Theorem~\ref{thm:main} is then obtained by restricting the latter via the canonical fully faithful inclusion $\bRings_\Z \subset \bH\bRings_\Z$. We would like to stress that the data needed to obtain such a faithful representation of a ring on a topological space (\ie the structure pre-sheaf of functions on the spectrum) is not concentrated in degree $0$ even if we start with a ring concentrated in degree $0$ (see Example~\ref{exa:homotopy_epimorphism} for more details on this). We also observe that if one accepts to use $\bH\bRings_\Z$ as the basic category to study, then our approach represents one of the easiest and most natural ways to develop a geometric study of this category.

The (anti-)equivalence of Theorem~\ref{thm:main} is obtained as a two-step procedure. First, we construct a functor
that assigns to any connective dg-algebra $R\in\bH\bRings_\Z$ a topological space $\SpecNC(R)$, that we call \emph{noncommutative spectrum}, and to any algebra homomorphism a continuous map of the associated topological spaces. 
Then we construct a homotopical pre-sheaf of connective dg-algebras on $\SpecNC(R)$, so that the global section $\mathcal{O}_X(\SpecNC(R))$ is isomorphic to $R$ in $\bH\bRings_\Z$. 
Combining these two steps, we obtain a contravariant duality functor
\begin{align*}
    \Gelf : & \bH\bRings_\Z \to \mathbf{PreRingSites} 
\end{align*} 
that we call \emph{noncommutative Gelfand duality}. By this we mean that the functor $\Gelf$ is a faithful functor. This is reminiscent of how affine schemes can be described as a (non-full) subcategory of $\mathbf{PreRingSp}$. We elaborate further on this matter in Section~\ref{sec:gelfand}, especially on our use of pre-ringed sites instead of (classical) ringed spaces.

Our notion of spectrum, although formally similar the Grothendieck spectrum of a commutative ring, is \emph{new} and does not usually agree with it, unless very specific cases are considered, like fields and commutative discrete valuation rings (see Section~\ref{exa:spectra}). However, given a commutative ring, it is always possible to compare 
% a canonical map of pre-ringed spaces\footnote{Actually, for commutative rings we have that $\SpecNC(R)$ is a ringed space --- see Proposition~\ref{prop:sheafy}.} from the 
our noncommutative spectrum with the Grothendieck spectrum (see Propositions~\ref{prop:compatision with Grothendieck} and~\ref{prop:projection_spectra_prepringed}).
 Remarkably, it is possible to define a relative version of our spectrum that agrees with Grothendieck spectrum for finitely generated commutative $\mathbb{C}$-algebras (see Subsection~\ref{sec:relative spectrum}). 

\paragraph{Challenges and novelty.}
To achieve our goal, one has to overcome a series of difficulties. The first consists in devising a Grothendieck topology suitable for noncommutative rings. In analogy with the classical Zariski topology of scheme theory, we defined what we call the \emph{formal homotopy Zariski topology} on $\bdAff_\Z=\bH\bRings_\Z^\op$ (see Defition~\ref{defn:homotopical_Zariski_open_immersion}). It is a highly non-trivial fact that such a topology even exists. Natural candidates for generalisations of the Zariski topology in $\bRings_\Z^\op$ fail to be suitable classes of maps. In the category of commutative rings $\bC\bRings_\Z$, the Zariski topology is defined by the class of flat epimorphisms of finite presentation. But flat maps are not stable by pushouts in $\bRings_\Z$, thus cannot be used to define a topology. On the other hand, the class of all epimorphisms of rings obviously defines a topology but it is too large, encompassing morphisms that do not correspond to Zariski open immersions of classical affine schemes. Therefore, the task is to find a class of morphisms that is not quite flat but close to be flat, such that in the commutative case it reduces to the class of flat epimorphisms. Surprisingly, such a class of morphisms does exist in $\bH\bRings_\Z$. This is the class of \emph{homotopical epimorphisms}. Similarly to the class of epimorphisms in $\bRings_\Z$, the class of homotopical epimorphisms is suitable to define a topology on $\bH\bRings_\Z$. Moreover, the computations of \cite{Laz2} show that the commutative and noncommutative notions of homotopical epimorphisms are compatible. Even more, \cite{TV3} shows that a morphism between discrete commutative rings is a flat epimorphism of finite presentation if and only if it is a homotopical epimorphism of finite presentation.

Let us now move to the second main obstacle. In all our previous considerations, we have been using morphisms of \emph{finite presentation}. It turns out that the notions of finite presentation in the commutative and in the noncommutative settings are fundamentally incompatible. We will devote Subsection~\ref{On the notion of finite presentation} to clarifying this matter. To circumvent such compatibility issue, unlike for the classical Zariski topology, in our definition we do not require the localization morphisms to be of finite presentation.
Dropping this requirement turns out to be essential to obtain a meaningful spectrum out of our topology on $\bdAff_\Z$. Indeed, 
if one assumes the finite presentation of the localization morphisms in $\bH \bRings_\Z$, then the topological space obtained as spectrum will be just a singleton for most commutative rings (see Proposition~\ref{prop:wrong_spectrum}).

The third main challenge one faces is to come up with a good notion of coverage in $\bH\bRings_\Z$. Now, in the commutative setting the notion of coverage is traditionally defined by requiring that the family of base change functors at the level of modules associated with localizations is a faithful family of functors. A closer examination reveals that the latter is not a suitable notion for the noncommutative setting, in that it does not satisfy the axioms of a Grothendieck topology. As it turns out, one can circumvent this issue by working directly in the category $\bH\bRings_\Z$ and requiring conservativity for the base change functors at the level of algebras. The price one has to pay is that the notion of coverage one obtains is harder to handle in many respects and, furthermore, a direct comparison with the classical Zariski topology is precluded. With this in mind, we introduce a stronger notion of cover leading to a finer topology --- the \emph{fine Zariski topology} (see Definition~\ref{defn:fine_Zariski_topology}) --- obtained by imposing the conservativity of the base change only for the subclass of \emph{homotopical epimorphisms}. As a matter of fact, this finer topology is no longer a Grothendieck topology on $\bH\bRings_\Z$ --- in general it is only functorial for homotopical epimorphism --- but it is easier to compute and, when worked out for commutative rings, it allows one to compare the topological spaces endowed with the different topologies.

Using the homotopical Zariski topology briefly outlined above, one can associate to any fixed object $A \in \bH\bRings_\Z$ the (essentially small) category of localizations (\ie homotopical epimorphism) starting from $A$. This category is, in fact, a poset that agrees with the poset of smashing localizations of $\bH \bMod_A$ (see Proposition \ref{prop:iso_smashing_poset}). This kind of posets have already been considered in the literature, mainly in the context of triangulated categories. One of the main observations about them is that these posets tend not to be distributive when the ring $A$ is not commutative. Therefore, it is not clear how to use them to obtain a topological space in a similar fashion to how Stone's duality associates a spectral topological space to any distributive lattice. In our setting, we have the extra datum of the homotopical Zariski topology, which endows the lattice of homotopical epimorphisms with a notion of coverage that allows one to canonically associate to it a sober topological space. The topological space thus obtained is denoted by $\SpecNC(A)$, is functorially associated to $A$ and always non-empty (actually, often of huge size). In Section~\ref{exa:spectra} we provide a description of the topological space $\SpecNC(A)$ for some key examples, both for commutative and noncommutative rings. Remarkably, the same construction also applies to the fine Zariski topology;  the resulting topological space will be denoted $\SpecNC_\fine$. As already mentioned, this association is not functorial, but it always admit canonical maps to both $\SpecNC$ and to $\SpecG$, allowing one to compare our spectrum $\SpecNC$ with Grothendieck's $\SpecG$ (see Proposition~\ref{prop:compatision with Grothendieck}).

This brings us to the last hurdle we faced in proving Theorem~\ref{thm:main}. Already in the setting of the theory of affine schemes, the topological space $\SpecG(A)$ does not contain enough information to allow one to reconstruct the ring from which it originated. The solution to this problem is to add to $\SpecG(A)$ the extra datum of a structure sheaf of rings. One would like to perform the same enrichment for the topological space $\SpecNC(A)$, with the requirement that the pullback of the structure sheaf thereon to $\SpecNC_\fine(A)$ remains a sheaf. It turns out that for noncommutative rings the structure presheaf on $\SpecNC_\fine(A)$ is not a sheaf and its sheafification does not satisfy the properties expected of a structure sheaf (see Example~\ref{exa:structure_not_sheaf}).

 We do not view this as a shortcoming of our construction, but actually as an interesting feature deserving deeper study and understanding. Indeed, although the structure pre-sheaf is not a sheaf in the usual sense, it still satisfies a form of comonadic descent that is reminiscent of the sheaf condition (see Theorem~\ref{thm:comonadic_descent} ), and actually reduces to it in the commutative case (see Proposition \ref{prop:sheafy}). Therefore, the structure pre-sheaf belong to a new class of pre-sheaves that we call \emph{descendable}, whose sections have the property of being reconstructible from their localizations on a cover, although this reconstruction procedure does not agree with the sheaf theoretic procedure in general. We will devote Section~\ref{sec:gelfand} to the study of this issue and to the proof of Gelfand duality.

\paragraph{Future outlook.} We would like to point out that the Gelfand duality  obtained in Theorem~\ref{thm:main} is not yet a perfect duality, and there is room for improvement. Let us elaborate on this.

For starters, we expect that the structural presheaf on $\SpecNC$ is, in fact, a sheaf --- see Conjecture~\ref{conj:structure_sheaf}.
This would automatically upgrade the Gelfand duality to a contravariant equivalence of $\bH\bRings_\Z$ to a subcategory of $\mathbf{PreRingSp}$.

Furthermore, one would like to
explicitly described the essential image of the duality functor $\Gelf$. This is analogous to describing the category of affine schemes as a non-full subcategory of the category of ringed spaces. In the latter situation, it has been realised that affine schemes embeds fully faithfully into the category of locally ringed spaces, giving a precise description of the essential image of the duality functor for affine schemes. We explain this matter in Section~\ref{sec:gelfand}. 
    The study of the essential image of the functor $\Gelf$ will be addressed in a separate work.

\paragraph{Structure of the paper.}

The paper is organised as follows. 

Section \ref{sec:preliminaries} contains some background material and fixes the notation used in the rest of the paper. 

Section \ref{sec:homotopy_algebras} recalls materials from homotopical algebra and, in particular, studies the differences (and incompatibility) between the notions of finite presentation of an algebra in the commutative and in the noncommutative settings. This incompatibility has profound implications on the relationship between the commutative and noncommutative versions of the theory of schemes. 

In Section~\ref{sec:spectrum}  we define the homotopical Zariski topology on $\bH \bRings_\Z^\op$ and explain how to construct a spectrum functor out of this datum. 

In Section~\ref{sec:properties_spectrum} we study some basic properties of the spectrum functor obtained previously and compare it with other notions present in literature. We introduce the fine Zariski spectrum as the main tool to achieve this goal.
%In particular, we show that our spectrum always projects onto Grothendieck's spectrum of prime ideals when the ring is commutative, and that it can be computed using the lattice of smashing subcategories of the category of (indifferently, left or right) modules. 

In Section~\ref{exa:spectra} we provide some explicit computations of the topological space $\SpecNC(A)$ and $\SpecNC_\fine(A)$ for some key examples, both commutative and noncommutative. 

In Section~\ref{sec:gelfand} we examine the issue of descent of modules for the (fine) homotopical Zariski topology. We show that modules can be reconstructed from their localizations using the machinery of comonadic descent. This, combined with the observation that the structure sheaf on $\SpecNC_\fine$ is not a sheaf for noncommutative rings, leads to the observation that the classical notion of sheaf is not adequate for noncommutative geometry and some sort of generalisation thereof is needed to progress further. We provide a first attempt in this direction and verify that our result is consistent with the theory developed in the commutative case. This leads us to a duality theory for $\bH \bRings_\Z$ that generalises and is compatible with the duality between affine schemes and commutative rings.

\subsection*{Acknowledgments}
We are grateful to Lidia Angeleri H\"ugel, Silvana Bazzoni, Ryo Kanda, Kobi Kremnizer, Rosanna Laking, Filippo Papallo, Nicola Pinamonti, Pino Rosolini, Manuel Saorin, Alexander Schenkel, Matteo Varbaro for helpful discussions related to the topic of this paper.

The authors acknowledge the affiliation to the GNAMPA, GNFM, and GNSAGA groups of INdAM.

\subsection*{Funding}
FB was partially supported by DFG research grant BA 6560/1-1 ``Derived geometry and arithmetic''. MC was supported by EPSRC Fellowship EP/X01021X/1. SM was partially supported by the MIUR Excellence Department Project 2023-2027 awarded to the Department of Mathematics of the University of Genoa. 

The completion of this paper was made possible by a Research-in-Groups programme funded by the International Centre for Mathematical Sciences, Edinburgh.

\subsection*{Notation}
\addcontentsline{toc}{subsection}{Notation}

The following table summarises the notation most often used throughout the paper.

%\begin{multicols}{2}
\begin{longtable}{l l}
\hline
\\ [-1em]
\multicolumn{1}{c}{\textbf{Symbol}} & 
  \multicolumn{1}{c}{\textbf{Description}} \\ \\ [-1em]
 \hline \hline \\ [-1em]
$\bH\bC\bRings_R:=\bHo(\bS\bC\bRings_R)$  & homotopy category of \emph{commutative simplicial rings over $R$} \\ \\ [-1em]
$\bH\bZ\bRings_R:=\bHo(\bZ\bD\bG\bA_R^{\le 0})$  & homotopy category of \emph{connective central dg-algebras over $R$} \\ \\ [-1em]
$\bH\bRings_R:=\bHo(\bD\bG\bA_R^{\le 0})$  & homotopy category of \emph{connective dg-algebras over $R$} \\ \\ [-1em]
$\bH\bMod_R:=\bHo(\bSMod_R)$  & homotopy category of \emph{simplicial right modules over $R$} \\ \\ [-1em]
$\bdAff_R := \bH\bRings_R^\op$ & category of \emph{derived noncommutative affine schemes}  \\ \\ [-1em]
$\bLoc(R)$ & category of \emph{localizations} of $R\in\bH\bRings_\Z$ (see Def.~\ref{def:loc-ouv}) \\ \\ [-1em]
$\bOuv(X)$ & category of \emph{open immersions} of $X\in\bdAff_\Z$ (see Def.~\ref{def:loc-ouv}) \\ \\ [-1em] 
$\bZar_X $ & small Zariski site (see Def.~\ref{defn:small_Zatiski_site})\\  \\ [-1em] 
$\mathbf{PreRingSites}$  & category of \emph{pre-ringed sites} (see Def.~\ref{defn:pre-ringed_space}) \\ \\ [-1em] 
$\textrm{Spec} :  \;\bH\bRings_\Z \to \bdAff_\Z$ & formal duality functor (see Sect.~\ref{sec:gelfand}) \\ \\ [-1em]
$\SpecNC :  \;\bH\bRings_R \to \bTop$ & noncommutative spectrum (see Sect.~\ref{sec:defncspec}) \\ \\ [-1em]
$\SpecNC_\fine :  \;\bC\bRings_\Z \to \bTop$ & fine noncommutative spectrum (see Proposition \ref{prop:funct} \\ \\ [-1em]
$\SpecG :  \;\bC\bRings_\Z \to \bTop$ & Grothendieck's spectrum of prime ideals \\ \\ [-1em]
\hline
\end{longtable}
%\end{multicols}

\section{Preliminaries}
\label{sec:preliminaries}

We collect here standard notions from category theory in an abridged fashion, mostly to settle notation, and refer the reader to \cite{Mac, Wei} for further details.

If $\sC$ is a category, we use
the same symbol, $\sC$, to denote the class of objects of $\sC$, so that $X \in \sC$ means that $X$ is an
object of $\sC$. We denote by $\id_\sC: \sC \to \sC$ the {\it identity functor}.
For an object $X \in \sC$ we denote by $\sC/X$ the category of {\it objects over $X$} and by $X/\sC$ the category of {\it objects under $X$}. For two objects $X, Y \in \sC$ we denote by $\Hom_{\sC}(X, Y)$ the set of morphisms between them. These are also called \emph{hom-sets}. If the category to which the hom-sets are referring to is clear from the context, we drop the $\sC$ and simply write $\Hom(X, Y)$.

A {\it monoidal structure} on a category $\sC$ is the data of a functor $\ootimes: \sC \times \sC \to \sC$ and an object $1_\ootimes \in \sC$, equipped with natural isomorphisms $1_\ootimes \ootimes X \stackrel{\cong}{\to} X$, $X \ootimes 1_\ootimes \stackrel{\cong}{\to} X$ and $(X \ootimes Y) \ootimes Z \stackrel{\cong}{\to} X \ootimes (Y \ootimes Z)$, for objects $X, Y, Z \in \sC$. The natural isomorphisms are required to satisfy appropriate coherence relations that we omit for the sake of brevity. The object $1_\ootimes$ is called the {\it identity} (or {\it unit}) of the monoidal structure $\ootimes$. The data of a category with a choice of a monoidal structure $(\sC, \ootimes, 1_\ootimes)$ is called {\it monoidal category}. If we also require the datum of a natural isomorphism $X \ootimes Y \stackrel{\cong}{\to} Y \ootimes X$, satisfying some further coherence relations, then the monoidal structure is called {\it symmetric} and $(\sC, \ootimes, 1_\ootimes)$ is called {\it symmetric monoidal category}. Note that the symmetry condition amounts to prescribing extra data. 

In this paper we will work with several symmetric monoidal categories. The most basic example of a symmetric monoidal category relevant to our work is the category of abelian groups $\bAb$ equipped with the standard tensor product $\otimes_\Z$. Another basic example is the category of sets $\bSets$ equipped with the binary direct product $\prod$.

Given a monoidal category $(\sC, \ootimes, 1_\ootimes)$ we define the category of {\it algebras} (also called {\it monoids}) over it, denoted by $\bAlg(\sC)$, as follows. The objects of $\bAlg(\sC)$ are objects $X \in \sC$ equipped with a morphism $m: X \ootimes X \to X$, called {\it multiplication}, and a morphism $\eta: 1_\ootimes \to X$, called {\it identity}, satisfying the usual relations (written in a diagrammatic form). 
Morphisms in $\bAlg(\sC)$ are morphisms of the underlying objects of $\sC$ that commute with multiplication and identities. If $\ootimes$ is symmetric, then we can also consider the category of {\it commutative algebras} (or {\it commutative monoids}), denoted by $\bCAlg(\sC)$, namely the full subcategory of $\bAlg(\sC)$ of objects whose multiplication map commutes with the symmetric natural transformation of $\ootimes$. For example, if we consider $(\bAb, \otimes_\Z)$, then $\bAlg(\bAb)$ is the category of rings and $\bCAlg(\bAb)$ is the category of commutative rings, whereas if we consider $(\bSets, \prod)$ then $\bAlg(\bSets)$ is the category of monoids and $\bCAlg(\bSets)$ is the category of commutative monoids.

To any object $A \in \bAlg(\sC)$ we associate the corresponding category of left (resp.~right) modules ${}_A \bMod$ (resp.~$\bMod_A$) as follows. Objects of ${}_A \bMod$ are objects $M \in \sC$ equipped with an action $A \ootimes M \to M$ and morphisms are morphisms of $\sC$ commuting with the action. If $\ootimes$ is symmetric and $A$ is commutative (\ie, $A\in \bCAlg(\sC)$), then ${}_A \bMod \cong \bMod_A$ canonically and $\bMod_A$ has an induced symmetric monoidal structure $\ootimes_A$. For example, $\bMod_\Z \cong \bAb$.

By $\bRings$ we denote the category of rings with identity-preserving homomorphisms as morphisms, while the full subcategory of commutative rings will be denoted by $\bC\bRings$. Observe that $\bRings \cong \bAlg(\bMod_\Z)$ and $\bC\bRings \cong \bCAlg(\bMod_\Z)$.
For a ring $R \in \bRings$ we call the category $\bRings_R := R/\bRings$ the category of {\it $R$-algebras} and by $\bMod_R$ the category of (right) modules over a ring $R \in \bRings$. We denote by $\bCh(\bMod_R)$ the category of (unbounded) {\it chain complexes} over $\bMod_R$ and by $\bCh^{\le 0}(\bMod_R) \subset \bCh(\bMod_R)$ the category of {\it connective complexes}, \ie, complexes concentrated in negative degrees (we use the cohomological convention for differentials, \ie~differentials increase the degree).\\
Given a commutative ring $R$, we can endow the categories $\bCh(\bMod_R)$ and $\bCh^{\le 0}(\bMod_R)$ with a (symmetric) monoidal structure as follow:  let $X^\bullet, Y^\bullet \in \bCh(\bMod_R)$, then
\begin{equation}
\label{section 2 equation 1}
(X^\bullet \otimes_R Y^\bullet)^n = \bigoplus_{i+j = n} X^i \otimes_R Y^j,
\end{equation}
and the unit is the complex whose only non-zero element is $R$ in degree $0$. A straightforward computation shows that $X^\bullet, Y^\bullet \in \bCh^{\le 0}(\bMod_R)$ implies $X^\bullet \otimes_R Y^\bullet \in \bCh^{\le 0}(\bMod_R)$.  Algebras over $\bCh(\bMod_R)$ for this monoidal structure are called\footnote{Here ``dg" stands for ``Differential Graded".} {\it (unbounded) dg-algebras over $R$} and denoted $\bD\bG\bA_R$, and algebras over $\bCh^{\le 0}(\bMod_R)$ are called {\it connective dg-algebras over $R$} and denoted $\bD\bG\bA_R^{\le 0}$.
Concretely, a dg-algebra over $R$ is a complex of $R$-modules $A^\bullet$ together with a graded multiplication $A^n \otimes A^m \to A^{n +m}$, satisfying the Leibniz rule
\[ d(a b) = d(a) b + (-1)^n a d(b) \]
for $a \in A^n, b \in A^m$. A dg-algebra is commutative if its multiplication is graded commutative, \ie~if
\[ a b = (-1)^{n m} b a \]
for all $a \in A^n, b \in A^m$.

We denote by $\Delta$ the \emph{simplex category}, namely the category of finite sets where each set $[n]:= \{ 0, \ldots,  n \}$ is equipped with the canonical total ordering $0 < 1 < \ldots < n$ and where $\Hom_\Delta([n], [m])$ consists only of non-decreasing maps for this ordering. A functor $\Delta^\op \to \sC$ is called a \emph{simplicial object} in $\sC$ and a functor $\Delta \to \sC$ is called a \emph{cosimplicial object} in $\sC$.
The class of simplicial and cosimplicial objects in a category is organised into a category whose morphisms are natural transformations of functors. In general we will denote the category of simplicial objects in $\sC$ by $\bS \sC$. We also notice that the set $\Hom_\Delta([n], [m])$ is generated under composition by simple injective and surjective maps called \emph{coface} and \emph{codegeneracy} maps. Therefore, in order to specify a functor $\Delta^\op \to \sC$ it is enough to specify its values on these generating morphisms.

As per the above, we denote by $\bSMod_R$ the category of simplicial objects in $\bMod_R$, \ie~the category of functors $\Delta^\op \to \bMod_R$. If $R$ is commutative, the category $\bSMod_R$ is a symmetric monoidal category with monoidal structure given by
\[ (X^\bullet \otimes_R Y^\bullet)^n = X^n \otimes_R Y^n, \qquad  X^\bullet, Y^\bullet \in \bSMod_R \,.\]
Notice that $\bAlg(\bSMod_R)$ is equivalent to the category of simplicial objects in $\bAlg(\bMod_R)$ and $\bCAlg(\bSMod_R)$ is equivalent to simplical objects in $\bCAlg(\bMod_R)$. For this reason, we will also use the notation $\bS\bRings_R := \bAlg(\bSMod_R)$ and $\bS\bC \bRings_R := \bCAlg(\bSMod_R)$.
 
The Dold--Kan correspondence states that $\bCh^{\le 0}(\bMod_R)$ is equivalent to the category $\bSMod_R$. The equivalence is realised by an adjoint pair of functors
\[ N: \bS\bMod_R \leftrightarrows \bCh^{\le 0}(\bMod_R): \Gamma\,, \] 
where $N$ is called the {\it normalised complex functor} and $\Gamma$ the {\it nerve functor}. 
We will not be defining the functors $N$ and $\Gamma$ explicitly here, in that we will only be relying on some of their formal properties. Their most important property is that, remarkably, $\Gamma$ sends quasi-isomorphisms of complexes to weak equivalences of simplicial set and, dually, $N$ sends weak equivalences to quasi-isomorphisms. This property implies that one can compute the cohomology of $X \in \bCh^{\le 0}(\bMod_R)$ as the homotopy groups of $\Gamma(X)$, and, viceversa, one can compute the homotopy groups of $S \in \bS\bMod_R$ as the cohomology of the complex $N(S)$. We will further elaborate on this later on, when we will introduce model structures and homotopy categories.

Another important property of the Dold--Kan correspondence is that, when $R$ is commutative, the functors $N$ and $\Gamma$ are both lax monoidal and oplax monoidal functors, \ie~there exist natural transformations
\[ \Delta_{S, T}: N(S) \otimes_R N(T) \to N(S \otimes_R T), \qquad \ \nabla_{S, T}: N(S \otimes_R T) \to N(S) \otimes_R N (T) \]
for all $S, T \in \bS\bMod_R$ called the \emph{Eilenberg-Zibler map} and the \emph{Alexander-Whitney map}, respectively, which, by formal properties of adjoint functors, induce natural transformations
\[ \Gamma(X) \otimes_R \Gamma(Y) \to \Gamma(X \otimes_R Y), \qquad \ \Gamma(S \otimes_R T) \to \Gamma(S) \otimes_R \Gamma(T). \]
Neither of these maps is a natural isomorphism. Thus, whilst the lax monoidal property implies that $(N \vdash \Gamma)$ induces an adjunction
\[ N: \bS\bRings_R \leftrightarrows \bDGA_R^{\le 0}: \Gamma\,, \]
this adjunction is not an equivalence. However, it is possible to show that $\Gamma$ is fully faithful \cite[Proposition 2.13]{SS}. This means, loosely speaking, that the two categories of algebras are quite close to each other. The same cannot be said for commutative algebras as the lax structure on $\Gamma$ does not preserve the symmetry of the tensor product and, hence, does not send commutative algebras to commutative algebras. Therefore, there is no commutative analogue of the adjunction between simplicial algebras and dg-algebras, and the commutativity of an algebra is not a property that is preserved by these functors.

Let us briefly recall some notions from homotopy theory needed in this work. 
Rather than using the language of $\infty$-categories, which would be preferable but is less accessible, we will work here, for simplicity, in the framework of model categories. A model structure on a category $\sC$ is the datum of three distinguished classes $(W, F, C)$ of morphisms of $\sC$, called \emph{weak equivalences, fibrations and cofibrations}, respectively, satisfying appropriate conditions, for which we refer the reader to \cite[Definition 1.1.3]{Hov}. 
For the reader less accustomed to model categories, it may be helpful to think about weak equivalences as the analogue of (weak) `homotopy equivalences', fibrations as the analogue `well-behaved surjections' and cofibrations as analogue of `well-behaved inclusions'.
Of these three classes of maps the most important are weak equivalences $W$, in that $W$ is used to define the homotopy category of $\sC$. The latter is defined as $\bHo(\sC):=\sC[W^{-1}]$, \ie~ as the category obtained from $\sC$ by formally inverting the morphisms in the class $W$. The other two classes are additional data used for computing objects and morphisms in $\sC[W^{-1}]$. Let us emphasise that the notion of model structure contains redundant information. For example, we will use the fact that the classes $W$ and $F$ uniquely determine $C$ (and, \emph{vice versa}, that $W$ and $C$ determine $F$).

A {\it model category} is a complete and cocomplete category $\sC$ equipped with a model structure $(W,F,C)$. We recall that morphisms belonging to $W \cap F$ are called \emph{trivial fibrations} and to $W \cap C$ are called \emph{trivial cofibrations}. Also, an object $X \in \sC$ such that the morphism from the initial object $\void \to X$ is a cofibration is called \emph{cofibrant}, whereas an object from which the canonical morphism to the final object $X \to 0$ is a fibration is called \emph{fibrant}.

Given two model categories $\sC$ and $\sD$, a pair of adjoint functors 
$(F\dashv G): \sC \leftrightarrows \sD$
is said to be a \emph{Quillen adjunction} if $F$ preserves cofibrations and acyclic cofibrations (or, equivalently, if $G$ preserves fibrations and acyclic fibrations). Given a pair of Quillen adjoint functors $(F\dashv G)$, the \emph{left} (resp.~\emph{right}) \emph{derived functor} $\L F$ (resp.~$\R G$) is obtained by applying $F$ (resp.~$G$) to cofibrant (resp.~fibrant) objects of $\sC$ (resp.~$\sD$). These latter functors form an adjoint pair $\L F: \bHo(\sC) \leftrightarrows \bHo(\sD) : \R G$ fitting into a commutative (up to natural transformation) diagram of functors
\[
\begin{tikzpicture}
\matrix(m)[matrix of math nodes,
row sep=2.6em, column sep=2.8em,
text height=1.5ex, text depth=0.25ex]
{ \sC &  \sD  \\
\bHo(\sC)  & \bHo(\sD) \\};
\path[->]
($(m-1-1.east)+(0,0.1)$) edge node[auto] {$F$} ($(m-1-2.west)+(0,0.1)$);
\path[->]
($(m-1-2.west)+(0,-0.1)$) edge node[auto] {$G$} ($(m-1-1.east)-(0,0.1)$);
\path[->]
(m-1-2) edge node[auto] {} (m-2-2);
\path[->]
(m-1-1) edge node[auto] {} (m-2-1);
\path[->]
($(m-2-1.east)+(0,0.1)$) edge node[auto] {$\L F$} ($(m-2-2.west)+(0,0.1)$);
\path[<-]
($(m-2-1.east)+(0,-0.1)$) edge node[below] {$\R G$} ($(m-2-2.west)+(0,-0.1)$);
\end{tikzpicture}.
\]
A Quillen adjoint pair of functors $(F \dashv G)$ between two model categories such that $(\L F \dashv \R G)$ is an equivalence adjunction is called \emph{Quillen equivalence}. We refer the reader to \cite[Section 1.3]{Hov} for more about Quillen functors.

 We will use the following standard model structures and we refer to~\cite[Chapter 2]{Hov} for further details. 
\begin{itemize}
\item
 On $\bCh^{\le 0}(\bMod_R)$ we consider the \emph{projective model structure}, for which weak equivalences are quasi-isomorphisms of complexes, fibrations are morphisms of complexes that are surjective in negative degrees and cofibrations are morphisms of complexes with projective cokernel that are injective in all degrees. 
\item 
 On $\bS\bMod_R$ we consider the \emph{standard model structure}, for which weak equivalences are weak homotopy equivalences, fibrations are the Kan fibrations of the underlying simplicial sets and cofibrations are uniquely determined.
\end{itemize} 

The next proposition, whose proof can be found in \cite[Section 4.1]{SS}, relates the two model structures above.

\begin{prop} \label{prop:derived_Dold_Kan}
The Dold--Kan correspondence sends the weak equivalences for the projective model structure on $\bCh^{\le 0}(\bMod_R)$ to weak equivalences for the standard model structure on $\bS\bMod_R$ and viceversa. More precisely, it induces an equivalence 
\[ \L N: \bHo(\bS\bMod_R) \leftrightarrows \bHo(\bCh^{\le 0}(\bMod_R)): \R \Gamma \]
between homotopy categories, and $(N\vdash\Gamma)$ is a Quillen equivalence.
\end{prop}

Let us also observe that by definition 
\[ \bHo(\bCh^{\le 0}(\bMod_R)) \cong D^{\le 0}(\bMod_R) \]
where $D^{\le 0}(\bMod_R)$ is the standard derived category of connective complexes of $\bMod_R$.

Now, the two model structures introduced so far -- the standard model structure on $\bS\bMod_R$ and the projective model structure on $\bCh^{\le 0}(\bMod_R)$   -- possess numerous special properties, in addition to those shared by all model structures. 
As far as our paper is concerned, a key property thereof is the remarkable fact that
when $R$ is commutative both model structures satisfy the \emph{monoid axiom} with respect to the monoidal structures $\otimes_R$ presented earlier in this section. This implies that the categories $\bD\bG\bA_R^{\le 0}$ and $\bSRings_R$ inherit model structures form $\bCh^{\le 0}(\bMod_R)$ and $\bS\bMod_R$, respectively, by pulling back the model structure via the forgetful functor. Thus, a morphism in $\bD\bG\bA_R^{\le 0}$ (resp.~$\bSRings_R$) is a weak equivalence (resp.~a fibration) if and only if the morphism of underlying objects of $\bCh^{\le 0}(\bMod_R)$ (resp.~$\bMod_R$) is a weak equivalence (resp.~a fibration). This also uniquely determines cofibrations.

In light of the above discussion, the most interesting consequence of the Dold--Kan correspondence for this work is the following, see \cite[Theorem 1.1 (3)]{SS}.

\begin{prop} \label{prop:monoidal_Dold_Kan}
Let $R$ be a commutative ring. Then the monoidal Dold--Kan correspondence 
\[ N: \bSRings_R \leftrightarrows \bD\bG\bA_R^{\le 0}: \Gamma \]
induces Quillen equivalences
\begin{equation}
\label{prop:monoidal_Dold_Kan equation 1}
\L N: \bHo(\bSRings_R) \leftrightarrows \bHo(\bD\bG\bA_R^{\le 0}): (\L N)^{-1}
\end{equation}
and 
\begin{equation}
\label{prop:monoidal_Dold_Kan equation 2}
(\R \Gamma)^{-1}: \bHo(\bSRings_R) \leftrightarrows \bHo(\bD\bG\bA_R^{\le 0}): \R \Gamma\,. 
\end{equation}
\end{prop}

Let us emphasise that the derived functors $\L N$ and $\R \Gamma$ appearing in Proposition \ref{prop:monoidal_Dold_Kan} are both Quillen equivalences, but they are not inverses of each other, so they do not form a Quillen adjunction. Nevertheless, they both have their respective adjoints, described in \cite{SS}. A key observation to understand \eqref{prop:monoidal_Dold_Kan equation 1} and \eqref{prop:monoidal_Dold_Kan equation 2} is that, although the Eilenberg-Zibler maps and the Alexander-Whitney maps are not invertible (and in particular they are not inverse of each other), they are homotopy equivalences. Therefore, homotopically $N$ and $\Gamma$ are strongly monoidal functors. 

We denote the category $\bHo(\bD\bG\bA_R^{\le 0})$ by $\bH\bRings_R$ and refer to it as the category of \emph{homotopy (connective) $R$-algebras}. Henceforth we will make no distinction between $\bHo(\bSRings_R)$ and $\bHo(\bD\bG\bA_R^{\le 0})$, unless otherwise stated, implicitly using the above equivalences. If $A \in \bHo(\bD\bG\bA_R^{\le 0})$ we define $\bH\bRings_A$ as the category of objects under $A$.

Perhaps surprisingly at first, the picture is more complicated for commutative differential graded algebras, which explains the need to use simplicial objects in this paper. Chain complexes are easier to deal with, are preferable for doing computations, and some of the key computations in literature upon which this work relies have been carried out using chain complexes. Hence, working with chain complexes would be a natural choice for us. However, we also would like to compare our results about noncommutative geometry with the usual derived algebraic geometry --- that is, the commutative version of our theory. 
Whilst in this case, too, chain complexes may be preferable, their use is not straightforward, because the model structure on chain complexes does not always satisfy the \emph{symmetric monoid axiom} (see~\cite{Whi} for further details). The issue here is related to the fact that the functors $N$ and $\Gamma$ given by the Dold--Kan correspondence are not symmetric monoidal, so that, although the two underlying categories are equivalent, their respective monoidal structures have different properties. In particular, for $R$ commutative, the monoidal structure on $\bS \bMod_R$ always satisfies the symmetric monoid axiom, whereas the monoidal structure on $\bCh^{\le 0}(\bMod_R)$ satisfies the symmetric monoid axiom only if the characteristic of $R$ is $0$. This leads to two possible solutions. The first solution consists in using the functor $N$ to transfer the well-behaved monoidal structure from $\bS \bMod_R$ to $\bCh^{\le 0}(\bMod_R)$. The resulting monoidal structure is often called \emph{shuffle product}. Alternatively, one can work with the category $\bS \bMod_R$ and use (with all the care required) the Dold--Kan correspondence to move back and forth between chain complexes and simplicial objects whenever one of the two settings is more convenient. We pursue the second approach. Therefore, if $R$ is a commutative ring, we define the category of \emph{homotopy (connective) commutative $R$-algebras} as
\[ \bH\bC\bRings_R := \bHo(\bS\bC\bRings_R), \]
where we used the fact that the monoidal structure on $\bS\bMod_R$ satisfies the symmetric monoid axiom and therefore it induces naturally a model structure on $\bCAlg(\bS\bMod_R) \cong \bS\bC\bRings_R$ by the right transfer of model structures. In other words, weak-equivalences (resp.~fibrations) of objects of $\bS\bC\bRings_R$ are weak-equivalences (resp.~fibrations) of the underlying objects of $\bS\bMod_R$.

\section{The homotopy category of algebras}
\label{sec:homotopy_algebras}

The main goal of this section is to study some of the basic properties of the category $\bH\bRings_R$ introduced in the previous section --- properties needed for developing and understanding the derived geometry associated with it. 

\subsection{Categories of algebras} \label{Categories of algebras}

Throughout this subsection $R\in \bC\bRings$ is a fixed commutative $\Z$-algebra.

We first recall how the various categories of algebras we have introduced so far are related. There is a diagram of faithful inclusions of categories
\[
\begin{tikzpicture}
\matrix(m)[matrix of math nodes,
row sep=2.6em, column sep=2.8em,
text height=1.5ex, text depth=0.25ex]
{ \bC\bRings_R &  \bH\bC\bRings_R  \\
\bRings_R  & \bH\bRings_R \\};
\path[->]
(m-1-1) edge node[auto] {} (m-1-2);
\path[->]
(m-1-2) edge node[auto] {} (m-2-2);
\path[->]
(m-1-1) edge node[auto] {} (m-2-1);
\path[->]
(m-2-1) edge node[auto] {} (m-2-2);
\end{tikzpicture}.
\]
The vertical functors are just the inclusions of commutative rings into the noncommutative ones, whereas the horizontal functors are the inclusion of rings in the homotopy category as discrete objects. We notice that all the functors are fully faithful but the inclusion $\bH \bC \bRings_R \to \bH \bRings_R$ that is only faithful (later on we will describe the hom-sets of the homotopy categories and it will be clear why this is the case). We also notice that the fact that $\bH \bC \bRings_R$ is a balanced category implies that the inclusion into $\bH \bRings_R$ is a conservative functor.

The functors above can be described explicitly as follows. An object $A \in \bRings_R$ can be mapped functorially to the simplicial ring
\[ A^\disc = [ \ldots \mathrel{\substack{\textstyle\rightarrow\\[-0.6ex]
                      \textstyle\rightarrow \\[-0.6ex]
                      \textstyle\rightarrow\\[-0.6ex]
                      \textstyle\rightarrow}}
                       A \mathrel{\substack{\textstyle\rightarrow\\[-0.6ex]
                      \textstyle\rightarrow \\[-0.6ex]
                      \textstyle\rightarrow}} A \rightrightarrows A ] \]
where all the face and degeneracy maps are the identity map. Then, 
\[ \pi^n(A^\disc) = \begin{cases}
A & \text{ for } n = 0 \\
0 & \text{ for } n \ne 0
\end{cases}
 \]
where $\pi^n(A^\disc)$ denotes the $n$-th homotopy group of the simplicial set underlying $A^\disc$. Note that for any $B \in \bS \bRings_R$ the set $\pi^0(B)$ has an induced structure of $R$-algebra.
 
The functor $(\minus)^\disc$ realises the horizontal full embeddings of the diagram above. 
We observe that all the embedding functors above are the right adjoints of an adjoint pair of functors. Therefore, the complete diagram of functors reads
\[
\begin{tikzpicture}
\matrix(m)[matrix of math nodes,
row sep=2.6em, column sep=2.8em,
text height=1.5ex, text depth=0.25ex]
{ \bC\bRings_R &  \bH\bC\bRings_R  \\
\bRings_R  & \bH\bRings_R \\};
\path[right hook->]
(m-1-1) edge node[auto] {$(\minus)^\disc$} (m-1-2);
\path[->, transform canvas={yshift=-2.5mm}, font=\scriptsize]
(m-1-2) edge node[auto] {$\pi_0$} (m-1-1);
\path[->]
(m-1-2) edge node[auto] {} (m-2-2);
\path[->, transform canvas={xshift=-2.5mm}, font=\scriptsize]
(m-2-2) edge node[auto] {$(\minus)^{\L \ab}$} (m-1-2);
\path[right hook->]
(m-1-1) edge node[auto] {} (m-2-1);
\path[->, transform canvas={xshift=-2.5mm}, font=\scriptsize]
(m-2-1) edge node[auto] {$(\minus)^\ab$} (m-1-1);
\path[right hook->]
(m-2-1) edge node[auto] {$(\minus)^\disc$} (m-2-2);
\path[->, transform canvas={yshift=-2.5mm}, font=\scriptsize]
(m-2-2) edge node[auto] {$\pi_0$} (m-2-1);
\end{tikzpicture}.
\]
The functor $\pi_0$ just maps a simplicial ring to its $0$-th homotopy ring, $(\minus)^\ab$ is the \emph{abelianization} functor defined by
\[ A \mapsto \frac{A}{(xy - yx\ |\ x,y \in A)} \]
and $(\minus)^{\L \ab}$ is the derived functor of the abelianization functor. We observe that this diagram of functors is commutative up to natural isomorphism. Also, the fact that $(\minus)^\disc$ is fully faithful can be checked easily by noticing that the counit map of the adjunction $\pi^0(A^\disc) \to A$ is a natural isomorphism. 

The categories $\bC\bRings_R$ and $\bRings_R$ have all limits and colimits. We briefly describe them. In both $\bC\bRings_R$ and $\bRings_R$ products are given by the Cartesian product of the underlying set with coordinatewise operations. Similarly, pullbacks are given by the pullbacks of the underlying sets equipped with the natural ring operations induced by the inclusion into the Cartesian product. On the other hand, coproducts and pushouts in $\bC\bRings_R$ are very different from those in $\bRings_R$. We first describe the former case. The pushout of $A \to B$ and $A \to C$ in $\bC\bRings_R$ is given by the ring
\[ B \otimes_A C \]
and the coproduct of any small family $\{ A_i \}_{i \in I} \subset \bC\bRings_R$ is given by
\[ \bigotimes_{i \in I} A_i = \limind_{(i_1, \ldots, i_n) \in I^n} A_{i_1} \otimes_R \cdots \otimes_R A_{i_n}  \] 
where the colimit is taken over the family of all finite subsets of $I$. The pushout of two objects in $\bRings_R$ is given by the \emph{free product} of algebras. For two morphisms $A \to B$ and $A \to C$ in $\bRings_R$ this is defined by
\begin{equation}
\label{free product}
B \ast_A C = \frac{T_{\otimes_A}(B \oplus C)}{({ b \otimes_A b' - b b', c \otimes_A c'- cc', 1_B - 1_C \ |\  1_B, b,b' \in B, 1_C, c, c' \in C})}
\end{equation}
where $T_{\otimes_A}(\minus)$ denotes the tensor algebra construction over $A$. The coproduct of any small family $\{ A_i \}_{i \in I} \subset \bRings_R$ is given by
\[ \bigast_{i \in I} A_i = \limind_{(i_1, \ldots, i_n) \in I^n} A_{i_1} \ast_R \cdots \ast_R A_{i_n} \] 
where again the colimit is indexed by all finite subsets of $I$.
We thus notice that the inclusion $\bC\bRings_R \to \bRings_R$ does not commute with colimits. But the fact that $(\minus)^\ab$ is left adjoint immediately implies that
\[ \left ( \bigast_{i \in I} A_i \right )^\ab = \bigotimes_{i \in I} A_i^\ab, \ \  (B \ast_A C)^\ab = B^\ab \otimes_{A^\ab} C^\ab.  \]

We recall that a colimit $\limind_{i \in I} X_i$ in a category $\bC$ is said to be \emph{filtered} if the indexing set $I$ is a directed poset.
Although the functor $\bC\bRings_R \to \bRings_R$ does not in general commute with colimits, it has the following (well-known) property.

\begin{prop} \label{prop:inclusion_commutes_fil_colim}
The inclusion $\bC\bRings_R \to \bRings_R$ commutes with filtered colimits.
\end{prop}
\begin{proof}
Consider $\limind_{i \in I} A_i \in \bC\bRings_R$ and let us denote $\iota: \bC\bRings_R \to \bRings_R$ the inclusion functor. Then, we need to show that the canonical map
\[ \limind_{i \in I} \iota(A_i) \to \iota(\limind_{i \in I} A_i) \]
is an isomorphism. Notice that since $\iota$ is the right adjoint to the abelianization functor we have that
\[ \left ( \limind_{i \in I} \iota(A_i) \right )^\ab \cong  \limind_{i \in I} \iota(A_i)^\ab \cong \limind_{i \in I} A_i \]
because $\iota$ is fully faithful. It is thus enough to check that $\limind_{i \in I} \iota(A_i)$ is commutative. Let us consider $x, y \in \limind_{i \in I} \iota(A_i)$. Since the colimit is filtered and the forgetful functor $\bRings_R \to \bSets$ commutes with filtered colimits, there exists two elements $\tilde{x} \in \iota(A_{i_x})$ and $\tilde{y} \in \iota(A_{i_y})$ such that $\a_{i_x}(\tilde{x}) = x$ and $\a_{i_y}(\tilde{y}) = y$, where $\a_x: \iota(A_{i_x}) \to \limind_{i \in I} \iota(A_i)$ and $\a_y: \iota(A_{i_y}) \to \limind_{i \in I} \iota(A_i)$ are the canonical maps. The fact that $I$ is directed implies that there exists $k \in I$ such that $i_x \le k$ and $i_y \le k$. Thus, since $A_k$ is commutative we have
\begin{multline}
 x \cdot y = \a_k(\phi_{i_x}(\tilde{x})) \cdot \a_k(\phi_{i_y}(\tilde{y})) = \a_k(\phi_{i_x}(\tilde{x}) \cdot \phi_{i_y}(\tilde{y})) 
\\
=  \a_k(\phi_{i_y}(\tilde{y}) \cdot \phi_{i_x}(\tilde{x})) = \a_k(\phi_{i_y}(\tilde{y})) \cdot \a_k(\phi_{i_x}(\tilde{x})) = y \cdot x \,,
\end{multline}
where $\phi_{i_x}: A_{i_x} \to A_k$ and $\phi_{i_y}: A_{i_y} \to A_k$ are morphisms of the filtered system. Therefore $\limind_{i \in I} \iota(A_i)$ is commutative, and this completes the proof.
\end{proof}

\subsubsection{Finitely presented algebras}
\label{On the notion of finite presentation}

Before moving on to the homotopical setting, let us recall some fundamental differences in the notion of finite presentation between the commutative and noncommutative settings. For the sake of clarity, we shall discuss objects and morphisms separately. Let us start with the former. 
Let $R$ be a commutative ring.

\begin{defn}[Objects of finite presentation]
\label{defn obj fin pres}
A ring $A \in \bRings_R$ is said to be of \emph{finite presentation} (or \emph{finitely presented}) if the canonical morphism
\[ \limind_{i \in I} \Hom_{\bRings_R}(A, B_i) \to \Hom_{\bRings_R}(A, \limind_{i \in I} B_i) \]
is an isomorphism, for all filtered colimits $\limind_{i \in I} B_i$ of  $R$-algebras. An analogous definition is given for objects in $\bC\bRings_R$.
\end{defn}

We would like to recall explicitly, as this will play a key role in the forthcoming discussion, that the generators of the polynomial algebra with coefficients in $R$ can be obtained as the image, via the tensor product functor $(\minus) \otimes_\Z R$, of the generators of polynomial algebra with coefficients in $\Z$. So, we have the canonical isomorphism
\begin{equation}
\label{commmutative ring}
R [x_1, \ldots, x_n]\cong \Z[x_1, \ldots, x_n] \otimes_\Z R\,.
\end{equation}

For $R \in \bC \bRings$, let us denote by $R\langle x_1, \ldots, x_n\rangle$ the \emph{free algebra} with coefficients in $R$, with the property that $R$ is in the centre of the algebra. Concretely,
\[ R\langle x_1, \ldots, x_n\rangle :=  \left \{  \sum a_w w \, |\, w \text{ is a word in the variables } x_1, \ldots, x_n  \right \} \]

In the noncommutative case the computation of the free product functor $(\minus) \ast_\Z R$ on the above free algebras with coefficients in $\Z$ gives rise to the family of algebras
\begin{equation}
\label{noncommutative ring}
 R \{ x_1, \ldots, x_n \}:=\Z\langle x_1, \ldots, x_n\rangle \ast_\Z R.
\end{equation}

One might expect that $R\langle x_1, \ldots, x_n\rangle \cong R \{ x_1, \ldots, x_n \}$. However, for a general $R$, this is \emph{not} the case, even when $R$ is commutative. This happens because the generators of~\eqref{noncommutative ring} do not commute with all the elements of $R$.

\begin{exa} \label{exa:non-central_free_algebras}
    Consider the case $R = \Z[x]$. Then, 
    \[ R \{y\} = \Z[y] \ast_\Z R = \Z[y] \ast_\Z \Z[x] \cong \Z \lt x, y \gt.  \]
    In particular, the image of $R \to R \{y\}$ does not lie in the center, whereas 
    \[ R \lt y \gt = R[y] \cong \Z[x,y] \ne \Z \lt x, y \gt \cong R \{y\}. \]
\end{exa}

\begin{prop} \label{prop:finitely_presented_objects}
We have the following.
\begin{enumerate}[(i)]
\item A ring $A \in \bC\bRings_\Z$ is of finite presentation if and only if it is a quotient of $\Z[x_1, \ldots, x_n]$.
\item A ring $A \in \bRings_\Z$ is of finite presentation if and only if it is a quotient of $\Z\{ x_1, \ldots, x_n \}$ by a finitely generated ideal.
\end{enumerate}
\end{prop}

\begin{proof}
\begin{enumerate}[(i)]
\item This result is well-known. However, we provide a full proof for reader's convenience. 

Since $\Z[x_1, \ldots, x_n]$ is Noetherian, all its ideals are finitely generated. Suppose
\[ A \cong \frac{\Z[x_1, \ldots, x_n]}{I} \cong \Coeq(\Z[y_1, \ldots, y_m] \stackrel{\phi_1, \phi_2}{\rightrightarrows} \Z[x_1, \ldots, x_n]) \] 
where $\phi_1, \phi_2$ are ring homomorphisms such that $\im(\phi_1 - \phi_2)$ generates \footnote{For example, if $I = (f_1, \ldots, f_m)$ then we can consider $\phi_1(y_i) = 1$ for all $i$ and $\phi_2(y_i) = 1 - f_i$. So, $(\phi_1 - \phi_2)(y) \in I$ for all $y$ and the generators of $I$ are in the image.} the ideal $I$. Then,
\begin{multline*} \Hom_{\bC \bRings_\Z}(A, \limind_{i \in I} B_i)  \cong \Hom_{\bC \bRings_\Z}( \Coeq(\Z[y_1, \ldots, y_m] \stackrel{\phi_1, \phi_2}{\rightrightarrows} \Z[x_1, \ldots, x_n]), \limind_{i \in I} B_i)   
\\  \cong \eq\left( \Hom_{\bC \bRings_\Z}(\Z[y_1, \ldots, y_m], \limind_{i \in I} B_i) \stackrel{\phi_{1, *}, \phi_{2,*}}{\rightrightarrows}  \Hom_{\bC \bRings_\Z}( \Z[x_1, \ldots, x_n], \limind_{i \in I} B_i)   \right)
\\  \cong\eq\left( \Hom_{\bSets}(\{ y_1, \ldots, y_m \}, \limind_{i \in I} B_i) \stackrel{\phi_{1, *}, \phi_{2,*}}{\rightrightarrows}  \Hom_{\bSets}( \{ x_1, \ldots, x_n \}, \limind_{i \in I} B_i) \right)
\end{multline*}
where the last isomorphism comes from the adjunction between the category of sets and the category of commutative rings. But then we have 
\begin{equation}
\label{prop3.4 what fails}
\begin{aligned}
 \eq\left( \Hom_{\bSets}(\{ y_1, \ldots, y_m \}, \limind_{i \in I} B_i) \stackrel{\phi_{1, *}, \phi_{2,*}}{\rightrightarrows}  \Hom_{\bSets}( \{ x_1, \ldots, x_n \}, \limind_{i \in I} B_i) \right)
 \\
  \cong
 \limind_{i \in I} \eq\left( \Hom_{\bSets}(\{ y_1, \ldots, y_m \}, B_i) \stackrel{\phi_{1, *}, \phi_{2,*}}{\rightrightarrows}  \Hom_{\bSets}( \{ x_1, \ldots, x_n \}, B_i) \right) \,.
 \end{aligned}
\end{equation}
Here we are using the fact that finite sets are finitely presented objects in the category of sets and that filtered colimits of sets commute with equalizers. It then follows that
\begin{multline*}\limind_{i \in I} \eq\left( \Hom_{\bSets}(\{ y_1, \ldots, y_m \}, B_i) \stackrel{\phi_{1, *}, \phi_{2,*}}{\rightrightarrows}  \Hom_{\bSets}( \{ x_1, \ldots, x_n \}, B_i)   \right)
\\ \cong \limind_{i \in I} \eq\left( \Hom_{\bC \bRings_\Z}(\Z[ y_1, \ldots, y_m], B_i) \stackrel{\phi_{1, *}, \phi_{2,*}}{\rightrightarrows}  \Hom_{\bRings_\Z}( \Z[ x_1, \ldots, x_n], B_i)  \right)
\\  \cong \limind_{i \in I}  \Hom_{\bC \bRings_\Z}\Bigl(\Coeq( \Z[ y_1, \ldots, y_m] \stackrel{\phi_1, \phi_2}{\rightrightarrows} \Z[ x_1, \ldots, x_n]), B_i\Bigr) 
\\ 
\cong \limind_{i \in I}  \Hom_{\bC \bRings_\Z}(A, B_i)
\end{multline*}
where we just performed similar steps to those in the first chain of isomorphisms above, in reversed order.
On the other hand, it is clear that if $A$ does not have a presentation as a quotient of $\Z[ x_1, \ldots, x_n]$ then it must have infinitely many generators and the isomorphism~\eqref{prop3.4 what fails} does not hold, because $\Hom$ out of an infinite set would have to commute with filtered colimits, which it does not. Hence $A$ is not of finite presentation.

\item The proof for the noncommutative case is analogous to that for the commutative case, modulo the following changes:
\begin{itemize}
\item replace $\Z [x_1, \ldots x_n ]$ with $\Z \lt x_1, \ldots x_n \gt$; 
\item observe that the algebras $\Z \lt x_1, \ldots, x_n \gt$ are not Noetherian, hence require that the ideal $I$ be \emph{finitely generated}.
\end{itemize}
\end{enumerate}
This concludes the proof.
\end{proof}

As a straightforward consequence of Proposition~\ref{prop:finitely_presented_objects} we have the following.

\begin{cor} \label{cor:finitely_presented_objects}
Let $R$ be a commutative $\Z$-algebra. We have the following.
\begin{enumerate}[(i)]
\item The ring $A \in \bC\bRings_R$ is of finite presentation if and only if it is a quotient of $R[x_1, \ldots, x_n]$ by a finitely generated ideal.
\item The ring $A \in \bRings_R$ is of finite presentation if and only if it is a quotient of $R \{ x_1, \ldots, x_n \}$ by a finitely generated ideal.
\end{enumerate}
\end{cor}
\begin{proof}
\begin{enumerate}[(i)]
\item The claim for a general $R$ is obtained from the case $R=\Z$ --- i.e., Proposition~\ref{prop:finitely_presented_objects}(i) --- via the tensor product functor (base change) and formula~\eqref{commmutative ring}.
\item The claim for a general $R$ is obtained from the case $R=\Z$ --- i.e., Proposition~\ref{prop:finitely_presented_objects}(ii) --- via the free product functor and formula~\eqref{noncommutative ring}.
\end{enumerate}
\end{proof}

Let us now discuss the corresponding notion for morphisms.

\begin{defn}[Morphisms of finite presentation]
\label{defn mor fin pres}
A morphism $A \to B \in \bRings_R$ is said to be of \emph{finite presentation} (or \emph{finitely presented}) if the canonical morphism
\[ \limind_{i \in I} \Hom_{\bRings_A}(B, C_i) \to \Hom_{\bRings_A}(B, \limind_{i \in I} C_i) \]
is an isomorphism, for all filtered colimits $\limind_{i \in I} C_i$ of $A$-algebras. 
An analogous definition is given for morphisms in $\bC\bRings_R$.
\end{defn}

\begin{remark}
\label{remark confuse mor and obj}
Note that an object $A \in \bRings_R$ is of finite presentation if and only if the structure map $R \to A$ is a morphism of finite presentation, and a morphism $A \to B$ is of finite presentation if and only if $B$ is an object of finite presentation in $\bRings_A$. The same is true for commutative rings. Therefore, when it comes to finite presentation, the use of the terminology for morphisms and objects is compatible. Henceforth, we will often tacitly identify the two notions.
\end{remark}

The notions of finite presentation in the commutative and noncommutative cases are formally the same. Therefore, it is natural to ask whether the two are compatible; namely, whether a morphism in $\bC\bRings_R$ is of finite presentation if and only if it is of finite presentation as a morphism in $\bRings_R$. Perhaps surprisingly, this matter turns out to be very subtle, and the answer is in the negative, as the following example demonstrates.

\begin{exa} \label{exa:non_finitely_generated_free_algebras}
Let $R$ be the commutative ring of integral polynomials in countably many variables, $R:=\Z[x_1,x_2, \dots]$, and consider the polynomial ring $S:=R[y]$ in one variable with coefficients in $R$. Clearly, $S$ is commutative and the canonical inclusion morphism $R \to S$ is of finite presentation in $\bC \bRings_R$, since $S$ is a free finitely generated commutative $R$-algebra.

As an object in $\bRings_R$, $S$ can be realised as the quotient of $ R \{ y \} =  \Z\langle y\rangle \ast_\Z  R $
by the ideal 
\[
I:=\{ ( x_j y - y x_j  ) \ |\ j\in \N \}.
\]
Namely,
\[ S \cong \frac{R \{ y \}}{I}. \]
Let us emphasise that the variable $y$ of $R \{ y \}$ does not commute with the variables $x_j$, $j\in \N$, of $R$.

We claim that the morphism $f:R \to S$ is \emph{not} of finite presentation in $\bRings_R$. Indeed, for $n\in \N$ define $I_n:=\{ ( x_j y - y x_j ) \in I \ |\ j\le n \}$. 
Then the canonical map
\[ \underset{n \in \N}{\limind} \Hom_{\bRings_R} \left (S, \frac{R\{ y \}}{I_n} \right) \to  \Hom_{\bRings_R} \left(S, \underset{n \in \N}\limind \frac{ R\{ y \}}{I_n} \right ) \cong  \Hom_{\bRings_R} \left(S, S \right ) \]
fails to be an isomorphism because the identity $S \to S$ does not factor through any of the quotients on the left-hand-side. In particular, this shows that $S$ is not finitely presented as an object in $\bRings_R$, recall Remark~\ref{remark confuse mor and obj}.
\end{exa}

In plain English, the above example tells us that the incompatibility between the commutative and noncommutative notions of finite presentation boils down to the fact that the condition on the commutation of the $\Hom$ functor with filtered colimits is required on two different classes of diagrams. 

Now, in Example~\ref{exa:non_finitely_generated_free_algebras} the (commutative) ring $R$ does not map into the center of $R\{y\}$. This leads to a violation of the finite presentation property --- as soon as $R$ is not finitely generated --- because in order to write $R[y]$ as a quotient of $R \{y \}$ one must add infinitely many relations to ensure that in the resulting quotient $y$ commutes with all elements. 
There are (at least) two possible ways to reconcile the notions of finite presentation in the commutative and noncommutative settings. One way is to restrict oneself to the case where both $R$ and $R\{y\}$ (or more general version thereof) are finitely generated over $\Z$; for example, this holds for discrete algebras (also known as $\Z$-algebras). Another way is to consider the category of $R$-algebras for which $R$ maps into their centre. The latter opens up an intermediate situation which we will now examine.

\

Motivated by the above discussion, given a commutative ring $R$ we define $\bZ\bRings_R$ be the subcategory of $\bRings_R$ given by
\begin{equation}
\label{def ZRings}
\bZ\bRings_R:=\{S\in \bRings_R \ | \ R \to S \ \text{has image in the centre $Z(S)$ of S}\}\,.
\end{equation}
By definition the inclusion $\bC\bRings_R\subset \bZ\bRings_R$ as categories is full. As it turns out, despite $\bZ\bRings_R$ contains noncommutative rings, we have following result.

\begin{prop} \label{prop:finitely_presented_objects_central_algebras}
The following statements hold.
\begin{enumerate}[(i)]
\item $A \in \bZ\bRings_R$ is of finite presentation if and only if it is a quotient of $R \lt x_1, \ldots, x_n \gt$ by a finitely generated ideal.
\item $A \in \bC \bRings_R$ is of finite presentation if and only if it is of finite presentation in $\bZ \bRings_R$.
\item In general, objects of finite presentation of $\bZ \bRings_R$ may not be of finite presentation as objects of $\bRings_R$.
\end{enumerate}
\end{prop}
\begin{proof}
\begin{enumerate}[(i)]
\item The claim follows by arguing as in Proposition~\ref{prop:finitely_presented_objects} and using the fact that all objects of $\bZ\bRings_R$ can be written as suitable quotients of inductive limits of the algebras $R \langle x_1, \ldots, x_n \rangle$. We will revisit this in more detail in the next subsection.

\item
Suppose that $A \in \bC\bRings_R$ is of finite presentation as an object of $\bC\bRings_R$, then 
\[ A \cong \frac{R[x_1, \ldots, x_n]}{I} \]
for some $n$ and some (finitely generated) ideal $I$. 

Let $I_n$ be the ideal generated by the elements
\begin{equation}
\{x_jx_k-x_kx_j \ | \ j,k\in\{1,\dots, n\}, \ j<k\}
\end{equation}
in $R \lt x_1, \ldots, x_n \gt$.
Since 
\[ R[x_1, \ldots, x_n] \cong \frac{R \lt x_1, \ldots, x_n \gt}{I_n}\,, \]
then $A$ can be written as the quotient 
\[
A\cong\frac{R \lt x_1, \ldots, x_n \gt}{\widetilde I}
\]
where the ideal $\widetilde I$ is the preimage of $I$ through the canonical map $R \lt x_1, \ldots, x_n \gt \to R [x_1, \ldots, x_n ]$. Hence, $A$ is finitely presented in $\bZ\bRings_R$ because $\widetilde I$ is clearly finitely generated. 

Now suppose that $A \in \bC\bRings_R$ is finitely presented as an object of $\bZ\bRings_R$. This means that one can write
\[ A \cong \frac{R \lt x_1, \ldots, x_n \gt}{I}\,, \]
where $I$ is some finitely generated ideal. Then one has
\begin{equation}
\label{eq:proofZring equation}
 A \cong A^\ab \cong \left (  \frac{R \lt x_1, \ldots, x_n \gt}{I} \right )^\ab \cong \frac{R \lt x_1, \ldots, x_n \gt^\ab}{I^\ab} = \frac{R[ x_1, \ldots, x_n ]}{I^\ab}\,.
\end{equation}
Here to obtain the third isomorphism we used the fact that the functor $(\minus)^\ab$ is left adjoint to $\bC\bRings_R \hookrightarrow \bZ\bRings_R$:
\begin{multline*}
\Hom_{\bC \bRings_R}(A^\ab, S) 
\cong 
\Hom_{\bZ\bRings_R}(A, S) 
\\
 \cong
\Hom_{\bZ\bRings_R}( \Coeq(R\langle y_1, \ldots, y_m\rangle \stackrel{\phi_1, \phi_2}{\rightrightarrows} R\langle x_1, \ldots, x_n\rangle), S)
\\
 \cong
\eq\left(\Hom_{\bZ\bRings_R}(R\langle y_1, \ldots, y_m\rangle,S) \stackrel{\phi_1, \phi_2}{\rightrightarrows} \Hom_{\bZ\bRings_R}(R\langle x_1, \ldots, x_n\rangle, S)\right)
\\
 \cong
\eq\left(\Hom_{\bC\bRings_R}(R\langle y_1, \ldots, y_m\rangle^\ab,S) \stackrel{\phi_1, \phi_2}{\rightrightarrows} \Hom_{\bC\bRings_R}(R\langle x_1, \ldots, x_n\rangle^\ab, S)\right) 
\\
 \cong
\Hom_{\bC\bRings_R}( \Coeq(R\langle y_1, \ldots, y_m\rangle^\ab \stackrel{\phi_1, \phi_2}{\rightrightarrows} R\langle x_1, \ldots, x_n\rangle^\ab), S) 
\\
 \cong
\Hom_{\bC \bRings_R}\left(\frac{R \lt x_1, \ldots, x_n \gt^\ab}{I^\ab}, S\right) .
\end{multline*}
Formula~\eqref{eq:proofZring equation} implies that $A$ is finitely presented as an object of $\bC\bRings_R$, which concludes the proof.
\item
A counterexample is provided by
Example~\ref{exa:non_finitely_generated_free_algebras}, upon observing that $R[y]$ is also an element in $\bZ\bRings_R$ and using part (ii) of this proposition.
\end{enumerate}
\end{proof}

Proposition~\ref{prop:finitely_presented_objects_central_algebras}(iii) has profound implications for the noncommutative geometry over $R$, and warrants a remark.

\begin{rmk} \label{rmk:finite_presentation}
In algebraic geometry the open localizations of a commutative ring are taken to be morphisms of finite presentation (in most topologies), because, geometrically, this means that they are morphisms of finite dimension between the associated affine schemes. 
However, as illustrated by Example~\ref{exa:non_finitely_generated_free_algebras}, the notions of finite presentation in the commutative and noncommutative settings are not compatible.
Indeed, the Zariski open localizations used, say, in the theory of schemes or in the theory of algebraic varieties are no longer of finite presentation as morphisms of $\bRings_R$ for a general commutative ring $R$, which makes noncommutative geometry fundamentally incompatible with the theory of schemes, when the notion of finite presentation comes into play. Namely, in our setting, if one were to use morphisms of finite presentation in $\bRings_R$ to define the Zariski topology, then the associated spectrum built out of this theory would be a singleton --- cf.~Proposition~\ref{prop:wrong_spectrum}. 
\end{rmk}

\subsubsection{Reformulation in terms of algebraic theories}
\label{Reformulation as an algebraic category}

The best way to explain the incompatibility described so far between the notions of finite presentation is perhaps to think in terms of algebraic theories in the sense of Lawvere~\cite{Lawvere}. In this subsection we briefly explain how this is done. 

\

Let $\sC$ be an essentially small category with finite coproducts. The \emph{algebraic category} of $\sC$ is the full subcategory of the category of pre-sheaves on $\sC$
\[ \bP_\Sigma(\sC) = \bFunc^\times(\sC^\op, \bSets) \subset \bFunc(\sC^\op, \bSets) = \bPsh(\sC) \]
where $\bFunc^\times$ denotes the class of functors that preserve finite products (of $\sC^\op$). Therefore, objects of $\bFunc^\times(\sC^\op, \bSets)$ are contravariant functors $\sF: \sC^\op \to \bSets$ such that the canonical morphism
\[ \sF(A \coprod B) \to \sF(A) \times \sF(B) \]
is an isomorphism. 

In the same way it is possible to show that objects of $\bPsh(\sC)$ can be written as arbitrary colimits in $\bPsh(\sC)$ of representable pre-sheaves, one can also show that objects of $\bP_\Sigma(\sC)$ are precisely those pre-sheaves that can be written as \emph{sifted colimits}\footnote{A colimit $\limind_{i \in I} A_i$ is said \emph{sifted} if the partially ordered set $I$ is such that the diagonal map $I \to I \times I$ is cofinal. For example, filtered colimits are sifted, but not all sifted colimits are filtered.} of representable pre-sheaves. 
For this reason, the reader should think of the objects of $\sC$ as the free and finitely generated objects of some class of algebraic structures and of the objects of $\bP_\Sigma(\sC)$ as obtained from the free ones by considering infinite unions and quotients by relations. Almost all categories of algebraic objects can be presented as $\bP_\Sigma(\sC)$ for a suitable $\sC$. We provide below some basic examples.

\begin{exa} \label{exa:algebriac_theories}
\begin{itemize}
\item If $\sC$ is the category of finite sets, then $\bP_\Sigma(\sC)$ is the category of sets.
\item If $\sC$ is the category of finitely generated and free abelian groups, \ie abelian groups of the form $\Z^n$ for $n \in \N$, then  $\bP_\Sigma(\sC) \cong \bAb$. This is (roughly) equivalent to say that any abelian group can be written as a quotient of a free abelian group, \ie it has a presentation in terms of (possibly infinitely many) generators and relations.
\item The previous example can be generalized immediately to left (or right) modules over any ring. Indeed, if $R$ is any ring and $\sC$ is the category of finitely generated and free left (or right) modules then $\bP_\Sigma(\sC) \cong {}_R \bMod$ (resp. $\bP_\Sigma(\sC) \cong \bMod_R$).
\end{itemize}
\end{exa}

In the spirit of the previous subsection, let us focus our attention on the algebraic theories of rings. Let us start with the case $R = \Z$. 
Let $\bPoly_\Z \subset \bC\bRings_\Z$ denote the full subcategory spanned by polynomial algebras in finitely many variables, \ie the algebras $\Z[x_1, \ldots, x_n]$. It is well-known that
\[ \bP_\Sigma(\bPoly_\Z) \cong \bC\bRings_\Z. \]
One way to prove this equivalence is by applying \cite[Theorem 6.9]{ARV} noticing that $\bPoly_\Z$ is precisely the subcategory of perfectly presented generators of $\bC\bRings_\Z$, in the sense of \cite[Definition 5.3]{ARV}.
Intuitively, this is equivalent to saying that every commutative ring can be presented as a quotient of an algebra of polynomials in an arbitrary number of variables by some relations.

 The noncommutative situation is very similar. 
Let $\bFree_\Z \subset \bRings_\Z$ denote the full subcategory spanned by free algebras in finitely many variables, \ie the algebras $\Z \lt x_1, \ldots, x_n \gt$ introduced so far. It is well-known that
\[ \bP_\Sigma(\bFree_\Z) \cong \bRings_\Z. \]
Again the equivalence can be easily reduced to an application of \cite[Theorem 6.9]{ARV}.

Let us now recast the case of a general (commutative) $R$ as an algebraic category. Since the generators of the polynomial algebra with coefficients in $R$ can be obtained as the image, via the tensor product functor $(\minus) \otimes_\Z R$, of the polynomial algebras with coefficients in $\Z$ --- see~\eqref{commmutative ring} --- the slice category $\bC\bRings_R$ of $R$-algebras is the algebraic category associated to $\bPoly_R$. 

Let us define the category $\bFree_R \subset \bRings_R$ as the category spanned by the free noncommutative algebras in finitely many variables over $R$, \ie the algebras of the form $R \lt x_1, \ldots, x_n \gt$. One might wonder whether these objects generate the category $\bRings_R$; this turns out not to be the case, because $\bRings_R$ contains a lot of objects for which $R$ does not map into the center, see \eg Example \ref{exa:non-central_free_algebras} and  Example~\ref{exa:non_finitely_generated_free_algebras}. Consequently, the category $\bRings_R$ is not the algebraic category associated to the category $\bFree_R$ of free algebras over $R$, because it contains algebras that cannot be written as quotient of free algebras in any number of variables. Below we complete the analysis of Example~\ref{exa:non-central_free_algebras}, to impress on this crucial phenomenon.

\begin{exa} \label{exa:non_central_algebra}
Consider the inclusion $\Z[x] \to \Z \lt x, y \gt$. The ring on the LHS is commutative, whereas in the target $x y \ne y x$. Hence, $\Z[x]$ is not mapped into the center of the target ring. In particular, the $\Z[x]$-algebra $\Z \lt x, y \gt$ that cannot be written as a quotient of 
\[ \Z[x] \lt \{ y_k \}_{k \in K} \gt \]
for any set of variables, where $K$ is an arbitrary indexing set.
\end{exa}

Inspired by the above example, and in the spirit of what done in the previous subsection, one turns to $\bZ \bRings_R$, looking for an intermediate situation. One can show that $\bZ \bRings_R$ is, in fact, the algebraic category $\bP_\Sigma(\bFree_R)$:
\[ \bP_\Sigma(\bFree_R) \cong \bZ\bRings_R\,.\]

Finally, in order to view $\bRings_R$ as an algebraic category, we denote by $\bN\bC\bFree_R$ the full subcategory of $\bRings_R$ comprising algebras $R \{ x_1, \ldots, x_n \}$ defined in accordance with~\eqref{noncommutative ring}.
Then we have that
\[ \bP_\Sigma(\bN\bC\bFree_R) \cong \bRings_R. \]
Again, these claims can be proved by a straightforward application of \cite[Theorem 6.9]{ARV}.

\begin{rmk}
We should like to point out that the equivalence $\bP_\Sigma(\bN\bC\bFree_R) \cong \bRings_R$ is true and makes sense also when $R$ is noncommutative, whereas the category $\bZ \bRings_R$ does not make sense (or at least is a very unnatural object to consider) for noncommutative rings.
\end{rmk}

\subsection{Homotopical algebras}

As in the previous subsection, $R\in \bC\bRings$ is a fixed commutative $\Z$-algebra.
After recalling the theory of discrete rings over $R$, we now would like to provide a similar discussion in the homotopical setting. The categories $\bH\bC\bRings_R$ and $\bH\bRings_R$ have formal properties that are very different from the ones of discrete objects, if they are considered as bare categories, \ie \,without considering their model structures that enrich them to $\infty$-categories. 
The main difference between the discrete and the homotopical settings is that homotopy categories of model categories usually do not have all limits and colimits. In particular, they usually do not have all pullbacks nor all pushouts. 
Nevertheless, it is possible to construct objects that are very similar to pullbacks and pushouts but lack the universal property\footnote{To restore the universal properties of these constructions and get rid of the ``negligible indeterminacy" of homotopy pushouts and homotopy pullbacks, it is necessary to use the language of $\infty$-categories. However, there is no need for us to discuss such matters in detail, in that it is possible to define homotopy pullbacks and homotopy pushouts directly in $\bH\bC\bRings_R$ and $\bH\bRings_R$; these pullbacks and pushouts will just lack universal properties characterising as $\infty$-categorical limits or colimits.}. These objects are called \emph{homotopy pullbacks} and \emph{homotopy pushouts} and we shall define them in the framework of model categories\footnote{
Homotopy limits and homotopy colimits can be considered in all the different settings where homotopy theory can be developed. We presented our homotopy categories via model categories, therefore we will discuss the definition of homotopy limits and colimits only for model categories.}. 

Let us start by recalling the following definition.

\begin{defn} \label{defn:derived_functor}
Let $F: \sC \to \sD$ be a functor between model categories. The \emph{left derived functor} of $F$, if it exists, is the right Kan extension
\[
\begin{tikzpicture}
\matrix(m)[matrix of math nodes,
row sep=2.6em, column sep=2.8em,
text height=1.5ex, text depth=0.25ex]
{ \sC &  \sD  \\
\bHo(\sC)  & \bHo(\sD) \\};
\path[->]
(m-1-1) edge node[auto] {$F$} (m-1-2);
\path[->]
(m-1-2) edge node[auto] {} (m-2-2);
\path[->]
(m-1-1) edge node[auto,xshift=1cm, yshift=0.1cm] {$\Downarrow$} (m-2-1);
\path[->]
(m-2-1) edge node[auto] {$\L F$} (m-2-2);
\end{tikzpicture}.
\]
The \emph{right derived functor} of $F$, if it exists, is the left Kan extension
\[
\begin{tikzpicture}
\matrix(m)[matrix of math nodes,
row sep=2.6em, column sep=2.8em,
text height=1.5ex, text depth=0.25ex]
{ \bC &  \bD  \\
\bHo(\sC)  & \bHo(\sD) \\};
\path[->]
(m-1-1) edge node[auto] {$F$} (m-1-2);
\path[->]
(m-1-2) edge node[auto] {} (m-2-2);
\path[->]
(m-1-1) edge node[auto,xshift=1cm, yshift=0.1cm] {$\Uparrow$} (m-2-1);
\path[->]
(m-2-1) edge node[auto] {$\R F$} (m-2-2);
\end{tikzpicture}.
\] 
\end{defn}

 Notice that the notion of derived functor introduced in Definition \ref{defn:derived_functor} is compatible with the notion of derived functors of a Quillen adjoint pair introduced before.
In general, derived functors may not exist because Kan extensions of functors do not always exist. The notion of derived functor permits us to define homotopy limits and homotopy colimits.

\begin{defn} \label{defn:homotopy_limit}
Let $\sC$ be a model category\footnote{We recall that by definition model categories have all limits and colimits, in the $1$-categorical sense.}, $I$ a poset and consider the functor $\limpro_{i \in I} (\minus): \sC^I \to \sC$. Then, the \emph{homotopy limit functor} is the derived functor $\R \limpro_{i \in I} (\minus)$. Similarly, the \emph{homotopy colimit functor} is the derived functor $\L \limind_{i \in I} (\minus)$ of the functor $\limind_{i \in I} (\minus)$.
\end{defn}

In Definition \ref{defn:homotopy_limit} we have taken for granted that a model structure on $\sC$ induces a model structure on the category of diagrams $\sC^I$. We refer to \cite[Chapter 13]{Dug} for an exposition of the standard methods to endow $\sC^I$ with model structures.

\begin{prop} \label{prop:existence_homotopy_lim_colim}
The categories $\bH\bC\bRings_R$ and $\bH\bRings_R$ have all homotopy limits and homotopy colimits.
\end{prop}
\begin{proof}
This is a well known result for which we provide a brief justification.
We notice that the categories of simplicial rings and simplicial commutative rings have all limits and colimits because they are categories of functors valued in complete and cocomplete categories and therefore limits and colimits are computed objectwise. Therefore, we only need to check that the limit and colimit functors are Quillen functors, which implies that they are derivable in the sense of Definition \ref{defn:homotopy_limit}. Since both $\bS \bC\bRings_R$ and $\bS \bRings_R$ are combinatorial model categories, for any diagram $I$ we can endow $(\bS \bC\bRings_R)^I$ (resp. $(\bS\bRings_R)^I$) with the projective (resp. injective) model structure making all the required adjunctions Quillen. Once again, we refer to \cite[Chapter 13]{Dug} for further details about these constructions.
\end{proof}

Later on, we will be interested in computing homotopy pushouts (of homotopy rings) and homotopy pullbacks (of affine noncommutative spaces). By Definition \ref{defn:homotopy_limit} the homotopy pushouts in $\bH\bC\bRings_R$ and $\bH\bRings_R$ are given by the derived functor of the pushout in $\bS\bC\bRings_R$ and $\bS\bRings_R$, respectively. By the description of pushouts in $\bC\bRings_R$ and $\bRings_R$, we can deduce that the pushouts in $\bS\bC\bRings_R$ are given by computing the commutative simplicial ring
\[ (X^\bullet \otimes_{Y^\bullet} Z^\bullet)^n = X^n \otimes_{Y_n} Z^n \]
for $X^\bullet, Y^\bullet, Z^\bullet \in \bS\bC\bRings_R$. Whereas, the the pushouts in $\bS\bRings_R$ are given by the simplicial ring
\[ (X^\bullet \ast_{Y^\bullet} Z^\bullet)^n = X^n \ast_{Y_n} Z^n \]
for $X^\bullet, Y^\bullet, Z^\bullet \in \bS\bRings_R$, see \cite[Section 14.8]{Stack}. Therefore, the homotopy pushouts in $\bH\bC\bRings_R$ and $\bH\bRings_R$ are computed as
\[ A \otimes_B^\L C \]
and 
\[ A \ast_B^\L C \,, \]
respectively, and Proposition \ref{prop:existence_homotopy_lim_colim} ensures that these derived functors exist.

To define finitely presented objects in $\bH\bC\bRings_R$ and $\bH\bRings_R$ we need to recall another concept from homotopy theory. We recall that both $\bS\bRings_R$ and $\bS \bC\bRings_R$ are simplicial model categories. This means that given two objects $X, Y \in \bS\bC\bRings_R$ (or $X, Y \in \bC\bRings_R$) there is a natural simplicial set of morphisms between them
\[ \ul{\Hom}(X, Y). \] 
To describe the simplicial set $\ul{\Hom}(X, Y)$ we notice that the categories $\bS\bRings_R$ and $\bS\bC\bRings_R$ satisfy the hypothesis of \cite[Proposition 4.2]{RSS} and therefore we can use their formula to define it. So, for any set $X$ and any $A \in \bRings_R$ (resp. $A \in \bC \bRings_R$) we define
\[ A \cdot X = \bigast_{x \in X} A \ \ \ \ \text{(resp. } A \cdot X = \bigotimes_{x \in X} A  \text{)} \]
and then for any $X^\bullet \in \bS\bSets$ and $A^\bullet \in \bS\bRings_R$ (resp. $A^\bullet \in \bH\bC\bRings_R$) we define the simplicial set
\[ (A \otimes X)^n = A^n \cdot X^n. \]
With this notation, we define
\[ \ul{\Hom}_{\bS\bRings_R}(X, Y)^n = \Hom_{\bS\bRings_R}(X \otimes \Delta_n, Y) \]
for $X, Y \in \bS\bRings_R$, and similarly in the case of $\bS\bC\bRings_R$. One can show that the description just given for the mapping spaces of $\bS\bRings_R$ and $\bRings_R$ implies that the inclusion $\bH\bC\bRings_R \to \bH\bRings_R$ is not fully faithful.

We also recall that the category of simplicial sets has a natural model structure defined by Quillen (see \eg~\cite[Chapter 2]{Hov}), therefore it makes sense to ask for the existence of the derived functor that would make the following diagram commutative
\[
\begin{tikzpicture}
\matrix(m)[matrix of math nodes,
row sep=2.6em, column sep=2.8em,
text height=1.5ex, text depth=0.25ex]
{ \bS\bC\bRings_R &  \bS\bSets  \\
\bH\bC\bRings_R  & \bH\bSets \\};
\path[->]
(m-1-1) edge node[auto] {$\ul{\Hom}$} (m-1-2);
\path[->]
(m-1-2) edge node[auto] {} (m-2-2);
\path[->]
(m-1-1) edge node[auto] {} (m-2-1);
\path[->, dashed]
(m-2-1) edge node[auto] {$\R \ul{\Hom}$} (m-2-2);
\end{tikzpicture},
\]  
where $\bH\bSets = \bHo(\bS\bSets)$ is the homotopy category of $\bS\bSets$ with respect to the Quillen model structure, that is equivalent to the homotopy category of topological spaces with the standard model structure.
It turns out that the derived functor $\R \ul{\Hom}$ exists and can be computed as
\[ \R \ul{\Hom}(X, Y) =  \ul{\Hom}(Q(X), Q(Y)) \]
where $Q(\minus)$ is the cofibrant replacement functor. One important property of the derived hom-spaces is that
\[ \pi_0(\R \ul{\Hom}(X, Y)) = \Hom_{\bH \bC \bRings_R}(Q(X), Q(Y)). \]

\begin{defn} \label{defn:homotopically_finite_presentation}
An object $X \in \bH\bRings_R$ (resp. $X \in \bH\bC\bRings_R$) is said \emph{homotopically of finite presentation} if the natural map
\[ \L \limind_{i \in I} \R \ul{\Hom} (X, Y_i) \to \R \ul{\Hom} (X, \L \limind_{i \in I} Y_i ) \]
is an equivalence for all filtered systems $\{ Y_i \}_{i \in I}$.
\end{defn}

\begin{prop} \label{prop:homotopically_finite_presentation}
\begin{enumerate}
\item  $X \in \bH\bRings_R$ is of finite presentation if and only if it is a retract of a finite cell of finitely generated free algebras.
\item  $X \in \bH\bC\bRings_R$ is of finite presentation if and only if it is a retract of a finite cell of finitely generated polynomial algebras.
\end{enumerate}
\end{prop}
Before proving the proposition we briefly explain the terminology used. We refer to \cite[Definition 1.2.3.4]{TV2} for a more detailed discussion and to \cite[Proposition 1.2.3.5]{TV2} for a more general version of our Proposition \ref{prop:homotopically_finite_presentation}. We say that an object $X \in \bH\bRings_R$ is a \emph{strict finite cell of finitely generated free algebras} if there exists a finite sequence of morphisms
\[ R = X_0 \to X_1 \to \cdots \to X_n = X  \]
and for any $i$ a push-out square 
\[
\begin{tikzpicture}
\matrix(m)[matrix of math nodes,
row sep=2.6em, column sep=2.8em,
text height=1.5ex, text depth=0.25ex]
{ R \{ x_1, \ldots x_m \} &  R \{ y_1, \ldots y_l \}  \\
  X_i  & X_{i+1} \\ };
\path[->]
(m-1-1) edge node[auto] {} (m-1-2);
\path[->]
(m-1-2) edge node[auto] {} (m-2-2);
\path[->]
(m-1-1) edge node[auto] {} (m-2-1);
\path[->]
(m-2-1) edge node[auto] {} (m-2-2);
\end{tikzpicture}.
\] 
We say that $X$ is a \emph{finite cell of finitely generated free algebras} if it is equivalent to a strict finite cell of finitely generated free algebras. The same definition is given for objects of $\bH\bC\bRings_R$ replacing the free algebras with polynomial algebras.
\begin{proof}[Proof of Proposition~\ref{prop:homotopically_finite_presentation}
]
We first notice that $\bH\bRings_R$ (resp. $\bH\bC\bRings_R$) are compactly generated and the objects $R \{ x_1, \ldots, x_m \}$ (resp. $R[x_1, \ldots, x_m]$) are a set of $\omega$-compact and $\omega$-small generators. Therefore, \cite[Proposition 1.2.3.5 (4)]{TV2}, applies to both $\bH\bRings_R$ and $\bH\bC\bRings_R$. More precisely, we can use the same reasoning of \cite[Proposition 1.2.3.5 (4)]{TV2} to show that all the objects $R \{ x_1, \ldots, x_m \}$ (resp. $R[x_1, \ldots, x_m]$) are homotopically finitely presented. Once this is proved it is easy to check that reatracts of finite cells of the $\omega$-compact and $\omega$-small generators are finitely generated and, conversely, every finitely presented object is a retract of such a finite cell.
\end{proof} 

Recall that there are fully faithful inclusions $\bRings_R \subset \bH\bRings_R$ and $\bC\bRings_R \subset \bH\bC\bRings_R$ where the classical categories are embedded as discrete objects. So, it is natural to ask what is the relation between the classical notion of finite presentation and the homotopical version of finite presentation. It turns out that the homotopical version is strictly stronger than the classical version. We briefly explain why this is the case in the next example.

\begin{exa} \label{exa:classical_vs_homotopical_finite_pres}
Suppose that $A \in \bC\bRings_R$ is of finite presentation (in the noncommutative setting we can reason in a similar fashion). This means that we can write
\[ A \cong \coker( R[x_1, \ldots, x_n] \to R[y_1, \ldots, y_m]) \]
in $\bC\bRings_R$. Now, the operation of taking cokernel does not exist in $\bH\bC\bRings_R$ and is replaced by the homotopical cokernel (also known as cofiber or cone). So, we can consider the object
\[ \L \coker( R[x_1, \ldots, x_n] \to R[y_1, \ldots, y_m]) \]
that is well-defined in $\bH\bRings_R$. Notice that it is always true that 
\[ \pi_0( \L \coker( R[x_1, \ldots, x_n] \to R[y_1, \ldots, y_m]) ) \cong A \]
but, in general, for specific examples of $A$,  $\pi_n( \L \coker( R[x_1, \ldots, x_n] \to R[y_1, \ldots, y_m]) ) \ne 0$ for $n > 0$. Such examples are obtained by quotienting $R[y_1, \ldots, y_m]$ by an ideal whose generators do not form a regular sequence. Therefore, the (finite) presentation given by the homotopical cokernel above is not a presentation of $A$ in $\bH\bC\bRings_R$, but it presents another object whose homotopy groups are not concentrated in degree $0$. This means that $A$ has a different presentation in $\bH\bC\bRings_R$ and this different presentation may be finite or not. We do not provide an explicit example of an object of finite presentation in $\bC\bRings_R$ for which its presentation in $\bH\bC\bRings_R$ is not finite anymore, but the reader should be convinced that it is possible to manufacture one. 
\end{exa}

We would like to emphasise that although Example~\ref{exa:classical_vs_homotopical_finite_pres} shows that being homotopically of finite presentation is different from being classically of finite presentation for a morphism of commutative rings, for many important classes of morphisms the two notions are equivalent. Since the mismatch between the two notions is given by higher cohomology groups, for flat maps the two notions of finite presentation agree. Therefore, for example, for smooth or \'etale maps of commutative rings, being of finite presentation or homotopical finite presentation is equivalent.

Moreover, the above discussion implies that in the homotopical world the same incompatibility between the notions of finite presentation in the commutative and noncommutative settings remains. Indeed, since for a morphism of commutative rings being homotopically of finite presentation implies being classically of finite presentation, then being not classically of finite presentation implies being not homotopically of finite presentation. Thus, we have see in Example \ref{exa:non_finitely_generated_free_algebras} that there are commutative rings $R$ for which the algebras $R[x_1, \ldots, x_n]$ are not of finite presentation in $\bRings_R$ and hence these are not of finite presentation in $\bH\bRings_R$. We do not dwell further into the study of the homotopical finite presentation of morphisms of $\bH \bRings_R$ because the main constructions of this work get rid of the hypothesis of finite presentations on the morphisms.

We will also use the following property of the homotopical monoidal Dold-Kan correspondence.

\begin{lemma} \label{lemma:monoidal_Dold_Kan_commute_free_product}
Let 
\[ \L N: \bH\bRings_R \leftrightarrows \bHo(\bD\bG\bA_R^{\le 0}): ( \L N:)^{-1} \]
and 
\[ (\R \Gamma)^{-1}: \bH\bRings_R \leftrightarrows \bHo(\bD\bG\bA_R^{\le 0}) : \R \Gamma. \]
be the adjunctions given by the homotopical monoidal Dold-Kan equivalence, \cf Proposition \ref{prop:monoidal_Dold_Kan}.
Then,
\[ \L N(A \ast_R^\L B) \cong \L N (A) \ast_R^\L \L N (B) \]
and 
\[ \R \Gamma(C \ast_R^\L D) \cong \R \Gamma (C) \ast_R^\L \R \Gamma (D) \]
for all $A, B \in \bHo(\bSRings_R)$ and $C, D \in \bHo(\bD\bG\bA_R^{\le 0})$.
\end{lemma}
\begin{proof}
By Proposition~\ref{prop:monoidal_Dold_Kan} both $\L N$ and $\R \Gamma$ commute with the tensor product. The free product of algebras can be written in terms of direct sums and tensor products, therefore both $\L N$ and $\R \Gamma$ commute with both these operation because they are both equivalences. Or, equivalently, we can notice that Quillen equivalences commute with homotopy limits and homotopy colimits in general, so all the above functors commute with homotopy pushouts.
\end{proof}

\subsection{Modules over homotopical rings}
\label{Modules over homotopical rings}

We now consider $R \in \bH\bRings_\Z$. So $R$ is a fixed, not necessarily commutative, homotopical ring in this subsection. We briefly discuss the category of (left, right or bi-) modules over an object $A \in \bH\bRings_R$ (or $A \in \bH\bC\bRings_R$ if $R$ is commutative). It is convenient to work at the level of model categories. We can think of an object $A \in \bH\bRings_R$ as a simplicial ring; the theory of monoidal categories allows one to associate to $A$ the category of simplicial left modules ${}_A \bS \bMod$, simplicial right modules $\bS \bMod_A$ and simplicial bimodules $\bS \bBiMod_A$ over $A$. All these categories inherit a model structure from the model structure on ${}_R \bS \bMod$, $\bS \bMod_R$, and $\bS \bBiMod_R$, respectively. For each of these categories one can then construct the associated homotopy category:
\[ {}_A \bH \bMod := \bHo({}_A \bS \bMod), \ \ \bH \bMod_A := \bHo(\bS \bMod_A), \ \ \bH \bBiMod_A := \bHo(\bS \bBiMod_A). \]

For later convenience, let us discuss the relation between a morphism $f: A \to B$ in $\bH\bRings_R$ and the corresponding induced functors between the homotopy categories just introduced. The morphism $f$ induces the following adjunctions:
\begin{eqnarray} 
\label{eq:adj_R_mod}
\L f^* =  B \otimes_A^\L (\minus) &:& {}_A \bH \bMod \leftrightarrows {}_B \bH \bMod: f_*\,, \\ 
\label{eq:adj_L_mod}
\L f^* = (\minus) \otimes_A^\L B &:& \bH \bMod_A \leftrightarrows \bH \bMod_B: f_*\,, \\
\label{eq:adj_bimod}
\L f^* = (\minus) \otimes_A^\L B &:& \bH \bBiMod_A 
\leftrightarrows \bH \bBiMod_B: f_*\,, \\  \label{eq:adj_alg}
\L f^* = (\minus) \ast_A^\L B &:& \bH \bRings_A \leftrightarrows \bH \bRings_B: \R f_*\,, 
\end{eqnarray} 
where the right adjoints are obtained via the morphism $f$ by forgetting the $B$-module (resp. $B$-algebra) structure. These functors will play a fundamental role in all that follows. As a final remark, we notice that when we write $f_*$ in place of $\R f_*$ we mean that the derived functor functor is trivial, in the sense that the left Kan extension of Definition \ref{defn:derived_functor} is a natural isomorphism.

\section{The noncommutative spectrum}  \label{sec:spectrum}

This section is the core of this paper. Using derived geometry, we will introduce a robust notion of spectrum for homotopical rings. This will be achieved by defining a Grothendieck topology on the category $\bH\bRings_\Z^\op$, in analogy with the Zariski topology used in scheme theory. 
We will argue that $\bH \bRings_\Z^\op$ is a good candidate for a noncommutative replacement of the category of affine schemes. Denoting by $\bdAff_\Z:=\bH\bRings_\Z^\op$ the category of \emph{derived noncommutative affine schemes}, let us introduce the formal duality functor (\ie~the formal contravariant functor
mapping an object to itself and a morphism to the opposite morphism)
\begin{align}
\label{Spec formal}
\textrm{Spec} : & \;\bH\bRings_\Z \to \bdAff_\Z\\
\nonumber
& A \mapsto \Spec(A)=:X\,. 
\end{align}

In the remainder of this section we will provide a geometric construction of the functor $\Spec$ that closely mimics the classical construction of the spectrum of a commutative ring given by Grothendieck. Namely, we will define a Grothendieck topology on $\bdAff_\Z$ that almost precisely generalises the (classical) Zariski topology of scheme theory.

Let us begin by giving a preliminary definition.

\begin{defn} \label{defn:homological_homotopical_epimorphism}
A morphism $f: A \to B$ in the category $\bH \bRings_\Z$ is said to be a 
\begin{itemize}
\item \emph{homological epimorphism} if the canonical map $B \otimes_A^\L B \to B$ is an equivalence;
\item \emph{homotopical epimorphism} if the codiagonal map $B \ast_A^\L B \to B$ is an equivalence.
\end{itemize}
\end{defn}

The next proposition motivates the choice of homotopical epimorphisms as noncommutative localizations and guarantees that our construction will be `compatible' with the classical one when applied to commutative rings. For further details, we refer to Section \ref{sec:grothen}.

\begin{prop} \label{prop:characterization_homotopy_epi}
Let $f$ be a morphism in $\bH\bRings_\Z$. The following statements are equivalent:
\begin{enumerate}[(a)]
\item The morphism $f$ is a homological epimorphism.
\item The morphism $f$ is a homotopical epimorphism.
\item The functor $\R f_*:\bH\bRings_B \to \bH\bRings_A$ is fully faithful. 
\item The functor $f_*:\bH\bMod_B \to \bH\bMod_A$ is fully faithful. 
\item The functor $f_*:{}_B \bH\bMod \to {}_A \bH\bMod$ is fully faithful. 
\item The functor $f_*:\bH\bBiMod_B \to \bH\bBiMod_A$ is fully faithful.
\end{enumerate}
 
\end{prop}
\begin{proof}
The equivalence between assertions $(a)$ and $(b)$ was proved in \cite[Theorem 4.4]{Laz2}. Observe that \cite[Theorem 4.4]{Laz2} is stated for unbounded dg-algebras but it directly implies our statement in the category $\bH\bRings_\Z$. Indeed, the homotopy category of connective dg-algebras $\bHo(\bDGA^{\le 0})$ sits fully faithfully into the homotopy category of unbounded dg-algebras $\bHo(\bDGA)$ and, since the free product is a left derived functor, the free product of connective dg-algebras is always connective, independently of whether it is considered as objects in $\bHo(\bDGA^{\le 0})$ or in $\bHo(\bDGA)$. So, for a morphism $f: A \to B$ in $\bHo(\bDGA^{\le 0})$ the condition
$B \ast_A^\L B \stackrel{\simeq}{\to} B $
is equivalent to require that the same condition holds in $\bHo(\bDGA)$. Then, by Lemma~\ref{lemma:monoidal_Dold_Kan_commute_free_product} the functor $\R \Gamma$ commutes with free products and therefore
\[ B \ast_A^\L B \stackrel{\cong}{\to} B \ \ \ \then \ \ \ \R \Gamma (B) \ast_{\R \Gamma(A)}^\L \R \Gamma(B) \stackrel{\cong}{\to} \R \Gamma(B)\,; \]
since $\R \Gamma$ is a monoidal equivalence, its quasi-inverse $\L \Gamma^{-1}$ is strongly monoidal, too. Therefore, the isomorphism
$ \L \Gamma^{-1}(B \ast_A^\L B) \stackrel{\cong}{\to} \L \Gamma^{-1}(B)$ implies $\L \Gamma^{-1}(B) \ast_{\L \Gamma^{-1}(A)}^\L \L \Gamma^{-1}(B) \stackrel{\cong}{\to} \L \Gamma^{-1}(B)$;
this gives us that also in the connective case the notions of homological and homotopical epimorphism coincide, because $\R \Gamma$ is also compatible with the (derived) tensor product.

For the equivalence of the assertion $(a)$ with each of $(d)$, $(e)$, and $(f)$, we notice that the canonical morphism
$ B \otimes_A^\L B \to B $
is the counit of the adjunctions \eqref{eq:adj_R_mod}, \eqref{eq:adj_L_mod}, and \eqref{eq:adj_bimod}, respectively, because it is a map of bimodules. Then, the requirement that the counit transformation of an adjunction is a natural isomorphism is equivalent to the right adjoint being fully faithful. Similarly, the equivalence between $(b)$ and $(c)$ is obtained by observing that the codiagonal morphism
$B \ast_A^\L B \to B $
is the counit of the adjunction \eqref{eq:adj_alg}, and asking that it is a natural isomorphism is equivalent to ask that the right adjoint is fully faithful. This completes the proof.
\end{proof}

\subsection{The formal homotopy Zariski topology}

In order to define a Grothendieck topology on $\bdAff_\Z$ we first need to specify appropriate notions of \textit{open embedding} and \textit{cover} for the topology. More explicitly:

\begin{defn} \label{defn:homotopical_Zariski_open_immersion}
We call  \emph{formal homotopy Zariski topology} the topology $\mathscr{T}_{\Zar}$ whose open embeddings are formal homotopical Zariski open embedding and the covers are formal covers, where:
\begin{itemize}
\item A morphism $f: \Spec(B) \to \Spec(A)$ in $\bdAff_\Z$ is said to be a \textit{formal homotopical Zariski open embedding} if it is a homotopical epimorphism in the category $\bH \bRings_\Z$. 
\item A finite family of formal homotopical Zariski open embedding  $\{f_i:\Spec(B_i) \to \Spec(A)\}_i$ is called a \textit{formal cover} if the family of functors 
$\{\L f_i^*: \bH \bRings_A \to \bH \bRings_{B_i} \}_i$ is conservative.
\end{itemize}
\end{defn}

\begin{prop}\label{prop:Zar site}
The formal homotopy Zariski topology on $\bdAff_\Z$ defines a Grothendieck topology.
\end{prop}
\begin{proof}
To prove our claim we have to check that three conditions are satisfied.
\begin{enumerate}
\item That if a morphism $f \in \bdAff_\Z$ is an isomorphism, then it is also a homotopical Zariski open embeddings.  This is clearly true.
\item That, given a formal homotopical Zariski cover $\{U_i\to X\}$ and a morphism $Y \to X$, the pullback family $\{U_i\times_X^\R Y \to Y\}$ is formal homotopical Zariski cover. To this end,  if we write $X = \Spec(A)$, $Y = \Spec(B)$, and $U_i = \Spec(A_i)$ then $U_i\times_X^\R Y = \Spec(A_i \ast_A^\L B)$. So, since homotopy pushouts preserve homotopy epimorphisms, then the fact that the morphisms $f_i: A \to A_i$ are homotopy epimorphisms imply that the morphisms $f_i \ast^\L_A B: B \to A_i \ast^\L_A B$ are homotopy epimorphisms as well. Furthermore, we have that the family for functors $\{ \L f_i^*: \bH \bRings_A \to \bH \bRings_{A_i} \}$ is conservative by hypothesis. Let $C,C' \in \bH \bRings_B$ and assume that there is a morphism $C \to C'$ that induces isomorphisms  
$$C \ast_B^\L B \ast_B^\L A_i\simeq C' \ast_B^\L B \ast_A^\L A_i\,$$ 
for all $i$. We need to check that $C \ast_A^\L A_i \simeq C' \ast_A^\L A_i$. But by the dual of the (homotopical) pullback lemma (see \cite[Lemma 4.4.2.1]{HTT}), we get that in the diagram
\[
\begin{tikzpicture}
\matrix(m)[matrix of math nodes,
row sep=2.6em, column sep=2.8em,
text height=1.5ex, text depth=0.25ex]
{ A &  B & C  \\
A_i & B \ast^\L_A A_i & C \ast_B^\L B \ast^\L_A A_i  \\
};
\path[->, font=\scriptsize]
(m-1-1) edge node[auto] {} (m-1-2);\\
\path[->, font=\scriptsize]
(m-1-1) edge node[auto] {} (m-2-1);
\path[->, font=\scriptsize]
(m-1-2) edge node[auto] {} (m-2-2);
\path[->, font=\scriptsize]
(m-1-2) edge node[auto] {} (m-1-3);
\path[->,font=\scriptsize]
(m-2-1) edge node[auto] {} (m-2-2);
\path[->,font=\scriptsize]
(m-2-2) edge node[auto] {} (m-2-3);
\path[->,font=\scriptsize]
(m-1-3) edge node[auto] {} (m-2-3);
\end{tikzpicture}
\]
the fact that the two small squares are homotopy pushouts imply that also the outer square is a homotopy pushout. Therefore, we get a (non-canonical, up to a contractible space of choices) isomorphism
\[ C \ast_B^\L B \ast^\L_A A_i \cong C \ast^\L_A A_i. \]
And since the same is true for $C'$, we get isomorphisms
\[ C \ast^\L_A A_i \cong C' \ast^\L_A A_i \]
for all $i$. By hypothesis this implies that $C$ and $C'$ are isomorphic as objects of $\bH \bRings_A$; but the forgetful functor $\bH \bRings_B \to \bH \bRings_A$ is conservative, therefore the isomorphism lifts to an isomorphism of objects of $\bH \bRings_B$.

\item Finally it remains to check  the stability of covers by composition. Given a homotopical Zariski cover $\{U_i\to X\}$ and, for every $i$, homotopical Zariski covers $\{V_{i,j}\to U_i\}$ of $U_i$, this amounts to check that the family $\{V_{i,j} \to X\}$ is a homotopical Zariski cover of $X$. Again, let us adopt the notation $X = \Spec(A)$, $U_i = \Spec(A_i)$, and $V_{i, j} = \Spec(B_{i, j})$, and let us work algebraically in the category $\bH \bRings_\Z$.
It is easy to check that the composition of two homotopy epimorphisms is a homotopy epimorphism, so the composite morphism $V_{i, j} \to U_i\to X$ is a homotopy epimorphism for all $i,j$. It remains to check the conservativity condition of the cover. Let $C, C' \in \bH \bRings_A$ and assume that there is a morphism $C \to C'$ that induces isomorphisms  
$$ C \ast_A^\L B_{i,j} \simeq C' \ast_A^\L B_{i,j}\,$$ 
for all $i, j$. We use again the (dual of the) homotopical pullback lemma (\cf \cite[Lemma 4.4.2.1]{HTT}) to deduce that in the diagram
\[
\begin{tikzpicture}
\matrix(m)[matrix of math nodes,
row sep=2.6em, column sep=2.8em,
text height=1.5ex, text depth=0.25ex]
{ A &  A_i & B_{i,j}  \\
 C & C \ast^\L_A A_i & C \ast^\L_A  A_i \ast^\L_{A_i} B_{i,j} \\
};
\path[->, font=\scriptsize]
(m-1-1) edge node[auto] {} (m-1-2);\\
\path[->, font=\scriptsize]
(m-1-1) edge node[auto] {} (m-2-1);
\path[->, font=\scriptsize]
(m-1-2) edge node[auto] {} (m-2-2);
\path[->, font=\scriptsize]
(m-1-2) edge node[auto] {} (m-1-3);
\path[->,font=\scriptsize]
(m-2-1) edge node[auto] {} (m-2-2);
\path[->,font=\scriptsize]
(m-2-2) edge node[auto] {} (m-2-3);
\path[->,font=\scriptsize]
(m-1-3) edge node[auto] {} (m-2-3);
\end{tikzpicture}
\]
the outer square is a homotopy pushout square. And the same with $C$ replaced by $C'$. This gives isomorphisms
\[ C \ast^\L_A B_{i,j} \cong C \ast^\L_A  A_i \ast^\L_{A_i} B_{i,j}, \ \  C' \ast^\L_A B_{i,j} \cong C' \ast^\L_A  A_i \ast^\L_{A_i} B_{i,j}. \]
So, the data of these isomorphism and the fact that for all $j$ the family $\{V_{i,j}\to U_i\}$ is a cover imply that for all $i$ we have isomorphisms
\[ C \ast^\L_A A_i \cong C' \ast^\L_A A_i\,. \]
This family of isomorphisms plus the fact that $\{U_i\to X\}$ is a cover imply that
\[ C \cong C' \]
as claimed.
\end{enumerate}
This concludes the proof.
\end{proof}

Proposition \ref{prop:Zar site} can be easily generalised as follows.

\begin{cor}
Let $R \in \bH\bRings_\Z$, not necessarily discrete. The formal homotopy Zariski topology on $\bH\bRings_\Z$ restricts to a topology on $\bH\bRings_R$.
\end{cor}
\begin{proof}
In the proof of Proposition \ref{prop:Zar site} only formal properties of homotopical epimorphisms have been used: the base ring $\Z$ plays no role. It is straightforward to check that the same arguments work in $\bH\bRings_R$.
\end{proof}

\begin{rmk} \label{rmk:failure_base_chage}
We point out another striking difference between the commutative and the noncommutative situations. In (derived) algebraic geometry, for affine morphisms of schemes the Base Change Theorem holds. In this situation, the theorem takes the following form. Let
\[
\begin{tikzpicture}
\matrix(m)[matrix of math nodes,
row sep=2.6em, column sep=2.8em,
text height=1.5ex, text depth=0.25ex]
{ A &  B  \\
  C & C \otimes_A^\L B \\
};
\path[->, font=\scriptsize]
(m-1-1) edge node[auto] {} (m-1-2);
\path[->, font=\scriptsize]
(m-1-1) edge node[auto] {} (m-2-1);
\path[->,font=\scriptsize]
(m-2-1) edge node[auto] {} (m-2-2);
\path[->,font=\scriptsize]
(m-1-2) edge node[auto] {} (m-2-2);
\end{tikzpicture}
\]
be a pushout square of commutative algebras, and let $M \in \bH\bMod_C$, then $M \otimes_C^\L C \otimes_A^\L B \cong M \otimes_A^\L B$. Of course, in this situation, the theorem is just a basic cancellation property of the tensor product functor. This theorem has several consequences in terms of relating the geometry of $\Spec(A)$ to the category $\bH\bMod_A$. But in the noncommutative situation, pushouts are computed via the free product of algebras. So, in this case we have that pushouts of algebras are given by diagrams of the form
\[
\begin{tikzpicture}
\matrix(m)[matrix of math nodes,
row sep=2.6em, column sep=2.8em,
text height=1.5ex, text depth=0.25ex]
{ A &  B  \\
  C & C \ast_A^\L B \\
};
\path[->, font=\scriptsize]
(m-1-1) edge node[auto] {} (m-1-2);
\path[->, font=\scriptsize]
(m-1-1) edge node[auto] {} (m-2-1);
\path[->,font=\scriptsize]
(m-2-1) edge node[auto] {} (m-2-2);
\path[->,font=\scriptsize]
(m-1-2) edge node[auto] {} (m-2-2);
\end{tikzpicture}
\]
therefore, the base change theorem would have the form of an isomorphism
\[ M \otimes_C^\L C \ast_A^\L B \cong M \otimes_A^\L B \]
for all $M \in \bH \bMod_C$, that is not true in general, not even with $M = C$. This creates extra complications in the relations between the geometry of $\Spec(A)$ and the properties of $\bH \bMod_A$ in the noncommutative case that are at the heart of the discussion in Section \ref{sec:gelfand}. We also notice, that instead, if we work with algebras over $A$, the base change theorem holds, \ie it is true that
\[ D \ast_C^\L C \ast_A^\L B \cong D \ast_A^\L B \]
for any $D \in \bH \bRings_C$. And this property is what we used to prove Proposition \ref{prop:Zar site}.
\end{rmk}

We will see in Section \ref{sec:properties_spectrum} that the connection between the geometry of $\Spec(A)$ and the category $\bH \bMod_A$ is not lost in the noncommutative case as homotopical epimorphisms are precisely determined by specific localizations of the category $\bH \bMod_A$.

\subsection{The category of localizations}

We now want to associate to any object $R \in \bH\bRings_\Z$ a geometric space using the homotopy Zariski topology on $\bH\bRings_R$. The procedure we adopt here is the same that can be used to obtain the Grothendieck spectrum of a commutative ring from the data of the classical Zariski topology on the category of affine schemes\footnote{In Proposition \ref{prop:compatision with Grothendieck} we will give more details about this procedure in the classical setting.}. 
To do this we introduce the following notation.

\begin{defn}\label{def:loc-ouv}
Let $R \in \bH\bRings_\Z$ and let us denote $X = \Spec(R) \in \bdAff_\Z$. 
\begin{itemize}
\item We define $\bLoc(R)$, the category of \emph{localizations}  of $R$, to be the subcategory of all homotopy epimorphisms in $\bH\bRings_R$ whose domain is $R$. 
\item We define $\bOuv(X)$, the category of \emph{open immersions} of $X$, to be the subcategory of all homotopy monomorphisms in $\bdAff_R$ 
whose codomain is $X$. 
\end{itemize}
\end{defn}

Observe that, by definition, we have $\bLoc(R) = \bOuv(X)^\op$.
Before moving forward with the study of the categories $\bLoc(R)$ and $\bOuv(X)$, let us discuss some examples of formal homotopy Zariski localizations.

\begin{exa} \label{exa:homotopy_epimorphism}
\begin{enumerate}[(i)]
\item \textbf{Classical Zariski localizations.} Let $R$ be a (discrete) commutative ring and consider a \emph{Zariski localization}, namely a morphism of commutative rings $R \to S$ that induces an open embedding $\Spec(S) \to \Spec(R)$ for the classical Zariski topology. As shown in \cite[Lemma 2.1.4 (1)]{TV3}, a Zariski localization induces a homotopical epimorphism in $\bH\bRings_R$. For the sake of completeness, let us mention that the family of Zariski localizations can be also characterised algebraically as the family of flat epimorphisms of finite presentation. Notice also that \cite[Lemma 2.1.4]{TV3} shows that the homotopy Zariski topology on commutative simplicial rings precisely restricts to the classical Zariski topology on discrete commutative rings. In this sense, the homotopical Zariski topology is a generalization of the classical Zariski topology of scheme theory.

\item \textbf{Flat epimorphisms.} A generalisation of the previous example is the class of flat epimorphisms between (discrete) rings $f: R \to S$. This class encompasses the class of Zariski localizations and although it is \emph{a priori} a huge class, it is well-known that the family of isomorphism classes (in the category of $R$-algebras) of epimorphisms with fixed source ring $R$ is a set. The sub-class of flat ring epimorphisms has been thoroughly studied in the literature and, by definition, is a sub-class of homotopical epimorphisms of ring. Indeed, if $S$ is flat over $R$, then the canonical map $S \otimes_R^\L S \to S \otimes_R S$ is an equivalence; since $f$ is an epimorphism, $S \otimes_R S \to S$ is an isomorphism, and the fact that the composition $S \otimes_R^\L S \to S \otimes_R S \to S$ is an equivalence is precisely the request of $R \to S$ being a homotopical epimorphism. 

The class of flat ring epimorphism is a nice and easy-to-study class of morphisms --- a natural candidate for a noncommutative generalisation of the class of Zariski localizations. However, at closer inspection one realises that this class cannot be used in the definition of a Grothendieck topology for noncommutative rings, because it is not stable under push-outs of rings. To see this, consider the following situation. Let $R$ be a ring and $s \in R$ be such that the multiplicative subset $\{1, s, s^2, \ldots \} \subset R$ is not an Ore subset and the localization $R \to R[s^{-1}]$ is not flat. Note that such examples exist (see, \eg, Example (vi) below). Then the diagram
\[
\begin{tikzpicture}
\matrix(m)[matrix of math nodes,
row sep=2.6em, column sep=2.8em,
text height=1.5ex, text depth=0.25ex]
{ \Z[x] &  \Z[x, x^{-1}]  \\
 R & R[s^{-1}] \\
};
\path[->, font=\scriptsize]
(m-1-1) edge node[auto] {} (m-1-2);
\path[->, font=\scriptsize]
(m-1-1) edge node[auto] {} (m-2-1);
\path[->,font=\scriptsize]
(m-2-1) edge node[auto] {} (m-2-2);
\path[->,font=\scriptsize]
(m-1-2) edge node[auto] {} (m-2-2);
\end{tikzpicture}
\]
is a cocartesian diagram of rings (\ie, $R[s^{-1}] \cong R \ast_{\Z[x]} \Z[x, x^{-1}]$), where the left vertical map is given by $x \mapsto s$. But $\Z[x] \to \Z[x, x^{-1}]$ is the prototype of a flat morphism of commutative rings, whereas $R \to R[s^{-1}]$ is by construction not flat.

\item \textbf{Ore localizations.} If $S \subset R$ is a multiplicative subset that satisfies the  Ore condition\footnote{The (right) Ore condition for a multiplicative subset $S$ of a ring $R$ is that for $a \in R$ and $s \in S$, the intersection $aS \cap sR \neq \emptyset$.}, then the localization map $R \to R[S^{-1}]$ is a flat ring epimorphism, and therefore it is a homotopical epimorphism. If $S$ does not satisfy the Ore conditions, then the localization map $R \to R[S^{-1}]$ is not flat and in general is not even a homotopical epimorphism.

\item \textbf{Localizations of discrete commutative rings.} If $A$ is a commutative simplicial ring (\ie, $A \in \bH \bC \bRings_\Z$) and $A \to B$ a homotopical epimorphism then $B$ is commutative (\ie, $B \in \bH \bC \bRings_\Z$) \footnote{An easy way to prove this assertion is the following. It is known, and we will give a proof in Proposition \ref{prop:smashing_homotopi_epi}, that any homotopy epimorphism $A \to B$ determines a unique smashing localization of $D(\bMod_A)$. In this case, the smashing subcategory generated by $A$, that is the unit of the $\otimes$-triangulated structure in $D(\bMod_A)$, is $D(\bMod_A)$ itself and it follows that all smashing subcategories, \ie\ kernels of smashing localizations, are $\otimes$-ideals. It is then a standard consequence of the Eckmann–Hilton argument that the algebra constructed in Proposition \ref{prop:smashing_homotopi_epi} is commutative in this situation.}. Moreover, if $A$ is Noetherian and discrete, then $B$ is also discrete by \cite[Theorem 1.1 (2)]{AMSTV}. In the next example we briefly discuss the non-Noetherian case, so let us focus on the Noetherian case here. In this situation, by the main results of \cite{NB}, localizing subcategories of $D(\bMod_A)$  precisely correspond to subsets of prime ideals spectrum $\Spec(A)$ and so do in particular smashing subcategories. Therefore, the smashing localizations of $D^{\le 0}(\bMod_A)$ correspond to inclusions $\Spec(B) \to \Spec(A)$ where the image is a generalization closed subset. 
This property of homotopical localizations of discrete rings being concentrated in degree $0$ crucially fails for noncommutative rings because of the existence of non-flat localizations. The Example (vi) illustrates an explicit instance where this phenomenon occurs.

\item \textbf{Localizations of non-Noetherian valuation rings.}  In the case of non-Noetherian commutative rings, the situation is more complicated than in the Noetherian case. Still, at least in some case, the family of homotopy epimorphisms can be explicitly described. In \cite[Section 2]{Baz} there is a summary of the classification of homotopical epimorphisms for valuation rings $R$. For example, in this situation, if $R$ is a valuation ring with maximal ideal $\mathfrak{m}$ such that $\mathfrak{m} = \mathfrak{m}^2$, then it is possible to show that the (non-flat) homomorphism $R \to \frac{R}{\mathfrak{m}}$ is a homotopical epimoprhism. A similar phenomenon happens in the theory of commutative $C^*$-algebras as shown in \cite{BamMi}.

\item \textbf{Localizations of discrete noncommutative rings.} We now give an example of a homotopical epimorphism from a discrete noncommutative ring to a non-discrete simplicial ring. This example is taken from \cite[Example 5.5]{Laz} and adapted to the conventions used in this paper. Let us work with connective dg-algebras and tacitly use the Dold--Kan equivalence whenever needed.  Let $R$ be a commutative ring that, for simplicity, we assume to be of characteristic zero. Consider the $R$-algebra
\[ A = \frac{R \lt s, t, u \gt}{(s t, u s)}. \]
One can show that the multiplicative subset $\{1, s, s^2, s^3, \ldots \} \subset A$ does not satisfy the Ore conditions and therefore the localization $A \to A[s^{-1}]$ is not flat. Moreover, one can check that it is not true that $A[s^{-1}] \ast_A^\L A[s^{-1}] \cong A[s^{-1}]$ and therefore this localization is not a homotopical epimorphism. This failure can be corrected by considering the homotopical localization of $A$ at $s^{-1}$, obtained, roughly speaking, by inverting $s$ homotopically rather than in the category of rings. We refer to \cite{Laz} for more details about this  procedure of homotopical inversion of elements of dg-algebras in general. For the purpose of this example it is enough to know that the result of this operation is always a homotopical epimorphism but the target ring might not be discrete even if the domain is. In our situation we can define the localization of $A$ at $s$ via the free product
\[ A \ast_{R[s]}^{\L} R [s, s^{-1}], \]
that is meant to be the push-out of the diagram
\[
\begin{tikzpicture}
\matrix(m)[matrix of math nodes,
row sep=2.6em, column sep=2.8em,
text height=1.5ex, text depth=0.25ex]
{ R[x] &  R[x, x^{-1}]  \\
 A & A \ast_{R[s]}^{\L} R [s, s^{-1}] \\
};
\path[->, font=\scriptsize]
(m-1-1) edge node[auto] {} (m-1-2);
\path[->, font=\scriptsize]
(m-1-1) edge node[auto] {} (m-2-1);
\path[->,font=\scriptsize]
(m-2-1) edge node[auto] {} (m-2-2);
\path[->,font=\scriptsize]
(m-1-2) edge node[auto] {} (m-2-2);
\end{tikzpicture},
\]
where the map $R[x] \to A$ is defined by $x \mapsto s$. Notice that $R[x] \to R[x, x^{-1}]$ is one of the most basic Zariski localizations of algebraic geometry, and hence it is a homotopical epimorphism by Example (i). And by Proposition \ref{prop:Zar site} this implies that $A \to A \ast_{R[s]}^{\L} R [s, s^{-1}]$ is a homotopical epimorphism. To explicitly compute the derived free product, we can follow \cite[Example 5.5]{Laz} where it is shown that the cofibrant $R$-dg-algebra
\[ C = R \lt s, t, u, z, w, a \gt \]
is a $R[s]-$cofibrant replacement of $A$.
Here $z$ and $w$ are generators in degree\footnote{We should warn the reader that in \cite[Example 5.5]{Laz} the degrees are positive, because \cite{Laz} adopts homological differentials for dg-modules, rather than cohomological differentials used in the present paper.} $-1$, $a$ is a generator in degree $-2$, and the differentials are defined by $dz = st$, $dw = u s$ and $da = uz - wt$.  Therefore, a calculation gives us
\[ A \ast_{R[s]}^{\L} R [s, s^{-1}] \cong C \ast_{R[s]} R[s, s^{-1}]. \]
One can check that the latter dg-algebra has cohomology in all even negative degrees (see~\cite[Example 5.5]{Laz} for a detailed computation).

\end{enumerate}
\end{exa}

With the next proposition, we shall show that the category $\bLoc(R)$  is equivalent to a complete join semi-lattice, namely, a partially ordered set that admits a least upper bound (join) for any nonempty arbitrary collection of subsets.	
 
\begin{prop} \label{prop:Loc_poset}
The category $\bLoc(R)$ is essentially small and is equivalent to a complete join semi-lattice.
\end{prop}
\begin{proof}
To prove that $\bLoc(R)$ is essentially small we notice that $\bH\bRings_\Z$ is a homotopically well-copowered category. One way see why this is true\footnote{Another possibility is to notice that well-copoweredness is always true for an algebraic category and also its homotopical version is true for homotopical algebraic categories. From our description of the category $\bRings_R$ as an algebraic category the reader can imagine that $\bH \bRings_R$ is a homotopical algebraic category. Since we have not introduced the homotopical version of the notion of algebraic category we prefer to argue in a more down-to-earth way to prove well-copoweredness.}, is to observe that Proposition \ref{prop:characterization_homotopy_epi} implies that the datum of a homotopy epimorphism is equivalent to a smashing localization of $\bH \bMod_R$ (see subsection~\ref{Comparison with the spectrum of smashing subcategories} for further details). These in turn are particular cases of smashing localizations of the unbounded derived category $D(R)$ and form a set, because $D(R)$ is compactly generated (\cf the introduction of \cite{Kra}). This shows that $\bLoc(R)$ is essentially small. To show that it is equivalent to a poset, we observe that for any $A, A' \in \bLoc(R)$ the set
\[ \Hom_{\bH \bRings_R}(A, A') \]
is either empty or a singleton, because the structure maps $R \to A$ and $R \to A'$ are homotopical epimorphisms. This poset has all joins that are given by the derived free product of algebras because it is the homotopical coproduct of algebras. 
\end{proof}

By duality, the above proposition tells us that the category $\bOuv(X)$ is equivalent to a complete meet semi-lattice, namely, a partially ordered set that admits a greatest lower bound (meet) for any nonempty arbitrary subset.  We will interpret the objects of $\bOuv(X)$ as open subsets of $X = \Spec(R)$, their meets as intersections of open subsets. We do not interpret joins of $\bOuv(X)$ as unions in $X$, because unions will be prescribed by the homotopy Zariski topology and the latter is not always compatible with the lattice structure of $\bOuv(X)$. There are two main reasons for why the joins in $\bOuv(X)$ do not precisely correspond to geometric unions in $X$. One is a phenomenon proper of noncommutative geometry that will be presented later on. The other is more classical, it is explained in the next remark and boils down to the fact that the objects of $\bOuv(X)$ form just a base for a topology not all the opens subsets of a topological space.

\begin{rmk} \label{rmk:finite_unions}
Already in the classical setting of the Zariski topology for commutative rings, it can happen that given two $U, V \in \bOuv(X)$, their join in $\bOuv(X)$ does not agree with their geometric union (this is equivalent to say that the union of two open affine subschemes should be computed in the category of schemes and not in the category of affine schemes). A simple example of this phenomenon happens when $X = \Spec(\C[x,y])$, $U = \Spec(\C[x,y][x^{-1}])$, and $V = \Spec(\C[x,y][y^{-1}])$. In this case $U \cup V$ is not an affine scheme (in the usual sense) and therefore it is not equivalent to $\Spec(A)$ for any $A \in \bH \bRings_\Z$. But clearly $U \cup V$ makes sense as a (non-affine) scheme. Alternatively, one can think of the scheme $U \cup V$ as an affine stacks in the sense of To\"en (see \cite{toen}) or even better in the framework of complicial algebraic geometry of \cite[Chapter 2.3]{TV2}. This amounts to work with the whole category of unbounded dg-algebras in place of the category of connective dg-algebras, which coincidentally is the setting of \cite{Laz} and \cite{Laz2}. In this alternative setting, we would get that the construction of $\bLoc(R)$ and $\bOuv(X)$ gives the (distributive if $R$ is commutative) lattice of smashing localizations of $D(\bMod_R)$ into which the $\bLoc(R)$ canonically embeds. But this difference is negligible because the sites associated to these two versions of the constructions are equivalent (see \cite{NB} for more information on the commutative side of this). We prefer to work with connective dg-algebras because we would like to distinguish the cohomological and the infinitesimal directions of derived geometry.
\end{rmk}

 Building upon Remark \ref{rmk:finite_unions}, it is easy to find homotopy epimorphisms from a discrete ring (even a commutative and Noetherian one) to a dg-algebra whose cohomology is in positive degrees (or equivalently to cosimplicial rings, via the  cosimplicial version or the stable version of the Dold--Kan equivalence). In classical algebraic geometry, such examples correspond to non-affine sub-schemes of affine schemes, and their dg-algebras correspond to their sheaf cohomology dg-algebra and are examples of affine stacks in the sense of To\"en (see \cite{toen} and \cite[Chapter 2.3]{TV2} for more on the notion of affine stacks).
 The fact that non-affine schemes must be detected by their cohomology is an immediate consequence of Serre's criterion for affineness. This observation says nothing new about the geometry of affine schemes, it just underlines that the geometry of affine schemes is interesting to study. A fact that does not require derived geometry to be appreciated. Therefore, we omit to discuss the ``complicial" version of the theory in this work (\ie\ using unbounded dg-algebras and modules), and focus on what we believe to be the more interesting connective part of the theory, from which the complicial theory follows formally by stabilisation.

\subsection{The noncommutative spectrum}\label{sec:defncspec}
By restricting the formal homotopy Zariski topology, that is defined on $\bdAff_\Z$, to the subcategory $\bOuv(X)$, with $X = \Spec(R)$, we are finally in a position to introduce the small homotopy Zariski site. Notice that for all $X$ we have that $\bOuv(X)$ is stable with respect to the homotopy Zariski topology in $\bH\bRings_\Z$, \ie it is stable by base changes and by homotopy Zariski covers.

\begin{defn} \label{defn:small_Zatiski_site}
For $X \in \bdAff_\Z$, we define $\bZar_X$ to be the small site whose underlying category is $\bOuv(X)$ and covers are homotopical Zariski covers.
\end{defn}

Since $\bOuv(X)$ is stable under pullbacks and covers in $\bdAff_\Z$, then Definition \ref{defn:small_Zatiski_site} is well posed. The association 
\begin{equation}
    \label{X to Zar(X)}
    X \mapsto \bZar_X
\end{equation}
immediately warrants the following questions.
\begin{itemize}
\item Does the site $\bZar_X$ have enough points?
\item Is the association \eqref{X to Zar(X)} functorial?
\item Is $\bZar_X$ equivalent\footnote{We recall that two sites $\bC$ and $\bD$ are called equivalent if there is a morphism of sites $\bC \to \bD$ that induces an equivalence $\bSh(\bD) \to \bSh(\bC)$ between the categories of sheaves. We underline that sites that a priori look very different can be equivalent.} to the site of a sober topological space?
\end{itemize}

In the remainder of this subsection, we will address these questions in turn, and answer each of them in the positive. 

\begin{prop} \label{prop:enough_points}
The site $\bZar_X$ has enough points and is equivalent to the site of a sober topological space.
\end{prop}
\begin{proof}
The topology of $\bZar_X$ is defined by families of covers whose cardinality is finite; therefore, it is a coherent site. By Deligne's Theorem, coherent sites have enough points (see \cite[Corollary IX.11.3]{MM}). Then, by \cite[Theorem IX.5.1]{MM} a topos is localic (\ie,  equivalent to the topos of a locale) if and only if can be defined on a poset. Finally, by \cite[Proposition IX.3.3]{MM} a localic topos is equivalent to the topos of a topological space if and only if it has enough points.
\end{proof}

A description of the points of the topos $\bZar_X$ will be provided later in Proposition \ref{prop:points_as_completelyprime_ultrafilters}.

\begin{prop}\label{prop:funct}
The association~\eqref{X to Zar(X)} is a 2-1 pseudofunctor\footnote{Since the class of all sites forms a $2$-category, not just an ordinary category, the correct notion of ``morphism" from a category to a $2$-category is that of a 2-1 pseudofunctor. This is a minor technical detail that disappears when working with $\infty$-categories, which is necessary because notions in category theory are defined up to equivalence and not up to isomorphism. It is possible to avoid the use of $2-1$ pseudofunctors by passing to equivalences classes and ``strictifying" 2-1 pseudofunctors to actual functors. The result of this process is not very natural, hence we refrain from doing this. The reader unfamiliar with the notion of 2-1 pseudofunctors and $2$-categories can ignore this nuance and just think in terms of functors: this usually does not lead to mistakes.} 
$\bdAff_R \to \bSites$, where $\bSites$ is the 2-category of small sites whose morphisms are continuous functors\footnote{These are functors that preserve cover families and are cover-flat.}.
\end{prop}
\begin{proof}
Let $f: B \to A$ be a morphism in $\bH \bRings_\Z$ and $\Spec(A)=Y\xrightarrow{f^{\op}}X=\Spec(B)$ the corresponding morphism in $\bdAff_\Z$. We define the functor $\tilde f:\bZar_X \to \bZar_Y \qquad $ as
$$(U\to X) \mapsto (Y\times_X^\R U \to Y) \,.$$
We first observe that this construction gives a 2-1 pseudofunctor rather than a functor. Indeed, if we consider another morphism $Y\xrightarrow{g^{\op}} Z=\Spec(C)$, then we have a canonical natural isomorphism of functors
$$ \widetilde {g \circ f} \cong \tilde g \circ \tilde f  $$
instead of an equality. This is due to the fact that the canonical isomorphism 
\[ Z \times_Y^\R (Y\times_X^\R U) \cong Z \times_X^\R U  \]
is not the identity in general.

We need to check that the functor $\tilde{f}$ is well defined and preserves covers for the homotopy Zariski topology. This is equivalent to show that the functor 
\[ f^*: \bH \bRings_A \to \bH \bRings_B \]
defined by $f^*(C) := C \ast_A^\L B$ preserves homotopy epimorphisms and homotopical Zariski covers. For the first property, we retrace the arguments from the proof of Proposition~\ref{prop:Zar site}, where a similar statement was proved for a morphism $A \to B$ that is a homotopy epimorphism, observing that the latter property did not play any role in the proof. As to the second property, let us consider a homotopical Zariski cover $\{U_i \to X\}$. We first notice that the morphisms of the family $\{ U_i\times_X^\R Y \to Y\}$ are homotopical Zariski open embeddings because the functor $(\minus) \ast_A^\L B$ preserves homotopy epimorphisms. Then, the fact that $\{ U_i\times_X^\R Y \to Y\}$ is a cover follows using the homotopical pullback lemma, as in Proposition~\ref{prop:Zar site}.

It remains to show that $\tilde{f}$ is cover-flat, \ie, that $\tilde{f}$ commutes with finite homotopy limits. 
To this end, it is enough to show that for any pair of maps $U \to V$ and $U \to W$ in $\bZar_X$ we have that $\tilde{f}(V\times_U^\R W)=\tilde{f}(V)\times_{\tilde{f}(U)}^\R \tilde{f}(W)$. To do this computation, we switch to the opposite category. By putting $U = \Spec(C)$, $V = \Spec(D)$ and $W = \Spec(E)$, we get that the homotopy push-out square
\[
\begin{tikzpicture}
\matrix(m)[matrix of math nodes,
row sep=2.6em, column sep=2.8em,
text height=1.5ex, text depth=0.25ex]
{ C &  D    \\
E & D\ast^\L_C E \\};
\path[->, font=\scriptsize]
(m-1-1) edge node[auto] {} (m-1-2);\\
\path[->, font=\scriptsize]
(m-1-1) edge node[auto] {} (m-2-1);
\path[->, font=\scriptsize]
(m-1-2) edge node[auto] {} (m-2-2);
\path[->,font=\scriptsize]
(m-2-1) edge node[auto] {} (m-2-2);
\end{tikzpicture}
\]
is mapped to the homotopy push-out square
\[
\begin{tikzpicture}
\matrix(m)[matrix of math nodes,
row sep=2.6em, column sep=2.8em,
text height=1.5ex, text depth=0.25ex]
{ C\ast^\L_A B & D \ast^\L_A B  \\
 E\ast^\L_A B & D\ast^\L_C E*^\L_A B \\};
\path[->, font=\scriptsize]
(m-1-1) edge node[auto] {} (m-1-2);\\
\path[->, font=\scriptsize]
(m-1-1) edge node[auto] {} (m-2-1);
\path[->, font=\scriptsize]
(m-1-2) edge node[auto] {} (m-2-2);
\path[->,font=\scriptsize]
(m-2-1) edge node[auto] {} (m-2-2);
\end{tikzpicture}
\]
because the derived free product preserves homotopy push-outs. Alternatively, one can check that there is a canonical chain of isomorphisms
$$D\ast_C^\L E\ast_A^\L B\simeq D\ast_C^\L C\ast_A^\L B\ast_{C\ast_A^\L B}^\L E\ast_A^\L B\simeq (D\ast_A^\L B)\ast_{C\ast_A^\L B}^\L (E\ast_A^\L B) $$
where we use the canonical isomorphism of functors
\[ D \ast^\L_C (\minus) \cong D\ast_C^\L C\ast_A^\L B\ast_{C\ast_A^\L B}^\L (\minus) \]
and
\[ D \ast^\L_C C \ast_A^\L (\minus) \cong D \ast_A^\L (\minus). \]
\end{proof}

Motivated by the previous results, we are finally in the position to define the noncommutative spectrum for a noncommutative ring.

\begin{defn}
For any $R\in\bH\bRings_\Z$, we call \emph{noncommutative spectrum} $\SpecNC(R)$ the sober topological space  equivalent to $\bZar_X$ given by Proposition \ref{prop:enough_points}. 
\end{defn}
Observe that the space $\SpecNC(R)$ is defined up to homeomorphism.

\begin{cor} \label{cor::funct}
The noncommutative spectrum 
\begin{equation}
\label{X to |Zar(X)|}
 \SpecNC: \bH\bRings_\Z \to \bTop \,,
\end{equation}
 is a functor from the homotopy category of connective dg-algebras over $\Z$ to the category of topological spaces.
\end{cor}
\begin{proof}
The association $\bZar_X \mapsto \SpecNC(R)$ is a functor from the 2-category of sites to the category of topological spaces. The pre-composition with the 2-1 pseudofunctor described in Proposition~\ref{prop:funct} yields the functor~\eqref{X to |Zar(X)|}.
\end{proof}

\subsection{The notion of point in the noncommutative spectrum}

In Proposition \ref{prop:enough_points} we have used abstract results from topos theory and the theory of locales to show that the topos associated to $\bZar_X$ has enough points. We now give a more concrete description of the points of the topological space $\SpecNC(A)$. In preparation for this, let us recall the following definition from topos theory.

\begin{defn} \label{defn:ideal_posite}
Let $(S, \tau)$ be a poset equipped with a topology. An \emph{ideal} of $(S, \tau)$ is a subset $I \subset S$ such that 
\begin{enumerate}[(i)]
\item given $s, t \in S$, the conditions $s \le t$ and $t \in I$ imply $s \in I$;
\item if $\{s_i \to s\}$ is a cover and $s_i \in I$ for all $i$, then $s \in I$.
\end{enumerate}
We denote the set of all ideals of $(S, \tau)$ by $\Id((S, \tau))$.
\end{defn}

Observe that, when $S$ is a distributive lattice and $\tau$ is the family of covers given by finite joins, it is well known that $\Id((S, \tau))$ is the frame of open subsets of a spectral topological space whose lattice of compact open subsets is isomorphic to $S$. Furthermore, all spectral topological spaces can be written as the frame of ideals of some distributive lattice.

\begin{prop} \label{prop:points_as_completelyprime_ultrafilters}
Let $A \in \bH\bRings_R$ and $X = \Spec(A)$. Then $\SpecNC(A)$ is the sober topological space associated with the frame of ideals $\Id(\bZar_X)$. Therefore, points of $\SpecNC(A)$ correspond to completely prime ultrafilters of $\Id(\bZar_X)$.
\end{prop}
\begin{proof}
Given a poset with a topology $(S, \tau)$, the poset $\Id((S, \tau))$ with respect to the inclusion of sets is a frame (see \cite[Proposition II.2.11]{Stone}). The fact that $(S, \tau)$ and $\Id((S, \tau))$ are equivalent sites is the content of \cite[Exercise IX.5 (c)]{MM}. Finally, the last claim is a straightforward application of the definition of points of a locale to $\SpecNC(A)$ (see, e.g., \cite[Chapter IX]{MM} and \cite[Chapter II]{Stone} for the basics of the theory of locales).
\end{proof}

The poset $\bZar_X$ and its topology contain all the information required to construct the topological space $\SpecNC(A)$ (in concrete terms, it contains a base for the topology). In particular, the frame $\Id(\bZar_X)$ is the frame of open subsets of $\SpecNC(A)$. 
A point $x \in \SpecNC(A)$ belongs to an open subset $U \subset \SpecNC(A)$ corresponding to an element $\cU \in \Id(\bZar_X)$ if and only if $\cU \in \sI_x$, where $\sI_x$ is the completely prime filter associated to $x$. This correspondence maps objects of $\Id(\bZar_X)$ to subsets of $\SpecNC(A)$, thus yielding an injective map $\Id(\bZar_X) \to \sP(\SpecNC(A))$ that intertwines the partial order in $\Id(\bZar_X)$ with the inclusion order in $\sP(\SpecNC(A))$. This map commutes with arbitrary unions and finite intersections, but does not commute with infinite intersections in general (because the intersection of an infinite family of open subsets is not an open subset in general).

\section{Properties of the noncommutative spectrum}
\label{sec:properties_spectrum}

The aim of this section is twofold: 
\begin{enumerate}
    \item  To construct a map from the noncommutative spectrum of a commutative ring to its Grothendieck spectrum of prime ideals;
    \item  To recast our definition of the spectrum in terms of smashing subcategories of $\bH \bMod_R$.
\end{enumerate}

\subsection{Comparison with the Grothendieck spectrum}\label{sec:grothen}

Our definition of (formal) Zariski covers differs from the one commonly used in derived algebraic geometry in a key aspect. In derived algebraic geometry one requires the pullback functors $\{\L f_i^*: \bH \bMod_A \to \bH \bMod_{B_i} \}_i $ to be conservative. If Definition~\ref{defn:homotopical_Zariski_open_immersion} is reinterpreted in the category of commutative (homotopical) rings, then one recovers the usual notion of cover, because for commutative (homotopical) rings the base change of algebras essentially agrees with the base change of modules. However, in the noncommutative setting the base changes of algebras is given by the free product, which usually does not agree with the base change of the underlying module. 

This raises the issue of comparing the new notion of cover with its commonly used counterpart. Before discussing this matter further, let us show by a simple counterexample that the two notions are, indeed, not equivalent for commutative rings.

\begin{exa} \label{exa:Kanda_counterexample}
Let $k$ be a commutative field. In the classical setting, the Grothendieck spectrum of $k\oplus k$ is nothing but $\SpecG(k \oplus k) = \{ \star, \star \}$ and a cover is given by the two projections $\pi_1: k \oplus k \to k$ and $\pi_2: k \oplus k \to k$. However, this is not the case in our setting, since the family of homotopical epimorphisms $\{\pi_1, \pi_2\}$ is not a cover in the sense of Definition~\ref{defn:homotopical_Zariski_open_immersion}. Indeed, consider the diagonal embedding $k \oplus k \to M_2(k)$ into the algebra of $2 \times 2$ square matrices. Then, it is easy to check that $k \ast_{k \oplus k}^\L M_2(k) \cong 0$ when computed via both maps $\pi_1$ and $\pi_2$. But $M_2(k) \ne 0$, therefore the family $\{\pi_1, \pi_2\}$ is not conservative in the sense of the definition\footnote{We would like to thank Ryo Kanda for drawing our attention to this example.}.
\end{exa}

\begin{comment}
\begin{lemma} \label{lemma:commutative_ast=tensor}
Let $A \in \bH\bC\bRings_\Z$, $A \to B$ be a homotopical epimorphism and $A \to C$ an object of $\bH\bRings_A$. Then we have a canonical isomorphism
\[ B \otimes_A^\L C \cong B \ast_A^\L C. \]
\end{lemma}
\begin{proof}
It is enough to prove that
\[ B \otimes_A C \cong B \ast_A C \]
when both $B$ and $C$ are cofibrant as $A$-algebras. In this case the structure maps $A \to B$ and $A \to C$ are split injections of the underlying $A$-modules, and therefore we can write
\[ B = A \oplus \ol{B}, \qquad C = A \oplus \ol{C}, \]
where $\ol{B}$ and $\ol{C}$ are the cokernels of the structure maps. By \cite[text after Definition 3.1 and Lemma 4.3]{Laz2} we have that $B \otimes_A \ol{B} = 0$, and by \cite[Lemma 2.1]{Laz2} there exists a filtration
\[ B \ast_A C = \bigcup_{n \in \N} F_n \]
with $F_0 = 0, F_1 = B$, and
\[ F_{2n+1}/F_{2n} \cong B \otimes_A (\ol{C} \otimes_A \ol{B})^{\otimes_A^n}, \]
\[ F_{2n+2}/F_{2n+1} \cong B \otimes_A (\ol{C} \otimes_A \ol{B})^{\otimes_A^n} \otimes_A \ol{C}. \]
So, if $A$ is commutative we get 
\[ F_{n+1}/F_n \cong 0 \]
for all $n > 1$. Thus, we get
\[  B \ast_A C = B \oplus B \otimes_A \ol{C} \cong B \otimes_A A \oplus B \otimes_A \ol{C} \cong B \otimes_A (A \oplus \ol{C}) \cong B \otimes_A C.  \qedhere\]
\end{proof}

Let us emphasise the fact that Lemma \ref{lemma:commutative_ast=tensor} holds only if $A$ is commutative. 

\end{comment}

In order to compare our new notion of cover with the classical one, we will resort to the following.

\begin{lemma} \label{lemma:conservative_zero_algebras}
Let $A \in \bH\bRings_\Z$.  A finite family of homotopical epimophisms $\{ A \to B_i \}_i$ is conservative if and only if for $C \in \bH\bRings_A$  it holds that 
$$C \cong 0 \iff C \ast_A^\L B_i \cong 0 \qquad \text{for all $i$.}$$ 
\end{lemma}
\begin{proof}
    Suppose that there exists $B \in \bH\bRings_A$ such that $C \ne 0$ and $C \ast_A^\L B_i = 0$ for all $i$. Then, the canonical morphism $C \to 0$ is sent to an isomorphism by all functors $(\minus) \ast_A^\L B_i$, but it is not an isomorphism. Thus, we do not have a conservative family of functors.
    
    Conversely, if $\{ A \to B_i \}_i$ is not conservative there exists a morphism $C \to D$ in $\bH\bRings_A$ that is not an isomorphism such that $C \ast_A^\L B_i \to D \ast_A^\L B_i$ is an isomorphism for all $i$. Taking the cofibers of the last morphisms gives $0$, so $\cofib(C \to D)$ is a non-zero object of $\bH\bRings_A$ such that $\cofib(C \to D) \ast_A^\L B_i = 0$ for all $i$, because derived functors commute with cofibers.
\end{proof}

\begin{prop} \label{prop:faithful_covers}
Let $\{ f_i: A \to B_i \}_{i \in I}$ be a cover for the homotopy Zariski topology. Then the family of functors
    \[ \{ \L f_i^*: \bH \bMod_A \to \bH \bMod_{B_i} \}_{i \in I} \]
    is faithful.
\end{prop}
\begin{proof}
    Suppose that $f: B \to C$ is a morphism in $\bH \bRings_A$ such that $f \ast_A^\L B_i \cong 0$ for all $i$. Then, since the algebras in $\bH \bRings_{B_i}$ are unital algebras and morphisms are morphisms of unital algebras, this implies that $C \ast_A^\L B_i \cong 0$ for all $i \in I$. Applying Lemma~\ref{lemma:conservative_zero_algebras} we find that $C \cong 0$, and hence $f = 0$. 
    We then notice that for any $A \in \bH \bRings_\Z$ we can embed the category $\bH \bMod_A$ faithfully into the category $\bH \bBiMod_A$ noticing that $\bH \bBiMod_A = \bH \bMod_{A \otimes_\Z^\L A^\op}$ and use the base change via the canonical morphism of rings $A \to A \otimes_\Z^\L A^\op$ given by $a \mapsto a \otimes a^\op$. It is easy to check that the obtained base change functor is faithful. 
    
    Then, we can faithfully embed the category of bimodules over $A$ into $\bH \bRings_{A_i}$ by means of the free algebra functor $\bFree_A: \bH \bBiMod_A \to \bH \bRings_A$. Notice that if $A \to B$ is a homotopical epimorphism, then for all $M \in \bH \bBiMod_A$ we have
    \[ \bFree_A(M) \otimes_A^\L B \cong \bFree_{B}(M \otimes_A^\L B). \]
    Indeed,
    \[ \bFree_A(M) \otimes_A^\L B \cong (A \oplus M \oplus M^{(\otimes_A^\L)^2} \oplus \cdots) \otimes_A^\L B \cong \]
    \[ (B \oplus M \otimes_A^\L B \oplus (M \otimes_A^\L B)^{(\otimes_B^\L)^2} \oplus \cdots) \cong \bFree_B(M \otimes_A^\L B)  \]
    because the relation $B \cong B \otimes_A^\L B$ can be used iteratively. We also observe that there is a natural isomorphism
    \[ \bFree_A(M) \ast_A^\L B \cong \bFree_{B}(M \otimes_A^\L B) \]
    because both algebras satisfy the same universal property. Indeed, let $C \in \bH \bRings_B$, then by working out definitions we get
    \[ \Hom_{\bH \bRings_B}(\bFree_{B}(M \otimes_A^\L B), C) \cong \Hom_{\bH \bBiMod_B}(M \otimes_A^\L B, C) \cong \] \[ \Hom_{\bH \bBiMod_A}(M, C) \cong \Hom_{\bH \bRings_A}(\bFree_A(M), C) \cong \Hom_{\bH \bRings_B}(\bFree_A(M) \ast_A^\L B, C). \]
    This shows that there is a commutative (up to natural isomorphism) square of functors
    \[
\begin{tikzpicture}
\matrix(m)[matrix of math nodes,
row sep=2.6em, column sep=4.8em,
text height=1.5ex, text depth=0.25ex]
{ \bH \bMod_A & \bH \bBiMod_A &  \bH \bRings_A  \\
\bH \bMod_B & \bH \bBiMod_{B}  & \bH \bRings_{B} \\};
\path[->] (m-1-1) edge node[auto] {$\scriptscriptstyle{(\minus) \otimes_A^\L (A \otimes_\Z^\L A^\op)}$} (m-1-2);
\path[->] (m-1-2) edge node[auto] {} (m-2-2);
\path[->] (m-2-1) edge node[auto] {$\scriptscriptstyle{(\minus) \otimes_B^\L (B \otimes_\Z^\L B^\op)}$} (m-2-2);
\path[->] (m-1-1) edge node[auto] {} (m-2-1);
\path[->] (m-1-2) edge node[auto] {} (m-1-3);
\path[->] (m-2-2) edge node[auto] {} (m-2-3);
\path[->] (m-1-3) edge node[auto] {} (m-2-3);
\end{tikzpicture}.
\]
We can put such squares together for the family of homotopical epimorphisms $A \to B_i$ in the statement of the proposition and thus obtain a commutative square of functors
    \[
\begin{tikzpicture}
\matrix(m)[matrix of math nodes,
row sep=2.6em, column sep=2.8em,
text height=1.5ex, text depth=0.25ex]
{ \bH \bMod_A &  \bH \bBiMod_A &  \bH \bRings_A  \\
\bigoplus_{i \in I}  \bH \bMod_{B_i} & \bigoplus_{i \in I} \bH \bBiMod_{B_i}  & \prod_{i \in I} \bH \bRings_{B_i} \\};
\path[->] (m-1-1) edge node[auto] {} (m-1-2);
\path[->] (m-1-2) edge node[auto] {} (m-2-2);
\path[->] (m-2-1) edge node[auto] {} (m-2-2);
\path[->] (m-1-1) edge node[auto] {} (m-2-1);
\path[->] (m-1-2) edge node[auto] {} (m-1-3);
\path[->] (m-2-2) edge node[auto] {} (m-2-3);
\path[->] (m-1-3) edge node[auto] {} (m-2-3);
\end{tikzpicture}
\]
where the horizontal functors are faithful. This fact plus the computation in the beginning of the proof show that the only morphism in $\bH \bMod_A$ that is sent to $0$ by the composition of the functors is the zero morphism. Since the bottom horizontal maps are faithful, the same is true also for remaining vertical maps. But this is equivalent to saying that the map is faithful because the left vertical functor is an additive functor. 
\end{proof}

\begin{rmk} \label{rmk:faithful_modules_not_covers}
 Note that Example \ref{exa:Kanda_counterexample} shows that, unlike in the classical situation, in the noncommutative setting there exist faithful families of localization functors $\{ \L f_i^*: \bH \bMod_A \to \bH \bMod_{B_i} \}_{i \in I}$ for $f_i$ homotopical epimorphisms that are not covers in the sense of Definition \ref{defn:homotopical_Zariski_open_immersion}.
\end{rmk}

Thus, even though --- as per Remark \ref{rmk:faithful_modules_not_covers} ---- the converse to Proposition \ref{prop:faithful_covers} does not hold, this is enough to compare $\SpecNC(A)$ to the classical $\SpecG(A)$ when $A$ is a classical commutative ring. To do this we introduce an intermediate topology.

\begin{defn} \label{defn:fine_Zariski_topology}
Let $A \in \bH\bRings_\Z$. We define the \emph{fine (homotopical) Zariski topology} on $\bOuv(A)$ to be the topology whose covers are finite families of homotopical epimorphisms $\{ B \to C_i \}_{i \in I}$ such that for all homotopical epimorphisms $B \to D$ we have $D \cong 0$ if and only if $D_i \ast_B^\L C_i \cong 0$ for all $i\in I$.
\end{defn}

\begin{rmk}
\label{rmk:fine_Zariski_topology}
Observe that by Lemma \ref{lemma:conservative_zero_algebras}, the condition of cover from Definition~\ref{defn:fine_Zariski_topology} is the same as that in Definition \ref{defn:homotopical_Zariski_open_immersion}, but restricted only to the subcategory of $\bH \bRings_A$ of homotopical epimorphisms out of $A$. 
\end{rmk}

\begin{prop} \label{prop:fine_Zariski_well_defined}
    The fine Zariski topology is well defined and endows $\bOuv(A)$ with the structure of a coherent posite, that is therefore equivalent to a topological space which we denote $\SpecNC_\fine(A)$.
\end{prop}
\begin{proof}
    The proof is similar to that of Proposition~\ref{prop:Zar site}. Let $\{A \to B_i\}_{i \in I}$ be a cover for the fine Zariski topology, and let $A \to C$ be homotopical epimorphism. We need to show that the family $\{ C \to C \ast_A^\L B_i \}_{i \in I}$ is a cover as well. Let $C \to D$ be a homotopical epimorphism such that $C \ast_A^\L B_i \ast_C^\L D \cong 0$ for all $i$, then we have
    \[ C \ast_A^\L B_i \ast_C^\L D \cong B_i \ast_A^\L C \ast_C^\L D \cong B_i \ast_A^\L D\,. \]
   Since the composition $A \to C \to D$ is a homotopical epimorphism, one can conclude that $D \cong 0$, and, hence, that $\{ C \to C \ast_A^\L B_i \}_{i \in I}$ is a cover. 
   
   It remains only to check that if $\{ A \to B_i \}_{i \in I}$ is a cover and for every $i\in I$ we have covers $\{ B_i \to C_j \}_{j \in J_i}$, then the composition $\{  A \to C_j \}_{j \in \cup_{i \in I} J_i}$ is a cover. Let $A \to D$ be a homotopical epimorphism such that $A \ast_{C_j}^\L D \cong 0$ for all $j \in \bigcup_{i \in I} J_i$. Then, by the isomorphism
    \[ A \ast_{C_j}^\L D \cong A \ast_{B_i}^\L B_i \ast_{C_j}^\L D \]
    we deduce that $B_i \ast_{C_j}^\L D \cong 0$ for all $j$. Since the families $\{ B_i \to C_j \}_{j \in J_i}$ are covers, it ensues that $B_i \ast_{C_j}^\L D \cong 0$ for all $i\in I$. The latter, in turn, implies $D \cong 0$ because $\{ A \to B_i \}_{i \in I}$ is a cover.
\end{proof}

\begin{rmk}
\label{rmk:fine_Zariski_topology_not_functorial}
Let us emphasise that the fine Zariski topology is not functorial in $\bH \bRings_\Z$. Indeed, the map $k \oplus k \to M_2(k)$ from Example~\ref{exa:Kanda_counterexample} induces a functor $\bOuv(k \oplus k) \to \bOuv(M_2(k))$ given by $(\minus) \ast_{k \oplus k} M_2(k)$ which sends the cover (for the fine Zariski topology) $\{\pi_1, \pi_2\}$ to the family of maps $\{M_2(k) \to 0, M_2(k) \to 0\}$. But the latter is not a cover for the fine Zariski topology. It follows that the fine Zariski topology is functorial when restricted to the subcategory of homotopical epimorphisms.
We will elaborate further on the geometry of this situation in Example~\ref{exa:Kanda_morphism}.
\end{rmk}

\begin{comment}
We can now prove our claim.

\begin{prop}
\label{proposition remark 4.4}
Let $A \in \bH\bC\bRings_\Z$ and $\{ A \to B_i\}$ be a finite family of homotopical epimorphisms, then the base change functors
\[ \{\L f_i^*: \bH \bMod_A \to \bH \bMod_{B_i} \}_i \]
form a conservative family if and only if the base change functors
\[ \{\L f_i^*: \bH \bRings_A \to \bH \bRings_{B_i} \}_i \]
form a conservative family.
\end{prop}
\begin{proof}
By Lemma \ref{lemma:commutative_ast=tensor} the two functors are computed by the same formula and by Lemma \ref{lemma:conservative_zero_algebras} conservativity can be checked on the zero object. Therefore, the conservativity for the two families of functors is equivalent.
\end{proof}

\begin{rmk} \label{rmk:faithfulness}
In descent theory another notion is traditionally of relevance, the notion of faithfulness. It is a classical result that for the familiy of base change functors 
\[ \{\L f_i^*: \bH \bMod_A \to \bH \bMod_{B_i} \}_i, \]
in the case when $A$ is commutative, conservativity is equivalent to faithfulness. This is no longer true in the noncommutative case.
\end{rmk}

\end{comment}

In the light of Proposition~\ref{prop:fine_Zariski_well_defined}, we are now in a position to compare the classical spectrum of a commutative ring $A$ to $\SpecNC(A)$ via $\SpecNC_\fine(A)$ by means of the following lemma.

\begin{lemma} \label{lemma:fine_covers}
    Let $A \in \bC\bRings_\Z$. 
    \begin{enumerate}
        \item Localizations for the classical Zariski topology are homotopical epimorphism.
        \item A family of classical localizations forms a cover for the classical Zariski topology if and only if it forms a cover for the fine Zariski topology.
    \end{enumerate}
\end{lemma}
\begin{proof}
    \begin{enumerate}
        \item This is a well known fact, whose proof we outline below. Let $f: A \to B$ be a homotopical epimorphism $\bH\bC\bRings_R$ that is of finite presentation (in the commutative sense). Then one can show that $B$ must be concentrated in degree $0$ (\cf \cite[Proposition 2.2.2.5 (4)]{TV2}) and that $f$ is a flat epimorphism of commutative rings (\cf \cite[Lemma 2.1.4 (i)]{TV3}). Therefore, the class of homotopical epimorphisms of finite presentation under $A$ precisely corresponds to the class of localizations for the classical Zariski topology on $A$. 
        \item Consider a cover $\{A \to B_i\}_{i \in I}$ for the classical Zariski topology. One can show (\cf \cite[Lemma 2.1.4 (ii)]{TV3}) that this is equivalent to require the derived descent condition for the category of modules (\ie by the conservativity of the family of functors $(\minus) \otimes_A^\L B_i$). But this is, in turn, equivalent to the condition of being a cover for the fine Zariski topology, because the latter is equivalent to ask that the objects $B_i$ generate $\bH \bMod_A$ as a pre-triangulated category.
    \end{enumerate}
\end{proof}

At this poing, we have all tools at our disposal to compare the noncommutative spectrum with the classical notion of spectrum of prime ideals of a commutative ring introduced by Grothendieck, which, for the sake of clarity, we denote by $\SpecG(A)$.

\begin{prop}\label{prop:compatision with Grothendieck}
Let $A \in \bC\bRings_\Z$. Then, there exist canonical maps of topological spaces 
\[
\begin{tikzpicture}
\matrix(m)[matrix of math nodes,
row sep=2.6em, column sep=2.8em,
text height=1.5ex, text depth=0.25ex]
{ \SpecNC_\fine(A) & \SpecG(A)  \\
  \SpecNC(A) \\};
\path[->, font=\scriptsize]
(m-1-1) edge node[auto] {$\pi_A$} (m-1-2);
\path[->, font=\scriptsize]
(m-1-1) edge node[auto] {$\iota_A$} (m-2-1);
\end{tikzpicture}.
\]
\end{prop}

\begin{proof}
By Lemma \ref{lemma:fine_covers}, the classical Zariski topology admits a homotopical characterisation, which implies that the posite that determines $\SpecG(A)$ is canonically a sub-posite of $\bOuv(A)$ with the fine Zariski topology. Moreover, the inclusion functor identifies the restriction of the fine Zariski topology to $\SpecG(A)$ with the classical Zariski topology. This gives the projection map $\pi_A$. The map $\iota_A$ is obtained by observing that Proposition~\ref{prop:faithful_covers} implies that the identity functor $\bOuv(A) \to \bOuv(A)$ is continuous when the domain is endowed with the homotopy Zariski topology and the codomain with the fine Zariski topology.
\end{proof}

The map $\pi_A$ defines a projection from $\SpecNC_\fine(A)$ to $\SpecG(A)$. We expect this map to be surjective and even a quotient map, at least in favourable cases.

\

We conclude the subsection with a final remark on the functoriality of the fine Zariski topology. For any $A \in \bH \bRings$, one always has the canonical map $\iota_A: \SpecNC_\fine(A) \to \SpecNC(A)$, Defined as in the proof of Proposition~\ref{prop:compatision with Grothendieck}. Thus, although $\SpecNC_\fine(A)$ is not functorial in $A$, for any morphism $B \to A$ in $\bH \bRings$ we have a diagram

\[
\begin{tikzpicture}
\matrix(m)[matrix of math nodes,
row sep=2.6em, column sep=2.8em,
text height=1.5ex, text depth=0.25ex]
{ \SpecNC_\fine(A) & \SpecNC_\fine(B) \\
  \SpecNC(A) &  \SpecNC(B) \\};
\path[->, dashed, font=\scriptsize]
(m-1-1) edge node[auto] {} (m-1-2);
\path[->, font=\scriptsize]
(m-1-1) edge node[auto] {$\iota_A$} (m-2-1);
\path[->, font=\scriptsize]
(m-1-2) edge node[auto] {$\iota_B$} (m-2-2);
\path[->, font=\scriptsize]
(m-2-1) edge node[auto] {} (m-2-2);
\end{tikzpicture}
\]
where the dashed arrow does not always exists as a continuous map of topological spaces but always exists as a continuous correspondence 
\[  \SpecNC_\fine(A) \to \SpecNC(B) \leftarrow \SpecNC_\fine(B). \]
This is reminiscent of other constructions of spectra of rings previously appeared in literature: \cf \cite[Section 4]{AS}. Therefore, albeit in a weaker sense, the fine Zariski topology has functorial properties. We will see in Example~\ref{exa:Kanda_morphism} one such correspondence that is not a function.

\subsection{Comparison with the spectrum of smashing subcategories}
\label{Comparison with the spectrum of smashing subcategories}

In order to perform explicit computations of the topological space $\SpecNC(A)$, let us give an alternative characterisation of homotopy epimorphisms. This will reduce the problem of describing $\SpecNC(A)$ to the problem of understanding the lattice of some specific kind of subcategories of $\bH\bMod_A$ that we now specify.
Let $A \in \bH\bRings_R$ and consider a full embedding of categories $i_*:\bC \hookrightarrow \bH\bMod_A$ such that
\begin{enumerate}[(i)]
\item the functor $i_*$ is a triangulated functor, \ie, it commutes with the shift operation;
\item the category $\bC$ is closed by finite coproducts, direct summands and cones;
\item the functor $i_*$ has both a left and a right adjoint, denoted $i^*$ and $i^!$, respectively.
\end{enumerate}

\begin{defn} \label{defn:smashing_localization}
We call   \emph{smashing localization}  of $\bH\bMod_A$ the functor $i^*: \bH\bMod_A\to\bC$ and \emph{smashing subcategory} of $\bH\bMod_A$ the kernel of $i^*$.
\end{defn}

\begin{rmk}
Observe that $\bH\bMod_A$ is not triangulated because the suspension functor is not an equivalence. This kind of category is sometimes called \emph{pre-triangulated} category. By stabilizing $\bH\bMod_A$ we obtain a triangulated category that is precisely $D(\bMod_A)$, the unbounded derived category of $A$-modules. It is easy to check that the notion of smashing localization introduced in Definition \ref{defn:smashing_localization} is compatible with the standard notion of smashing localization from the theory of triangulated categories. Therefore, Definition~\ref{defn:smashing_localization} determines a subclass of the smashing localizations of $D(\bMod_A)$.
\end{rmk}

The next proposition already appears in literature in several different versions. We refer to \cite[§4]{NS} for one of the earliest appearances. We provide a full proof for the version that is more suitable to our setting.

\begin{prop} \label{prop:smashing_homotopi_epi}
Let $A \in \bH\bRings_R$.
Smashing localizations of $\bH\bMod_A$ and homotopy epimorphisms of $A$ to connective dg-algebras are in 1-to-1 correspondence.
\end{prop}
\begin{proof}
Let $A \to B$ be a homotopical epimorphism. Then the inclusion of the full subcategory $\bH \bMod_B \to \bH \bMod_A$ is a smashing localization because the restriction of scalars functor always has both adjoints.

Let $i^*: \bH \bMod_A \to \bC  $ be a smashing localization. Since
$i_*$ is fully faithful, we have that
\[ \R \Hom_\bC(i^*(A), i^*(A)) \cong \R \Hom_{\bH \bMod_A}(A, i_* i^*(A)) \cong i^*(A) \]
equips $i^*(A)$ with a structure of dg-algebra for which the canonical map $A \to i^*(A)$ is a morphism of dg-algebras. Then, since $i^*$ is a left adjoint functor for $i_*$, we have that
\[ i^*(M) \cong i^*(A) \otimes_A^\L M  \]
for all $M \in \bH \bMod_A$, because the functors $i^*$ and $i^*(A) \otimes_A^\L (\minus)$ have the same values on the generators of $\bH \bMod_A$\footnote{It is well-known that $\bH \bMod_A$ is generated by $A$, see for instance \cite[Theorem 3.1.1]{SS2} for a much stronger result.} and both commute with direct sums and cones. But since $i^*$ is a localization, the functorial isomorphism
\[ i^* \circ i^* \cong i^* \]
translates, via the functor $i^*(A) \otimes_A^\L (\minus)$, into
the functorial isomorphism
\[ i^*(A) \otimes_A^\L (\minus) \circ i^*(A) \otimes_A^\L (\minus) \cong i^*(A) \otimes_A^\L (\minus)\,. \]
This is equivalent to say that $A$ is a homotopical epimorphism.
\end{proof}

\begin{rmk}
The same argument as in Proposition~\ref{prop:smashing_homotopi_epi} shows that there is a 1-to-1 correspondence between smashing localizations of $D(\bMod_A)$ and homotopy epimorphisms of $A$ to (unbounded) dg-algebras. From our perspective, dg-algebras with cohomology in positive degrees correspond to non-affine objects which can be viewed as affine stacks in the sense of To\"en. Thus, in the setting of complicial algebraic geometry these are still affine objects determined by their dg-algebra of functions. We also mention that a similar setting has been recently considered in \cite{Rak} where such algebras are called \emph{derived algebras}.
\end{rmk}

\begin{defn}
  We denote by $\bSma(\bH \bMod_A)$ the poset of smashing subcategories of $\bH \bMod_A$.   
\end{defn}

\begin{prop} \label{prop:iso_smashing_poset}
The bijection established by Proposition~\ref{prop:smashing_homotopi_epi} yields an isomorphism of posets
\[ \bSma(\bH \bMod_A) \cong \bOuv(X)\,, \]
where $X = \Spec(A)$.
\end{prop}
\begin{proof}
We only need to check that the bijection in Proposition \ref{prop:smashing_homotopi_epi} preserves the partial orders. Let $f: A \to B$ be a homotopy epimorphism and $\L f^*: \bH \bMod_A \to \bH \bMod_B$ be the corresponding localization. The partial order in $\bOuv(X)$ is given by the reversing the inclusion of the localizations $f_*$ with $f$ ranging in the set of homotopy epimorphisms with domain $A$. This is isomorphic (by Proposition~\ref{prop:smashing_homotopi_epi}) to the poset obtained by equipping the kernels of the localizations $\L f^*$ with partial order given by inclusions. The latter poset is, by definition, $\bSma(\bH \bMod_A)$.
\end{proof}

We showed in Proposition \ref{prop:Loc_poset} that $\bOuv(X)$ is a complete meet semi-lattice. By \cite[Proposition I.4.3]{Stone} this implies that $\bOuv(X)$ is a complete lattice, so that by the isomorphism of Proposition \ref{prop:iso_smashing_poset} we have that meets and joins on $\bSma(\bH \bMod_A)$ and $\bOuv(X)$ agree. More precisely, meets in $\bSma(\bH \bMod_A)$ are given by the intersections of subcategories whereas meets in $\bOuv(X)$ are the fiber products of homotopical monomorphisms. 
Joins in $\bSma(\bH \bMod_A)$ are given by the smallest smashing subcategory generated by a family of objects therein, whereas joins in $\bOuv(X)$ are the corresponding localizations.

\

Proposition \ref{prop:iso_smashing_poset} offers a way to simplify the computation of $\bOuv(X)$. This happens in particular under additional conditions on $A$.

\begin{cor} \label{cor:iso_smashing_poset}
Suppose that $A$ is concentrated in degree $0$ and $\bH \bMod_A$ satisfy the telescope conjecture\footnote{The telescope conjecture --- see, e.g., \cite{Kra} --- has been verified for certain rings (for example, hereditary rings), disproven for others (for instance, non-Noetherian rings), and it is an open question to characterise the class of rings or dg-algebras for which it holds.}. Then 
\[ \bSma(\bH\bMod_A) \cong \bThick( \bH\bMod_A^c )\,, \]
where $\bThick( \bH \bMod_A^c )$ is the lattice of thick subcategories\footnote{A subcategory of $\bH\bMod_A$ is called \emph{thick} if it satisfies the first two conditions in Definition \ref{defn:smashing_localization}.} of $\bH\bMod_A^c$, and $\bH\bMod_A^c$ is the full subcategory of compact objects\footnote{An object $C \in \bH\bMod_A$ is called \emph{compact} if the functor $\Hom_{\bH\bMod_A}(C, \minus)$ commutes with direct sums.} of $\bH\bMod_A$.
\end{cor}
\begin{proof}
The telescope conjecture (for a given pre-triangulated category $\bC$)  asserts that every smashing subcategory is determined by its restriction to the category of compact objects, which is precisely the claim. 
\end{proof}

At this point, one might wonder whether it is possible to characterise the condition of cover for the homotopy Zariski topology in terms of the localization functors of the categories of modules. 
Unfortunately, such a characterisation seems precluded, as Example \ref{exa:Kanda_counterexample}~shows. Nevertheless, it could still be the case that such a characterisation holds for the fine Zariski topology, as stated in the following conjecture.

\begin{conj} \label{conj:faithfulness_conjecture}
    If a family of homotopical epimorphisms $\{ f_i: A \to B_i \}_{i \in I}$ is a cover for the fine Zariski topology then the family of restriction functors
\[ \{ \L f_i^*: \bH \bMod_A \to \bH \bMod_{B_i} \}_{i \in I} \]
is faithful.
\end{conj}

Evidence for this conjecture is supplied by the examples we can explicitly work out. In the case where $A$ is commutative Noetherian ring we can prove that conjecture is true.

\begin{thm} \label{thm:commutative_faithful_conjecture}
Let $A \in \bH\bC\bRings_\Z$, and let $A \to B$ and $A \to C$ be homotopical epimorphisms. Then we have a canonical isomorphism
\[ B \otimes_A^\L C \cong B \ast_A^\L C. \]
\end{thm}
\begin{proof}
We observe that since $A$ is commutative, then both $B$ and $C$ are commutative.
It is enough to prove that
\[ B \otimes_A C \cong B \ast_A C \]
when both $B$ and $C$ are cofibrant as $A$-algebras. In this case, the structure maps $A \to B$ and $A \to C$ are split injections of the underlying $A$-modules, and therefore we can write
\[ B = A \oplus \ol{B}, \qquad C = A \oplus \ol{C}, \]
where $\ol{B}$ and $\ol{C}$ are the cokernels of the structure maps. By \cite[text after Definition 3.1 and Lemma 4.3]{Laz2} we have that $B \otimes_A \ol{B} = 0$, and by \cite[Lemma 2.1]{Laz2} there exists a filtration
\[ B \ast_A C = \bigcup_{n \in \N} F_n \]
with $F_0 = 0, F_1 = B$, and
\[ F_{2n+1}/F_{2n} \cong B \otimes_A (\ol{C} \otimes_A \ol{B})^{\otimes_A^n}, \]
\[ F_{2n+2}/F_{2n+1} \cong B \otimes_A (\ol{C} \otimes_A \ol{B})^{\otimes_A^n} \otimes_A \ol{C}. \]
So, since $A$, $B$ and $C$ are commutative we have that $\ol{C} \otimes_A \ol{B} \cong \ol{B} \otimes_A \ol{C}$. We thus get 
\[ F_{n+1}/F_n \cong 0 \]
for all $n > 1$. Finally, we get
\[  B \ast_A C = B \oplus B \otimes_A \ol{C} \cong B \otimes_A A \oplus B \otimes_A \ol{C} \cong B \otimes_A (A \oplus \ol{C}) \cong B \otimes_A C\,, \qedhere\]
which concludes the proof.
\end{proof}

\begin{cor}
\label{cor:commutative_faithful_conjecture}
    Let $A$ be a commutative Noetherian ring. Then a family of homotopical epimorphisms $\{ f_i: A \to B_i \}_{i \in I}$ is a cover for the fine Zariski topology if and only if the family of restriction functors
\[ \{ \L f_i^*: \bH \bMod_A \to \bH \bMod_{B_i} \}_{i \in I} \]
is faithful. For a general commutative ring, the same holds when the homotopical epimorphisms are of finite presentation. 
\end{cor}
\begin{proof}
    By the main results of \cite{AMSTV}, we know that homotopical epimorphisms of $A$ correspond to flat epimorphisms of rings (see also~Section \ref{exa:spectra} for more about this). By \cite{NB}, we know that $ f_i: A \to B_i$ identify $\SpecG(B_i)$ with a generalisation closed subset of $\SpecG(A)$. And a family $\{ f_i: A \to B_i \}_{i \in I}$ is faithful if and only if it is surjective on $\SpecG(A)$. Furthermore, the latter is equivalent to ask that for any other homotopical epimorphism $A \to C$ we have
    \[ \SpecG(C) \times_{\Spec(A)} \SpecG(B_i) \cong 0 \]
    for all $i \in I$ if and only if $C = 0$. But
    \[ \SpecG(C) \times_{\Spec(A)} \SpecG(B_i) \cong \SpecG(C \otimes_A B_i)\,. \]
    Since $B_i$ and $C$ are flat over $A$ we have
    \[ C \otimes_A B_i \cong C \otimes_A^\L B_i \]
    and by Theorem \ref{thm:commutative_faithful_conjecture} we have
    \[ C \otimes_A^\L B_i \cong C \ast_A^\L B_i, \]
    which yields the first claim. The second claim follows from the fact that homotopical epimorphisms of finite presentation are precisely open embeddings for the classical Zariski topology, and again faithfulness is equivalent to being surjective on the spectrum of prime ideals. 
\end{proof}

\subsection{The relative noncommutative spectrum}\label{sec:relative spectrum}

Motivated by the previous section, a natural question arises:
\textit{what happens if we use the homotopy Zariski topology\footnote{For the sake of clarity, let us recall that the homotopy Zariski topology differs from the formal homotopy Zariski topology by requiring that the family of open embeddings is formed by homotopical epimorphisms of finite presentation in the category $\bH\bRings_{R}$. } instead of the formal homotopy Zariski topology?}

From the point of view of formal logic, nothing prevents us from using the homotopy Zariski topology, and the resulting theory is still well posed. Let us denote by $\SpecNC_{\rm f.p.}$ the spectrum functor one would obtained by pursuing this alternative strategy, where ${\rm f.p.}$ stands for finite presentation. As it turns out,  $\SpecNC_{\rm f.p.}$ is badly incompatible with classical algebraic geometry in general. Indeed, Example \ref{exa:non_finitely_generated_free_algebras} shows that, for any commutative ring $R$ that is not finitely generated over $\Z$, any non-trivial classical Zariski localization $R \to S$ is a homotopy epimorphism that is not of finite presentation. So, there is no natural way of comparing $\SpecNC_{\rm f.p.}(R)$ with the Grothendieck spectrum $\SpecG(R)$, because the open localizations that form the basis of the topology of $\SpecG(R)$ are not in $\SpecNC_\fin(R)$. Therefore, nothing like the morphism $\iota_A$ of Proposition~\ref{prop:compatision with Grothendieck} exists in this case.
Let us make this observation more precise.

\begin{prop} \label{prop:wrong_spectrum}
Let $R$ be a commutative ring that is not finitely generated over $\Z$. Then
\[ \SpecNC_{\rm f.p.}(R) = \star \]
is the singleton topological space.
\end{prop}
\begin{proof}
Let  $ f: R \to S$
be a homotopy epimorphism whose domain is $R$ and which is of finite presentation in $\bH\bRings_\Z$. Then, since $R$ is commutative, so is $S$ (see Example \ref{exa:homotopy_epimorphism} (iii)), and since $R$ is discrete, also $S$ must be such (because being of finite presentation in the noncommutative sense implies being of finite presentation in the commutative sense, but not the converse). Therefore, $f$ is a classical Zariski localization and, by Example~\ref{exa:non_finitely_generated_free_algebras}, $f$ is not finitely presented (in the noncommutative sense) unless it is trivial. Hence, $R \cong S$ and $f$ is an isomorphism.
\end{proof}

The above observations are in perfect agreement with the no-go theorem of Reyes~\cite{nogo}: it is not possible to define a functor from the category $\bRings$ to the category $\mathbf{Sets}$ such that the spectrum of every
nonzero ring is nonempty and that extends the Grothendieck spectrum functor.
The key point in the proof of Proposition~\ref{prop:wrong_spectrum} is that the notion of finite presentation in the category of commutative $R$-algebras is not compatible with the notion of finite presentation in noncommutative setting.
This notwithstanding, if one restricts oneself to finitely generated (central) algebras over a ring $R$, matters improve. Let us explain why this is the case. 

As we observed in Section \ref{Categories of algebras}, to a given commutative ring $R$ one can associate two different categories of algebras under it: the category $\bRings_R$ of all $R$-algebras and the category $\bZ\bRings_R$ of central algebras over $R$. The difference between the two categories can be understood very clearly in terms of algebraic theories: the category $\bRings_R$ is the algebraic category generated by the class of algebras
\[ \{ \Z \lt x_1, \ldots, x_n \gt \ast_\Z R \}_{n \in \N}, \]
whereas $\bZ \bRings_R$ is the algebraic category generated by the class of algebras
\[ \{ R \lt x_1, \ldots, x_n \gt \}_{n \in \N}. \]
So, in analogy with the homotopy theory of $\bRings_R$ presented so far,  one can develop a homotopy theory of $\bZ\bRings_R$ without additional obstacles. Moreover, the two theories are compatible; by this we mean that, since 
\[ R \lt x_1, \ldots, x_n \gt \ast_R R \lt y_1, \ldots, y_m \gt = R \lt x_1, \ldots, x_n, y_1, \ldots, y_m \gt \]
for all $n$ and $m$, the category $\bZ\bRings_R$ is closed by pushouts in $\bRings_R$ and similarly $\bH \bZ\bRings_R$ is closed by homotopical pushouts in $\bH\bRings_R$. Hence, a morphism $\bH \bZ\bRings_R$ is a homotopical epimorphism if and only if it is such when regarded as a morphism in $\bH\bRings_R$. However,  the notion of finite presentation for objects in $\bH \bZ\bRings_R$ and in $\bH\bRings_R$ are very different --- see Proposition~\ref{prop:finitely_presented_objects_central_algebras}.

With the above discussion in mind, we want to define a notion of the noncommutative spectrum \emph{relative} to a fixed base commutative ring $R$. This functor will be constructed arguing as in Section~\ref{sec:spectrum}, but this time we will use the homotopy Zariski topology, instead of the formal homotopy Zariski topology. This topology differs from the formal homotopy Zariski topology by requiring that the open embeddings are deterimined by homotopical epimorphism of finite presentation in $\bH\bZ\bRings_R$. In this way, it is singled out a subposet $\bLoc_\fin(A) \subset \bLoc(A)$ and a subposet $\bOuv_\fin(X) \subset \bOuv(X)$, for any $A \in \bH\bZ\bRings_R$ and $X = \Spec(A)$. Again, to $\bOuv_\fin(X)$ we can associate a site $\bZar_{X, \fin}$ that for the same reasons as before has enough points and is equivalent to a site of a sober topological space, that we denote $\bZar_{X, \fin}$.

\begin{defn} \label{defn:noncommutative_spectrum}
Let $R$ be a commutative ring and let $A \in \bH\bZ\bRings_R$. We define the \textit{relative noncommutative spectrum} of $A$
to be the topological space 
$\SpecNC_R(A)$ associated to the site $\bZar_{X,\fin}$. We define the fine version of the relative spectrum, denoted by $\SpecNC_{R, \fine}(A)$, analogously.
\end{defn}

\begin{rmk} \label{rmk:relative_fine_Zariski}
Once again the fine Zariski topology is strictly finer than the Zariski topology (but not functorial in general), and we obtain a canonical continuous map
\[ \SpecNC_{R, \fine}(A) \to \SpecNC_R(A). \]
\end{rmk}

It is clear from the above discussion that the topological space $\SpecNC_R(A)$ does not only depend on the choice of $A$ but crucially also on the choice of the base commutative ring $R$.

\begin{prop}\label{prop:relative spectrum}
Given a finitely presented commutative $R$-algebra $A$, we have a canonical homeomorphism 
\[ \SpecNC_{R, \fine}(A) \cong \SpecG(A) \]
of topological spaces.
\end{prop}
\begin{proof}
 To establish that the relative noncommutative fine spectrum agrees with the Grothendieck spectrum, we observe that
\[ R[x_1, \ldots, x_n] \cong \frac{R \lt x_1, \ldots, x_n \gt}{(x_1 x_2 - x_2 x_1, \ldots, x_{n-1} x_n - x_n x_{n-1})} \]
is clearly a finitely presented $R$-algebra in $\bH\bZ\bRings_R$. Thus, all localizations of the classical Zariski spectrum are localizations in $\bH \bZ\bRings_R$. Finally, arguing as in Proposition~\ref{prop:wrong_spectrum}, one concludes that these are the only possible localizations. By Theorem \ref{thm:commutative_faithful_conjecture} (and argiung as in~Corollary \ref{cor:commutative_faithful_conjecture}) the notions of cover for the two topologies coincide.
\end{proof}

The relative spectrum could be useful to understand the ``local" geometry more easily compared to the ``global" perspective adopted in this work. The next example clarifies this heuristic remark.

\begin{exa}
Let $R = \C$, which is not finitely generated over $\Z$. For any finitely generated $\C$-algebra $A$ we have $$\SpecNC_\Z(A) = \star$$ but 
$$\SpecNC_{\C, \fine}(A) \cong \SpecG(A).$$
\end{exa}

\iffalse
  \fi

\subsection{A few properties of the noncommutative spectrum}

We conclude this section by discussing a few  basic properties of the new notion of spectrum defined above.

\begin{prop} \label{prop:spectrum_properties}
The following properties hold.
\begin{enumerate}
    \item If $A \cong A_1 \times A_2$  in $\bH\bRings_\Z$, then $\SpecNC(A) \not\cong \SpecNC(A_1) \coprod \SpecNC(A_2)$, but $\SpecNC_\fine(A) \cong \SpecNC_\fine(A_1) \coprod \SpecNC_\fine(A_2)$ as topological spaces.
    \item If $A \in \bRings_\Z$ is a domain\footnote{It is customary to call rings without zero divisors \emph{domains} and commutative rings without zero divisors \emph{integral domains}.}, then $\SpecNC_\fine(A)$ is connected.
    \item More generally, if we have two morphisms $\SpecNC(B) \to \SpecNC(A)$ and $\SpecNC(C) \to \SpecNC(A)$, then the fiber product exists and is computed as $\SpecNC(B \ast_A^\L C)$, which usually does not agree with $\SpecNC(B) \times_{\SpecNC(A)} \SpecNC(C)$ as topological space. 
    \item If for $A, B \in \bH \bRings_\Z$ are derived Morita equivalent (\ie $\bH \bMod_A \cong \bH \bMod_B$), then $\SpecNC_\fine(A) \cong \SpecNC_\fine(B)$. 
\end{enumerate}
\end{prop}
\begin{proof}
\begin{enumerate}
    \item The split cofiber sequence
    \[ \bH \bMod_{A_1} \to \bH \bMod_{A} \to \bH \bMod_{A_2}, \]
     identifies $\bH \bMod_{A_1}$ with the kernel of $\bH \bMod_{A_2}$, and a similar argument gives the opposite inclusion. This shows that $\bH \bMod_{A} \cong \bH \bMod_{A_1} \oplus \bH \bMod_{A_2}$.
     Hence, the projections $A \to A_1$ and $A \to A_2$ are both homotopical epimorphisms and kernels of homotopical epimorphisms. Their intersection is clearly empty. So, they determine subsets of $\SpecNC_\fine(A)$ that are both closed and open, and form a cover for the fine Zariski topology. Example \ref{exa:Kanda_counterexample} demonstrates that the two projections might not be covers for the homotopy Zariski topology.
    \item Similarly to the previous item, a closed and open subset of $\SpecNC_\fine(A)$ is uniquely determined by a smashing subcategory that is also a smashing localization. Moreover, this determines a decomposition of $A$ as a direct product of two subrings. Such a decomposition cannot exist if $A$ is a domain.
    \item Since $\bdAff_\Z$ is by definition the opposite category of $\bH\bRings_\Z$, fiber products in $\bdAff_\Z$ correspond to pushouts in $\bH\bRings_\Z$. The latter are given by the (derived) free product. It is not hard to come up with examples for which $\SpecNC(A \ast_\Z^\L B) \not\cong \SpecNC(A) \times \SpecNC(B)$ as the isomorphism rarely holds, already in classical algebraic geometry. Examples of a similar type can be devised for noncommutative spectra as well.
    \item By Proposition \ref{prop:iso_smashing_poset} the lattice of localizations of $A \in \bH \bRings_\Z$ depends only on $\bH \bMod_A$. The notion of cover for the fine Zariski topology also depends on the lattice of localizations only, because for two homotopical epimoprhisms $A \to B$, $A \to C$, one has $\bH \bMod_{B \ast_A^\L C} = \bH \bMod_B \cap \bH \bMod_C$ where the intersection happen inside $\bH \bMod_A$.
\end{enumerate}
\end{proof}

Further analysis is required to understand more properties of the functor $\SpecNC$ and the spectrum $\SpecNC_\fine$. For instance, it is natural to wonder whether given $A, B \in \bRings_\Z$ are that are derived Morita equivalent, then one has $\SpecNC(A) \cong \SpecNC(B)$. We do not know whether this is true in general, because even if the localizations of $A$ and $B$ depend only of the derived category, the notion of cover, in principle, does not.

%We would like also to comment on the classical scenario of affine commutative algebraic geometry, namely, that of a commutative ring $A$ and a finitely presented ideal $I \subset A$. This data determines a quotient ring $A/I$ and a corresponding geometric morphism 
%\[ \SpecG(A/I) \to \SpecG(A) \]
%which is a closed immersion of topological spaces. We expect the noncommutative spectrum to enjoy a similar property; we postpone a detailed examination of this matter until future works.

The analysis performed in this paper lays the foundations of the study of noncommutative geometry over $\bH\bRings_\Z$. So far, only affine objects have been considered; this is fine until one wishes to construct new spaces by gluing spaces together. In classical algebraic geometry this is done by computing colimits in the category of sheaves for the Zariski topology. And the natural approach would be to do the same in the noncommutative setting as well. This is not possible because representable pre-sheaves associated to noncommutative rings are not sheaves for the homotopical Zariski topology. But this has not to be considered as a problem of our theory, it is actually an interesting feature whose study we start in Section \ref{sec:gelfand}. We will see that noncommutative spaces manifest a local-to-global behaviour that is not fully captured by the notion of sheaf although they are objects that can be descended from covers of the underlying topological space.

\section{Examples of noncommutative spectra}\label{exa:spectra}

We now discuss some explicit examples of noncommutative spectra. We will examine relatively simple cases, as an illustration of the theory developed in this paper. A more thorough examination of the consequences of our new notion of noncommutative spectrum goes beyond the scope of the present paper and will be carried out in upcoming publications.
Let us emphasise that the computation of the spectrum requires the classification of smashing subcategories of $\bH\bMod_A$, which is usually not an easy task and an open problem for most classes of rings.
We plan to discuss further examples and provide additional details about the computations in future publications.

\subsection{Commutative rings}

Before moving on to noncommutative rings, we give a few key examples of noncommutative spectra of commutative rings. These can be compared with the classical notion of spectrum introduced by Grothendieck, as per Proposition \ref{prop:compatision with Grothendieck}; we will now see that they differ, except for very specific cases.

\paragraph{Commutative fields.} Let $k$ be a commutative field $k$. The class of smashing localizations of $\bH\bMod_k$ only contains trivial objects; hence, $\SpecNC(k) \cong \SpecNC_\fine(k) \cong \SpecG(k)$ and its underlying space is just a point.

 \paragraph{Discrete valuation rings.} This is the case, for example, of the ring of formal power series $R=k[\![x]\!]$ in one variable $x$ over some field $k$, or of the ring of p-adic integers $\Z_p$ for any prime $p$. 
The class of homotopy epimorphisms under $R$, by the classification results from~\cite[Theorem 1.1 (2)]{AMSTV}, contains only three elements: the identity $R \to R$, the localization $R \to {\rm Frac}(R)$, and $R \to 0$. All localizations are of finite presentation (in the commutative sense), so also in this case $\SpecNC(R) = \SpecG(R)$, and its underlying topological space is the usual Sierpinski space. Also in this case we have $\SpecNC(R) \cong \SpecNC_\fine(R)$.

\paragraph{The ring $k^n$.} Let $k$ be a field. Consider the direct product $k^n = k \times \cdots \times k$. Then, by generalising Example~\ref{exa:Kanda_counterexample}, one computes $\SpecNC(k)$ to be the boolean algebra of subsets of the set with $n$ elements equipped with the trivial topology. Therefore, $\SpecNC(k)$ has $2^n - n - 1$ specialization points and $n$ generic points. 
On the other hand, $\SpecNC_\fine(k)$ is the set of $n$ points with the discrete topology.

\paragraph{Noetherian commutative rings.}  This is the case, for example, of the ring of formal power series $R = N[\![x]\!]$ over a Noetherian ring $N$. The smashing localizations of $\bH\bMod_R$ have been classified in \cite[Theorem 1.1 (2)]{AMSTV}. These correspond precisely to flat ring epimorphisms $R \to S$. Thus, for such an $R$ the category $\bLoc_R$ is just the class of all flat ring epimorphisms whose domain is $R$. Note that in this case $S$ is necessarily commutative (see \cite[Corollary 1.2]{Sil}). Moreover, flat ring epimorphisms $R \to S$ determine generalization 
closed subsets of $\SpecG(R)$, as these are arbitrary intersections of open subsets (see \cite[Lemma 5.19.3]{Stack}). Therefore, $\SpecNC(R)$ is obtained from $\SpecG(R)$ by enlarging the family of compact open subset to encompass the family of all generalization closed subsets that are images of flat ring epimorphisms. Then, the cover condition for $\SpecNC(R)$ is difficult to check and not compatible with the geometry of $\SpecG(R)$ in general, as Example~\ref{exa:Kanda_counterexample} shows. Instead, Theorem \ref{thm:commutative_faithful_conjecture} shows that the cover condition for the topology of $\SpecNC_\fine(R)$ is that the family of generalization closed subsets covers $\SpecG(R)$ set-theoretically.

Let us focus on the special case $R = \Z$. The classification of smashing localization of $\bH\bMod_\Z$ performed in~\cite[Theorem 1.1 (2)]{AMSTV} implies that, in the case at hand, smashing localizations correspond to morphisms $\Z \to \Z[S^{-1}]$, where $S^{-1}$ is a subset of primes of $\Z$. Therefore, in this case the category $\bZar_{\Spec(\Z)}$ is a distributive lattice, where 
\[ \Spec(\Z[S^{-1}]) \wedge \Spec(\Z[T^{-1}]) \cong \Spec(\Z[S^{-1}] \otimes_\Z \Z[T^{-1}]) \cong  \Spec(\Z[(S \cup T)^{-1}]) \] 
and 
\[ \Spec(\Z[S^{-1}]) \lor \Spec(\Z[T^{-1}]) \cong \Spec(\Z[(S \cap T)^{-1}]). \] 
Now, upon establishing a correspondence between $\N$ and the set of prime numbers, we can view objects of $\bZar_{\Spec(\Z)}$ as subsets of $\N$, equipped with the order relation given by the opposite of inclusion and the lattice operations given by finite intersections and arbitrary unions. Therefore, the frame $\Id(\bZar_{\Spec(\Z)})$ is the frame of all ideals of the lattice of subsets of $\N$ plus the extra element corresponding to the trivial localization $\Z \to 0$. 
The topology of $\SpecNC(\Z)$ can be described as follows: a subset $C \subset \SpecNC(\Z)$ is closed if there exists $S \in \sP(\N)$ such that
\[ C = \{x \in \SpecNC(\Z) \ | \  S \in x \}. \]
In this particular case we have agreement between the commutative and noncommutative notions of cover because one can prove the following canonical isomorphism: $\Z \lt X \gt \ast_\Z^\L \Z [S^{-1}] \cong \Z \lt X \gt \otimes_\Z^\L \Z [S^{-1}]$ for all sets of variables $X$ and all subset $S$ of the prime numbers. This implies that for $A \in \bH \bRings_\Z$ one has $A \ast_\Z^\L \Z [S^{-1}] \cong 0$ if and only if $A \otimes_\Z^\L \Z [S^{-1}] \cong 0$, and hence $\SpecNC(\Z) \cong \SpecNC_\fine(\Z)$ because the cover conditions in the two situation happen to agree. 
This means that, as a set, $\SpecNC(\Z)$ corresponds to the set of ultrafilters of $\N$, plus an extra generic point.
This is a Zariski-like topology and it is not Hausdorff. An easy way to see that it is not Hausdorff is by observing that $\SpecNC(\Z_{(p)})$ is not Hausdorff, and the latter embeds into $\SpecNC(\Z)$ via the open localization $\Z \to \Z_{(p)}$.

We think that the given homeomorphism $\SpecNC(\Z) \cong \SpecNC_\fine(\Z)$ is very specific to $\Z$; in fact, we conjecture that this already fails for the affine line $\Z[x]$.  We postpone a deeper study of these example until future work.

\paragraph{Non-Noetherian commutative rings.}
Let $R$ be the non-Noetherian valuation domain considered in \cite[Example 5.24]{BazSto}. If $\mathfrak{m} \subset R$ denotes the maximal ideal of $R$ then $\mathfrak{m}^2 = \mathfrak{m}$ and the morphism $R \to \frac{R}{\mathfrak{m}}$ is a non-flat homotopy epimorphism. In \cite[Example 5.24]{BazSto} it is explained that $R \to \frac{R}{\mathfrak{m}}$, $R \to \rm{Frac}(R)$ and $R \to \frac{R}{\mathfrak{m}} \times \rm{Frac}(R)$ are the only non-trivial homotopical epimorphisms of $R$. If we denote $k = \frac{R}{\mathfrak{m}}$ and $Q = \rm{Frac}(R)$ we get the distributive lattice of homotopical epimorphisms
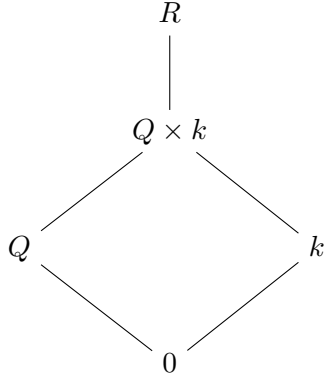
\[
\begin{tikzpicture}
\matrix(m)[matrix of math nodes,
row sep=2.6em, column sep=2.8em,
text height=1.5ex, text depth=0.25ex]
{ 
& R  \\
& Q \times k  \\
Q & & k  \\
& 0 \\};
\path[-, font=\scriptsize]
(m-1-2) edge node[auto] {} (m-2-2);\\
\path[-, font=\scriptsize]
(m-2-2) edge node[auto] {} (m-3-1);
\path[-, font=\scriptsize]
(m-2-2) edge node[auto] {} (m-3-3);
\path[-,font=\scriptsize]
(m-3-1) edge node[auto] {} (m-4-2);
\path[-,font=\scriptsize]
(m-3-3) edge node[auto] {} (m-4-2);
\end{tikzpicture}.
\]
The associated sober topological space for the fine Zariski topology has two generic points $\eta_1, \eta_2$ that share a common closed special points $s$, as per the following picture:
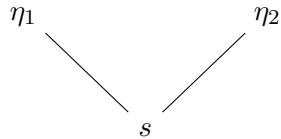
\[
\begin{tikzpicture}
\matrix(m)[matrix of math nodes,
row sep=2.6em, column sep=2.8em,
text height=1.5ex, text depth=0.25ex]
{
\eta_1 & & \eta_2 \\
 & s &  \\};
\path[-, font=\scriptsize]
(m-1-1) edge node[auto] {} (m-2-2);
\path[-, font=\scriptsize]
(m-1-3) edge node[auto] {} (m-2-2);

\end{tikzpicture}.
\]
In this representation, the open subsets of $\SpecNC_\fine(R) = \{ \eta_1, \eta_2, s\}$ are $\void$, $\{\eta_1 \}$, $\{\eta_2 \}$, $\{\eta_1, \eta_2 \}$, $\{\eta_1, \eta_2, s \}$. Thus, $s$ is a closed but not open point. Let us observe that the closure of $\{\eta_1 \}$ is the irreducible closed subset $\{\eta_1, s\}$ and the closure of $\{\eta_2 \}$ is the irreducible closed subset $\{\eta_2, s \}$. Therefore $\SpecNC_\fine(R)$ is not an irreducible topological space even if $R$ is an integral domain, unlike what happens in scheme theory. Finally, the projection 
\[ \SpecNC_\fine(R) \to \SpecG(R) \]
can be interpreted as the projection onto one of the closed irreducible subspaces $\{ \eta_1, s\}$ sending $\eta_2$ to $\eta_1$. It is straightforward to check that this map is continuous.

\subsection{Noncommutative rings}
\label{sec:noncommutative_spectra}

The noncommutative examples are more challenging to work out, and only few are currently known to us.

\paragraph{The ring $M_n(k)$ of square matrices over a field $k$.} 
 The first example is a trivial one. Since $M_n(k)$ is derived Morita equivalent to $k$, and by Proposition~\ref{prop:spectrum_properties} the noncommutative spectrum depends only on the derived category of left (or right) modules, we have that $\SpecNC(M_n(k)) \cong \SpecNC_\fine(M_n(k)) \cong \SpecNC(k)$ is a one point topological space. 

\paragraph{Division rings.}  Let $k$ be a division algebra. It is well known that the category $\bMod_k$ is formally identical to the category of vector spaces over a field: all objects are free, and hence projective. This immediately implies that $\bH \bMod_k$ has only trivial smashing localizations, like in the case of fields. So, $\SpecNC(k)$ is a singleton and, by the Morita invariance of the spectrum, so is $\SpecNC(M_n(k))$.

\paragraph{Path algebra for the quiver $A_n$.} Let us consider the path algebra for the quiver $A_n$ (over a field $k$) and denote it by $k A_n$. It is known that $k A_n$ has only a finite number of smashing localizations. In this example, we focus on the special case $k A_2$; the analysis for $k A_n$, $n>2$, is similar. 

We refer the reader to~\cite[Example 5.1.4]{GS} for the classification of the smashing subcategories of $k A_2$, from which we borrow the notation. Smashing subcategories of $k A_2$ form the lattice
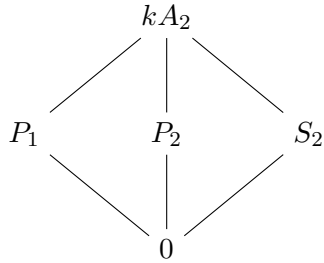
\[
\begin{tikzpicture}
\matrix(m)[matrix of math nodes,
row sep=2.6em, column sep=2.8em,
text height=1.5ex, text depth=0.25ex]
{ & k A_2  \\
P_1 & P_2 & S_2 \\
& 0 \\};
\path[-, font=\scriptsize]
(m-1-2) edge node[auto] {} (m-2-1);\\
\path[-, font=\scriptsize]
(m-1-2) edge node[auto] {} (m-2-2);
\path[-, font=\scriptsize]
(m-1-2) edge node[auto] {} (m-2-3);
\path[-,font=\scriptsize]
(m-2-1) edge node[auto] {} (m-3-2);
\path[-,font=\scriptsize]
(m-2-2) edge node[auto] {} (m-3-2);
\path[-,font=\scriptsize]
(m-2-3) edge node[auto] {} (m-3-2);
\end{tikzpicture}
\]
which is well-known to be non-distributive. 
One can show that that only trivial covers are covers for the homotopy Zariski topology. We refer to Example~\ref{exa:morphism_kronecker_to_A2quiver} for a proof of this fact. Now, the only covers for the homotopy Zariski topology in this case are $\{\Id_0\}$, $\{\Id_{P_1}\}$, $\{\Id_{P_2}\}$, $\{\Id_{S_2}\}$, and $\{\Id_{k A_2}\}$. Hence the list of all ideals for this topology is
\[ 0 := \{ 0 \}, \ 01 := \{0, P_1\}, \ 02 := \{0, P_2 \}, \ 03 := \{0, S_2\}, \]
\[ \ 012 := \{0, P_1, P_2\}, \ 013 := \{0, P_1, S_2 \}, \ 023 := \{0, P_2, S_2 \}   \]
\[ \ 0123 = \{0, P_1, P_2, S_2 \}, \ 01234 := \{0, P_1, P_2, S_2, k A_2 \}.   \]
The resulting frame is the following:
\[
\begin{tikzpicture}
\matrix(m)[matrix of math nodes,
row sep=2.6em, column sep=2.8em,
text height=1.5ex, text depth=0.25ex]
{ 
    &  01234  &   \\
    &  0123  &   \\
012 & 013 & 023  \\
01 & 02 & 03 \\
& 0 \\};
\path[-, font=\scriptsize]
(m-1-2) edge node[auto] {} (m-2-2);
\path[-, font=\scriptsize]
(m-2-2) edge node[auto] {} (m-3-1);
\path[-, font=\scriptsize]
(m-2-2) edge node[auto] {} (m-3-2);
\path[-, font=\scriptsize]
(m-2-2) edge node[auto] {} (m-3-3);
\path[-,font=\scriptsize]
(m-3-1) edge node[auto] {} (m-4-2);
\path[-,font=\scriptsize]
(m-3-2) edge node[auto] {} (m-4-2);
\path[-,font=\scriptsize]
(m-3-3) edge node[auto] {} (m-4-2);
\path[-,font=\scriptsize]
(m-3-1) edge node[auto] {} (m-4-1);
\path[-,font=\scriptsize]
(m-3-3) edge node[auto] {} (m-4-3);
\path[-,font=\scriptsize]
(m-3-2) edge node[auto] {} (m-4-1);
\path[-,font=\scriptsize]
(m-3-2) edge node[auto] {} (m-4-3);
\path[-,font=\scriptsize]
(m-4-1) edge node[auto] {} (m-5-2);
\path[-,font=\scriptsize]
(m-4-2) edge node[auto] {} (m-5-2);
\path[-,font=\scriptsize]
(m-4-3) edge node[auto] {} (m-5-2);
\end{tikzpicture}.
\]
It is then easy to check that $\SpecNC(k A_2)$ is the spectral space
\[
\begin{tikzpicture}
\matrix(m)[matrix of math nodes,
row sep=2.6em, column sep=2.8em,
text height=1.5ex, text depth=0.25ex]
{
\eta_1 & \eta_2 & \eta_3 \\
 & s &  \\};
\path[-, font=\scriptsize]
(m-1-1) edge node[auto] {} (m-2-2);
\path[-, font=\scriptsize]
(m-1-2) edge node[auto] {} (m-2-2);
\path[-, font=\scriptsize]
(m-1-3) edge node[auto] {} (m-2-2);
\end{tikzpicture}
\]
with three generic points that share a common specialization point. 

In order to describe $\SpecNC_\fine(k A_2)$, we observe that the family $\{k A_2 \to A_{P_1}, k A_2 \to A_{P_2}, k A_2 \to A_{S_2}\}$ is a non-trivial faithful family of functors. The presence of this extra cover has the effect of removing $0123$ from the list of ideals of the posite. 
It is not hard to check that the topological space associated with this latter frame is the discrete topological space with three points. More generally, we expect $\SpecNC_\fine(k A_n)$ to be a discrete space whose points correspond to indecomposable finite dimensional representations of $A_n$, and $\SpecNC(k A_n)$ to be the spectral topological space associated to the (finite) lattice of smashing localizations of $k A_n$ equipped with the trivial posite structure. Similar considerations can be made for quiver algebras of finite representation type.

\paragraph{Kronecker algebra.}
Let us consider the Kronecker algebra $A$ over a field $k$. This algebra can be defined as the path algebra of the (Kronecker) quiver $\star \rightrightarrows \star$, and is isomorphic to the matrix algebra
\[ \begin{pmatrix}
k & k^2 \\
0 & k
\end{pmatrix}\,, \]
that is, a four dimensional noncommutative algebra over the field $k$. Owing to the work of Beilinson \cite{Bei}, we know that there is an explicit equivalence of categories
\[ D(\bQ\bCoh(\P_k^1)) \cong D(\bMod_A)\,, \]
where on the left-hand side we have the derived category of quasi-coherent sheaves on $\P_k^1$. 
The smashing subcategories of $D(\bMod_A)$ have been classified in~\cite{KS}. Since $A$ is hereditary, the smashing subcategories of $D(\bMod_A)$ correspond to those of $\bH\bMod_A$ (because it is known that in this case homotopical epimorphisms are always in degree $0$, \ie given by ring epimorphisms, see \cite{BazSto}). So, by \cite{KS} we obtain the poset 
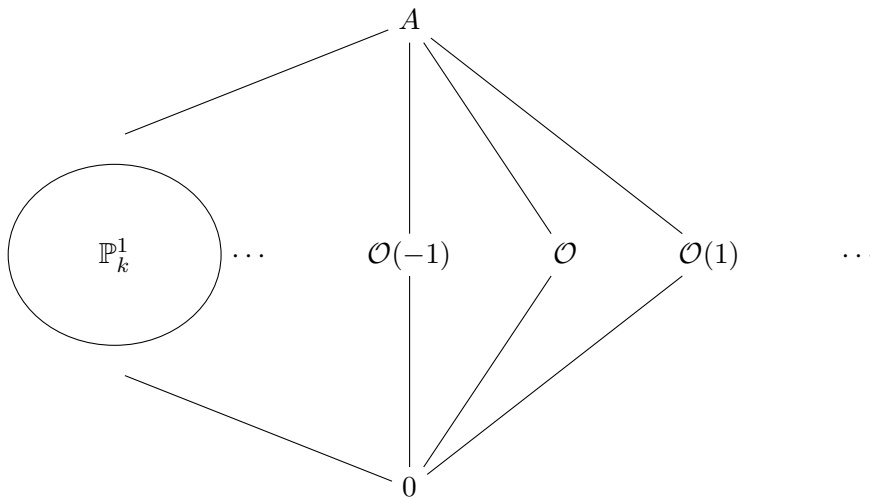
\[
\begin{tikzpicture}
\matrix(m)[matrix of math nodes,
row sep=2.6em, column sep=2.8em,
text height=1.5ex, text depth=0.25ex]
{ & & A  \\
 {} &  \\
\P_k^1 & \cdots & \cO(-1) & \cO & \cO(1) & \cdots \\
 {}& \\
& & 0 \\};
\path[-, font=\scriptsize]
(m-1-3) edge node[auto] {} (m-2-1);\\
\path[-, font=\scriptsize]
(m-1-3) edge node[auto] {} (m-3-3);
\path[-, font=\scriptsize]
(m-1-3) edge node[auto] {} (m-3-4);
\path[-, font=\scriptsize]
(m-1-3) edge node[auto] {} (m-3-5);
\path[-,font=\scriptsize]
(m-4-1) edge node[auto] {} (m-5-3);
\path[-,font=\scriptsize]
(m-3-3) edge node[auto] {} (m-5-3);
\path[-,font=\scriptsize]
(m-3-4) edge node[auto] {} (m-5-3);
\path[-,font=\scriptsize]
(m-3-5) edge node[auto] {} (m-5-3);
\draw (m-3-1) circle [x radius=1.4, y radius=1.2];
\end{tikzpicture}
\]
which is not distributive, and where on the left part of the diagram we have depicted the lattice of localizations corresponding to specialisation closed subsets of $\P_k^1$. The non-distributivity is due to the presence of sub-lattices of the form
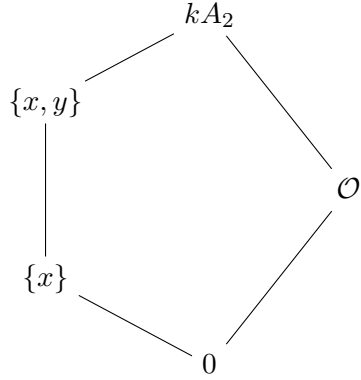
\[
\begin{tikzpicture}
\matrix(m)[matrix of math nodes,
row sep=1.6em, column sep=2.8em,
text height=1.5ex, text depth=0.25ex]
{ & k A_2  \\
\{x , y \}  \\
&   & \cO & \\
\{x \}  \\
& 0 \\};
\path[-, font=\scriptsize]
(m-1-2) edge node[auto] {} (m-2-1);
\path[-, font=\scriptsize]
(m-1-2) edge node[auto] {} (m-3-3);
\path[-, font=\scriptsize]
(m-2-1) edge node[auto] {} (m-4-1);
\path[-,font=\scriptsize]
(m-4-1) edge node[auto] {} (m-5-2);
\path[-,font=\scriptsize]
(m-3-3) edge node[auto] {} (m-5-2);
\end{tikzpicture}.
\]
The topological space $\SpecNC(A)$ can be described analogously to $\SpecNC(k A_2)$ above, so we omit the details. The final outcome of the construction is that there exists infinitely many points corresponding to the twisting sheaves $\cO(n)$ and another component whose points correspond to ultrafilters of rational points of $\P_k^1$ plus a generic point (similarly to what has been described for $\SpecNC(\Z)$). 
Since no union of these sets covers the whole space, these are connected to a common specialization point. We can represent this as 
\[
\begin{tikzpicture}
\matrix(m)[matrix of math nodes,
row sep=2.6em, column sep=2.8em,
text height=1.5ex, text depth=0.25ex]
{ \P_k^1 & \cdots & p_{\cO(n)} & p_{\cO(n + 1)} & \cdots \\
  &  s  \\};
\path[-, font=\scriptsize]
(m-1-1) edge node[auto] {} (m-2-2);
\path[-, font=\scriptsize]
(m-1-3) edge node[auto] {} (m-2-2);
\path[-, font=\scriptsize]
(m-1-4) edge node[auto] {} (m-2-2);
\end{tikzpicture}.
\]
The picture for $\SpecNC_\fine(A)$ is similar, only the topology of the $\P_k^1$ component changes in this case. We postpone a more detailed study of this example until future work.

\subsection{Morphisms of spectra}

Before moving on to discuss our noncommutative version of Gelfand's duality, let us give a couple of relevant examples of continuous maps between noncommutative spectra. 

The first example, stems from~Example \ref{exa:Kanda_counterexample}.

\begin{exa} \label{exa:Kanda_morphism}

As in Example~\ref{exa:Kanda_counterexample}, consider a field $k$ and the diagonal inclusion morphism $\phi: k \oplus k \to M_2(k)$. The latter morphism induces a morphism of spectra $\SpecNC(M_2(k)) \to \SpecNC(k \oplus k)$ that we now describe. We have already seen that 
$M_2(k) \ast_{k \oplus k}^\L k \cong 0$ for both projections $k \oplus k \to k$. Moreover, the isomorphism $M_2(k) \ast_{k \oplus k}^\L k \oplus k \cong M_2(k)$ is obvious. Therefore, the morphism $\phi$ induces a functor
\[ \bOuv(k \oplus k) \to \bOuv(M_2(k)) \]
that is non-trivial outside the generic points of $\SpecNC(k \oplus k)$. More precisely, one can check that this means that the prime filter corresponding to the point of $\SpecNC(M_2(k))$ is sent to the prime filter corresponding to the specialization point of $\SpecNC(k \oplus k)$. 

Let us complete the picture by considering also the inclusion morphism $\SpecNC_\fine(k \oplus k) \to \SpecNC(k \oplus k)$. If we do so, we obtain the following diagram of continuous maps (solid arrows)

\[
\begin{tikzpicture}
\matrix(m)[matrix of math nodes,
row sep=2.6em, column sep=2.8em,
text height=1.5ex, text depth=0.25ex]
{ \SpecNC_\fine(M_2(k)) & \SpecNC_\fine(k \oplus k) \\
 \SpecNC(M_2(k))  & \SpecNC(k \oplus k)  \\};
\path[->, font=\scriptsize]
(m-1-1) edge node[auto] {$\cong$} (m-2-1);
\path[->, font=\scriptsize]
(m-1-2) edge node[auto] {} (m-2-2);
\path[->, font=\scriptsize]
(m-2-1) edge node[auto] {} (m-2-2);
\path[->, dashed, font=\scriptsize]
(m-1-1) edge node[auto] {} (m-1-2);
\end{tikzpicture}.
\]
Here the dashed horizontal map cannot be realised as a map of topological spaces because the image of the map $\SpecNC(M_2(k)) \to \SpecNC(k \oplus k)$ lies outside the image of the map $\SpecNC_\fine(k \oplus k) \to \SpecNC(k \oplus k)$. Therefore, the dashed arrow can only be the empty relation. 

Let us observe that even if there is no continuous map $\SpecNC_\fine(M_2(k)) \to \SpecNC_\fine(k \oplus k)$, we can always consider a sheaf on $\SpecNC_\fine(k \oplus k)$, push it to $\SpecNC(k \oplus k)$, and pull it back to $\SpecNC_\fine(M_2(k))$ --- and the other way around: push forward sheaves from $\SpecNC_\fine(M_2(k))$ to $\SpecNC(k \oplus k)$ and then pullback to $\SpecNC_\fine(k \oplus k)$, thus obtaining a rich enough geometrical setting.

\end{exa}

In the next example we show that the homotopical Zariski topology induces the trivial topology on the poset of smashing localizations of $k A_2$.

\begin{exa} \label{exa:morphism_kronecker_to_A2quiver}
Consider the inclusion morphism  
\[ k A_2 = \begin{pmatrix}
k & 0 \\
k & k
\end{pmatrix} \mapsto 
\begin{pmatrix}
k & 0 \\
(k, 0) & k
\end{pmatrix} \to A. \]
By functoriality, this induces a morphism of spectra
\[ \SpecNC(A) \to \SpecNC(k A_2).  \]
Now, suppose that the topology on $\bOuv(k A_2)$ is non-trivial. This means, borrowing the notation from Example~\ref{exa:Kanda_morphism}, that the family $\{ k A_2 \to A_{P_1}, k A_2 \to A_{P_2}, k A_2 \to A_{S_2} \}$ must be a cover, as it is the only non-trivial faithful family of localization functors. But then the morphism $k A_2 \to A$ must map this cover of $\bOuv(k A_2)$ to a cover of $\bOuv(A)$. However, $A$ does not have any finite family of faithful functors. Therefore, $\{ k A_2 \to A_{P_1}, k A_2 \to A_{P_2}, k A_2 \to A_{S_2} \}$ is not a cover for the homotopical Zariski topology on $\SpecNC(k A_2)$. 

Elaborating further on this example (by considering homotopical epimorphisms of $A$ that are disjoint from the image of $\SpecNC(k A_2)$ into $\SpecNC(A)$), one can provide explicit instances of non-zero algebras whose derived free product over $\SpecNC(k A_2)$ with all $A_{P_1}, A_{P_2}, A_{S_2}$ is zero, giving a more direct proof of the fact that $\{ k A_2 \to A_{P_1}, k A_2 \to A_{P_2}, k A_2 \to A_{S_2} \}$ does not satisfy the cover condition for the homotopical Zariski topology. We leave these further details to the reader.
\end{exa}

\section{Towards a noncommutative notion of sheaf} \label{sec:gelfand}
%\section{A duality for the category of rings}
%\label{sec:gelfand}

Gelfand's duality realises a duality between the category of compact Hausdorff topological spaces and the category of commutative $C^*$-algebras. Some years after Gelfand's work, Grothendieck observed that, in order to obtain a similar duality for commutative rings, some extra data is needed. Indeed, for example, if we just regard the spectrum of a commutative ring as a topological space, we cannot tell any field apart from any other: all their spectra are isomorphic to the singleton topological space. The extra data required to realise a precise duality between commutative rings and their spectra is the structural sheaf of rings whose stalks at any point are local rings. The datum of a topological space equipped with such a structure sheaf is called a \emph{locally ringed space}. Grothendieck has shown that the category of affine schemes embeds fully faithfully into the category of locally ringed spaces, realising an algebraic version of Gelfand's duality.

Our goal is to achieve a similar result in the noncommutative setting using the notion of spectrum introduced above. Since we are in the homotopical setting, we would like to replace the classical notion of sheaf with its homotopical version commonly used in derived algebraic geometry. However, the noncommutative situation is more complicated for a number of reasons. For starters, we have considered two spectra so far: $\SpecNC(A)$ and $\SpecNC_\fine(A)$. 
Finally --- and most crucially --- one cannot directly use the theory of (homotopical) comonadic descent as in (derived) algebraic geometry. 

We shall now perform an examination of these issues in some detail.

\subsection{Descent and sheaf theory}

In this subsection we would like start investigating applications of sheaf theory to $\SpecNC(A)$ and $\SpecNC_\fine(A)$. The main issue we will face is that we can provide an example of noncommutative ring where the natural structure pre-sheaf is a sheaf on $\SpecNC(A)$ but its pullback (as a pre-sheaf) to $\SpecNC_\fine(A)$ is not a sheaf. Furthermore, its sheafification does not satisfy the properties a structure sheaf on $\SpecNC_\fine(A)$ should possess. This motives us to generalise the notion of sheaf using the theory of comonadic descent.

\

We start by discussing an example which shows that the natural definition of the structure pre-sheaf on $\SpecNC_\fine(A)$ is not compatible with the fine Zariski topology, in the sense that it is not a sheaf for this topology.

\begin{comment}
    
this does not work: for noncommutative rings the structure pre-sheaf is not a sheaf --- not even a homotopical one. Nevertheless, we do not view this as a shortcoming of our theory! Instead, this is where interesting new mathematics is happening. In this section we start unravelling these new phenomena. Before doing so, let us analyse a simple explicit example of such a non-sheafy behaviour of the structure pre-sheaf for noncommutative rings, to guide our intuition.

\end{comment}

\begin{exa} [Sheafyness problem]\label{exa:structure_not_sheaf}
    Let us consider again $A = k A_2$, the path algebra of the $A_2$ quiver. In subsection \ref{sec:noncommutative_spectra} we have seen that $X = \SpecNC_\fine(A) = \{p_1, p_2, p_3\}$ is the discrete space with three points. These three points, being open subsets, correspond to  three homotopical epimorphisms $A \to k$, $A \to M_2(k)$ and $A \to k$ (see Example \ref{exa:A_2_descent} below for more information about how to compute these homotopical epimorphisms). Thus, the structural pre-sheaf is given by
    \[ \cO_X(p_1) = k, \ \  \cO_X(p_2) = M_2(k), \ \ \cO_X(p_3) = k. \]
   If this pre-sheaf were a sheaf, then one would have that the cover of $X$ made by singletons  gives
    \[ \cO_X(X) = k \oplus M_2(k) \oplus k\,; \]
    but $\cO_X(X) = A$, and the latter has dimension $3$ as $k$-vector space.
\end{exa}

This shows that the natural definition of the structural pre-sheaf $\cO_X$ does not satisfy in general the sheaf condition on $\SpecNC_\fine(A)$ when $A$ is not commutative. However, we know that the covers of $X$ for the homotopical Zariski topology induce faithful, hence conservative, functors at the level of base changes of modules (see Proposition \ref{prop:faithful_covers}). We will show that this implies that such family of functors satisfy descent (in the homotopical sense). Therefore, there is still a local-to-global principle that the pre-sheaf $\cO_X$ over $\SpecNC_\fine(A)$ and modules over it satisfy. A
localization and reconstruction of the data of such pre-sheaves does not follow the rules of sheaf theory but those of comonadic descent (of which sheaf theory is a particular case). Let us elaborate on this point for the special case of the path algebra of the $A_2$ quiver, before addressing the abstract theory.

\begin{exa}[Descent for the quiver $A_2$]\label{exa:A_2_descent}
Let us examine the example of $A = k A_2$ and its localizations more closely. 
The algebra $A$ over a field $k$ is isomorphic to the algebra of lower triangular matrices of the form
\[ A = \begin{pmatrix}
k & 0 \\
k & k
\end{pmatrix}. \]
Let us denote a $k$-basis of $A$ by $e_1, e_2, e_3$, so that
\[
A = \begin{pmatrix}
k e_1 & 0 \\
k e_2 & k e_3
\end{pmatrix}.
\]
We denote 
\[ P_1 = k e_3, \ \ P_2 = k e_1 \oplus k e_2, \ \ S_2 = k e_1. \]
Note that $P_1$ and $P_2$ are projective left $A$-modules, $A \cong P_1 \oplus P_2$ as a left $A$-module, whereas $S_2$ has the left $A$-module structure induced by the quotient map $P_2 \to S_2$ and it is obviously indecomposable. These submodules correspond to indecomposable representations of $A_2$ and to generators of the (non-trivial) smashing subcategories of ${}_A \bH \bMod$, as discussed in subsection \ref{sec:noncommutative_spectra}. 
We also have the right $A$-modules
\[ Q_1 = k e_1, \ \ Q_2 = k e_2 \oplus k e_3, \ \ T_2 = k e_3\,, \]
whose description is specular to that of $P_1, P_2, S_2$. We also have two surjective epimorphisms of algebras
\[ A \to P_1 = k e_3 \]
and 
\[ A \to S_2 = k e_1. \]
It is straightforward to check that these are homotopical epimorphisms that correspond to the localizations at $P_1$ and $S_2$ considered in subsection \ref{sec:noncommutative_spectra}. 
    The third and last homotopical epimorphism of $A$ is given by the inclusion morphism 
\[ A \to M_2(k). \]
The $A$-bimodule structure on $M_2(k)$ is given by  regarding the latter as two copies of $P_2$ as a left module and two copies of $Q_2$ as right module. These actions are compatible with matrix multiplication of $M_2(k)$. Again, checking that this morphism is a homotopical epimorphism can be done by means of straightforward (albeit tedious) computations or one can appeal to more powerful abstract theories. We thus get the canonical morphism of $A$-algebras
\[ A \to k \oplus M_2(k) \oplus k \]
mentioned in Example \ref{exa:structure_not_sheaf}. If we denote the three localizations of ${}_A \bH \bMod$ corresponding to localizing at $P_1$, $P_2$ and $S_2$ by $i_1, i_2, i_3$ respectively, we get
\[ i_1(P_1) = P_1, \ \ i_1(P_2) = 0, \ \ i_1(S_2) = P_1\,, \]
and thus $\ker(i_1) = \lt P_2 \gt$. Here and further on, we denote by $\lt \minus \gt$ the smashing localization generated by the elements within the brackets. Analogously, we have
\[ i_2(P_1) = P_2, \ \ i_2(P_2) = P_2, \ \ i_2(S_2) = 0, \]
so that $\ker(i_2) = \lt S_2 \gt$ and
\[ i_3(P_1) = 0, \ \ i_3(P_2) = S_2, \ \ i_3(S_2) = S_2, \]
so that 
$\ker(i_3) = \lt P_1 \gt$. Concretely, $i_1$ is obtained by tensoring by the bimodule $k e_3$, $i_2$ by the bimodule $M_2(k)$, and $i_3$ by the bimodule $k e_1$. If we now go back to the discrete topological space $X = \SpecNC_\fine(A) = \{ p_1, p_2, p_3 \}$, we would expect the intersection of any two points to be empty. Indeed, if we regard $A \to P_1$ as corresponding to the point $p_1$, then one can check that
\[ P_1 \ast_A^\L P_2 \cong 0, \ \ P_1 \ast_A^\L S_2 \cong 0. \]
But the above formula show that
\[ P_1 \otimes_A^\L P_2 \cong 0, \ \ P_1 \otimes_A^\L S_2 \cong P_1. \]
So, $p_1$ and $p_3$ are not ``completely disjoint" inside ${}_A \bH \bMod$. This does not, however, prevent us from reconstructing $A$ (and hence every $A$-module) from the data of the three homotopical epimorphisms above. 

Let us define the functor 
\[ M: {}_A \bH \bMod \to {}_k \bH \bMod \oplus {}_{M_2(k)}\bH \bMod \oplus {}_k \bH \bMod   \]
given by $M(R) = (i_1(R), i_2(R), i_3(R))$, that is, the base change functor of the algebra $A \to k \oplus M_2(k) \oplus k$. Then, we have three adjunction maps: the map
\[ \eta_{P_1}: P_1 \to M(P_1) = i_1(P_1) \oplus i_2(P_1) \oplus i_3(P_1) = P_1 \oplus P_2\,,  \]
\[  a \mapsto (a, \alpha(a))\,, \]
where the inclusion $\a: P_1 \to P_2$ is defined as $e_1 \mapsto e_2$. Thus
\[ \coker(\a) \cong S_2. \]
The map
\[ \eta_{P_2}: P_2 \to M(P_2) = i_1(P_2) \oplus i_2(P_2) \oplus i_3(P_2) = P_2 \oplus S_2\,,\]
\[  a \mapsto (a, \beta(a))\,, \]
where $\be: P_2 \to S_2$ is the quotient map whose kernel is $\a$. And finally the map
\[ \eta_{S_2}: S_2 \to M(S_2) = i_1(S_2) \oplus i_2(S_2) \oplus i_3(S_2) = P_1 \oplus S_2\,, \]
\[  a \mapsto (0, a)\,. \]
Applying the functor again, we obtain
\[ M(M(A)) = M(M(P_1 \oplus P_2)) = M(P_1) \oplus M(P_2^2) \oplus M(S_2) =  P_1^2 \oplus P_2^3 \oplus S_2^3. \]
Now, we have two maps
\[ M(A) \rightrightarrows M(M(A))  \]
given by $\eta_{M(A)}$ and $M(\eta_A)$. 
The map 
\[ \eta_A: A = P_1 \oplus P_2 \to M(A) = P_1 \oplus P_2 \oplus P_2 \oplus S_2 \]
is given by
\[ (a, b) = (a, \a(a), b, \be(b)). \]
We have that the identity $\eta_A = \eta_{P_1} \oplus \eta_{P_2}$ gives
\[   M(\eta_A) = M(\eta_{P_1}) \oplus M(\eta_{P_2}) \]
and
\[ \eta_{P_1} = (\id_{P_1} \oplus \a) \circ \Delta_{P_1}\, \]
where $\Delta_{P_1}: P_1 \to P_1 \oplus P_1$ is the diagonal embedding. So,
\[ M(\eta_{P_1}) = M((\id_{P_1} \oplus \a) \circ \Delta_{P_1}) = (\id_{M(P_1)}, M(\a)) \circ \Delta_{M(P_1)}. \]
and
\[ M(\a) = (\id_{P_2}, 0). \]
So,
\[ M(\eta_{P_1})(a, b) = (a, b, b, 0). \]
Similarly, one can compute that
\[ M(\eta_{P_2}): P_2 \oplus S_2 \to P_2 \oplus S_2\oplus P_2 \oplus S_2 \]
\[ M(\eta_{P_2})(c, d) = (c, d, 0, d). \]
On the other hand,
\[ \eta_{M(A)}: P_1 \oplus P_2 \oplus P_2 \oplus S_2 \to P_1 \oplus P_2 \oplus P_2 \oplus S_2 \oplus P_2 \oplus S_2 \oplus P_1 \oplus S_2 \]
\[ \eta_{M(A)}(a, b, c, d) = (a, \a(a), b, \be(b), c, \be(c), 0, d). \]
Therefore,
\[ \eta_{M(A)} - M(\eta_A) = (0, \a(a) - b, 0, \be(b), 0, d - \be(c), 0, 0).  \]
All in all, we have the sequence
\[ A \stackrel{\eta_A}{\to} M(A) \stackrel{\eta_{M(A)} - M(\eta_A)}{\longrightarrow }{M(M(A))}\,.\]
Let us recall that
\[ \eta_A(a, b) = (a, \a(a), b, \be(b)) \]
is injective. So, it is clear that $\eta_A(A) \subset \ker(\eta_{M(A)} - M(\eta_A))$. The other inclusion is also easy to check. This shows that $A \cong \ker( \eta_{M(A)} - M(\eta_A) )$, so that the above sequence of morphisms is exact at $M(A)$. \end{exa}

The computations performed in Example~\ref{exa:A_2_descent} are not accidental: they can be explained through the theory of comonadic descent, which gives the correct formulas for reconstructing $A$-modules from their local components at a cover for the fine Zariski topology that is faithful.
In the remainder of this section, we shall partially develop the abstract theory. We shall limit ourselves to the algebraic aspects of the theory here, deferring the development of a fully fledged geometric theory until future work. Thus, to use the machinery from comonadic descent one needs to know that covers for the fine Zariski topology are faithful, \ie Conjecture~\ref{conj:faithfulness_conjecture}. From now on we assume this conjecture to be true, since it is verified in all examples we know.

\begin{ass} \label{ass:conjecture}
In what follows we assume that Conjecture \ref{conj:faithfulness_conjecture} is true.
\end{ass}

We would like to use covers for the fine Zariski topology to write the category of modules as a finite limit, as in \cite[Proposition 10.5]{Scholze}.
Let us immediately warn the reader that the latter proposition is stated in the language of $\infty$-categories, which would be the best language suited to the formulation of the results of this section. However, we will stick with our choice, made (and justified) earlier in the paper, of avoiding the language of $\infty$-categories in this work, even if at times this somewhat hinders the clarity of the results.

In our language, for the topological space $X = \SpecNC_\fine(A)$ with $A \in \bH \bRings_\Z$ and given homotopical epimorphism $A \to B_U$, \cite[Proposition 10.5]{Scholze}  provides sufficient conditions under which the association 
\[\SpecNC_\fine(B_U)= U \mapsto \bH \bMod_{B_U}\,, \]  is a homotopical sheaf of pre-triangulated categories.
This condition implies that the association
\[ U \mapsto B_U \]
is a homotopical sheaf of rings, which means, in turn, that given a cover $\{ A \to B_{U_i} \}_{i \in I}$ for the homotopical Zariski topology we have
\[ A \cong \R \lim N^\bullet(B_{U_i})\,. \]
Here $N^\bullet(B_{U_i})$ denotes the \u{C}ech nerve of the cover, often referred to
 as \emph{cobar construction}, given by the cosimplicial object
\[ \prod_{i \in I} B_{U_i} \rightrightarrows \prod_{i, j \in I} B_{U_i} \otimes_{A}^\L B_{U_j}  \mathrel{\substack{\textstyle\rightarrow\\[-0.6ex]
\textstyle\rightarrow \\[-0.6ex]
\textstyle\rightarrow}}\cdots. \]
Note that condition (i) in \cite[Proposition 10.5]{Scholze} essentially never holds unless $A$ is either a commutative ring or a cdga, because it would require the tensor product (of the underlying bimodules) of algebras to be commutative. In particular, it does not hold for the cover for the fine Zariski topology of $k A_2$ as we have seen in Example~\ref{exa:A_2_descent}, hence the failure to obtain a homotopical sheaf. But \cite[Proposition 10.5]{Scholze} has another interpretation in terms of comonadic descent. Indeed, its statement implies that, under the given hypothesis, the adjunction
\begin{equation} \label{eq:comonadic_adjunction}
\bH \bMod_A \leftrightarrows \bigoplus_{i \in I} \bH \bMod_{B_{U_i}} 
\end{equation}
defined by base change and restriction of scalars functors, is homotopically comonadic. The additional piece of information given by (the proof of) \cite[Proposition 10.5]{Scholze} is that under the proposition's hypothesis, the descent data is given by a finite diagram\footnote{The theory of $\infty$-categories makes this statement very clear. We do not comment further here.}. In our situation instead we still have homotopical descent, but we shall see that the descent data is no longer contained in a finite diagram, and the resulting cobar construction of modules and algebras over $A$ are not equal to the restriction of section of sheaves on the topological space $X$ as in the commutative case. It is always possible to reconstruct a module $M \in \bH \bMod_A$  uniquely from its localizations $M \otimes_A B_{U_i}$ if the cover is faithful, and, furthermore, any descent data on such a cover $\{ A \to B_{U_i} \}_{i \in I}$, formally given by a comodule over the comonad associated to the adjunction \eqref{eq:comonadic_adjunction}, comes from a module over $A$. Thus, our goal is to prove that the adjunction \eqref{eq:comonadic_adjunction} is comonadic. To do this we introduce the following generalisation (and adaptation) of the terminology introduced in \cite{Mathew} in the commutative setting.

\begin{defn} \label{defn:descendable_algebra}
Consider $A \in \bH \bRings_\Z$ and an $A$-algebra $A \to B$. We say that $B$ \emph{admits descent} or is \emph{descendable} if the smallest smashing subcategory of $\bH \bMod_A$ containing $B$ is the whole $\bH \bMod_A$.
\end{defn}

\begin{rmk} \label{rmk:zariski_covers_are_descendable}
Let us point out that it is easy to check that if $\{ A \to B_{i} \}_{i \in I}$ is a cover for the homotopy Zariski topology, then the $A$-algebra $A \to \prod_{i \in I} B_{U_i}$ is descendable. Indeed, by Proposition~\ref{prop:faithful_covers} we have that the family of functors $\{ \bH \bMod_A \to \bH \bMod_{B_i} \}$ is faithful and in particular conservative, which implies that the smallest smashing subcategory containing $\bH \bMod_{B_i}$ in $\bH \bMod_A$ --- and hence the one generated by $\prod_{i \in I} B_i$ --- is the whole $\bH \bMod_A$. For the fine Zariski topology the faithfulness of the covers are ensured by Assumption~\ref{ass:conjecture}.
\end{rmk}

Let us now show how to recontruct $\bH \bMod_A$ from the data of a descendable algebra $A \to B$. We will follow closely the discussion from \cite[Section 3]{Mathew}.

\begin{prop} \label{prop:cobar_limit_diagram}
An $A$-algebra $A \to B$ is descendable if and only if the canonical map
\[ A \to \R \lim N^\bullet(B) \]
is a quasi-isomorphism.
\end{prop}
\begin{proof}
    This is a generalisation of \cite[Proposition 3.20]{Mathew}. One can check that the proof never uses the commutativity of the algebras but only the canonical morphisms of the extension of scalars and restriction of scalars functors. Therefore, it can be applied also for the adjunction $\bH \bMod_A \leftrightarrows \bH \bMod_B$ verbatim.
\end{proof}

The next theorem shows how comonadic descent is related to homotopy limits.

\begin{thm} \label{thm:comonadic_descent}
If the $A$-algebra $A \to B$ is descendable, then the adjunction
\[ \bH \bMod_A \leftrightarrows \bH \bMod_B \]
is comonadic.
\end{thm}
\begin{proof}
It is enough to check that it is possible to apply the Lurie--Barr--Beck Theorem. This is done as in~\cite[Proposition 3.22]{Mathew}, where only abstract results from the theory of comonadic descent are used.
\end{proof}

\begin{rmk} \label{rmk:Balmer_descent}
Let us emphasise that we have been deliberately sketchy in the proofs of Theorem~\ref{thm:comonadic_descent} and Proposition~\ref{prop:cobar_limit_diagram}, because there exists another, more classical, way to obtain (essentially) the same results (under the stronger hypotheses that we will impose later on). Indeed, by the work of Balmer~\cite{Balmer}, we know that all faithful monads between triangulated categories satisfy descent\footnote{Here we employ Balmer's terminology on descent. Given a comonad $M: \bC \to \bC$, in Balmer's terminology, $M$ is said to satisfy \emph{descent} if the comparison functor $\bC \to \bCo \bMod_M$ is fully faithful and it is said to satisfy \emph{effective descent} if the comparison functor is an equivalence. We refer to \cite{Balmer} for details on these constructions. What we are calling comonadic descent in the present paper corresponds to Balmer's effective descent.} --- see \cite[Corollary 2.15]{Balmer}. Observe that Balmer's result can be stated by replacing the $\R \lim N^\bullet(B)$ with  the $\lim N^\bullet(B)$, which has the effect of ``switching" from the homotopical descent to the usual categorical descent. Therefore, Balmer's computations in \cite[Corollary 2.15]{Balmer} effectively represent the first step of the proof Theorem~\ref{thm:comonadic_descent}, under the faithfulness hypothesis. By iterating Balmer's computations, one obtains the full homotopical comonadic descent. Let us also emphasise that Balmer's computations yield the even stronger result that the complex obtained via comonadic descent in the faithful situation is not only exact but split exact.
\end{rmk}

In our situation, the meaning of Theorem~\ref{thm:comonadic_descent} is that the category $\bH \bMod_A$ can be reconstructed as a homotopical limit of the categories $\bH \bMod_{B_i}$. Thus, by following the prescription of comonadic descent, it is possible to reconstruct the global data (\ie $\bH \bMod_A$) from the local data (\ie the localizations $\bH \bMod_{B_i}$). But this reconstruction process does not give a sheaf in the usual sense on $\SpecNC_\fine(A)$, because in the noncommutative case the base change of algebras is not compatible with the base change of modules (concretely, Theorem \ref{thm:commutative_faithful_conjecture} fails). 

However, it is not difficult to check that it is still a well-defined pre-sheaf, as it is obvious by the property of stability of homotopical epimorphisms under composition. A less obvious statement is that for any object $M \in \bH \bMod_A$ the associatied pre-sheaf is well-defined. For the sake of completeness, we provide a proof of this basic fact below.

\begin{prop} \label{prop:associated_presheaf}
    Let $M \in \bH \bMod_A$. Then the association
    \[ \tilde{M}(U) = M \otimes_A^\L B_U \]
    is a pre-sheaf valued in $\bH \bMod_A$, on the site $\bOuv(\Spec(A))$, where $U$ is the open subset corresponding to the homotopical epimorphism $A \to B_U$.
\end{prop}
\begin{proof}
    Being a pre-sheaf just means being a functor $\bLoc(A) \to \bH \bRings_\Z$. So, the task at hand reduces to check that composition is well-defined. Let $A \to B_U$, $B_U \to C_V$ be two homotopical epimorphism. Then we have the natural isomorphism
    \[ \tilde{M}(V) = M \otimes_A^\L C_V \cong M \otimes_A^\L B_U \otimes_{B_U}^\L C_V\,. \]
    This yields the claim.
\end{proof}

From Proposition \ref{prop:associated_presheaf} it is not hard to deduce the following more refined result.

\begin{cor}\label{cor:NCpre-sheaf}
The association 
\[ U \mapsto \bH \bMod_{B_U} \]
defines a pre-sheaf\footnote{Technically this is a 2-to-1 pseudofunctor, because it is valued in the category of additive categories. We do not dwell into the details of the theory of 2-to-1 pseudofunctors here, because their use is not needed in the framework of $\infty$-categories.} of categories on $\bOuv(A)$.
\end{cor}

With this in mind, we are in the position to provide a new notion of \emph{descendable pre-sheaf of rings.}

\begin{defn} \label{defn:descendable_pre_sheaf}
Let $X$ be a topological space (or more generally a site).
A \emph{descendable  pre-sheaf of rings} $\cO_X: \bOuv(X)^\op \to \bH \bRings_\Z$ is a contravariant functor from the category of open sets of $X$ to homotopical rings $\bH \bRings_\Z$ satisfying the \emph{descent condition} that for any cover $\{U_i \to U\}_{i \in I}$ such that  $U = \bigcup_{i \in I} U_i$  we have
\[ \cO_X(U) = \R \lim N^\bullet \left( \cO_X({U_i})\right) \,, \]
where  $N^\bullet(\cO_X({U_i}))$ denotes the Amistur-\u{C}ech nerve of the cover given by the cosimplicial object
\[ \prod_{i \in I} \cO_X({U_i}) \rightrightarrows \prod_{i, j \in I} \cO_X({U_i}) \otimes_{\cO_X({U})}^\L \cO_X({U_j})  \mathrel{\substack{\textstyle\rightarrow\\[-0.6ex]
                      \textstyle\rightarrow \\[-0.6ex]
                      \textstyle\rightarrow}}\cdots. \]
\end{defn}

\begin{rmk}
    Observe that in Definition~\ref{defn:descendable_pre_sheaf} only the underlying $\cO_X$-module is reconstructed via descent, not the algebra structure of $\cO_X$. This is due to the fact that the Amistur-\u{C}ech complex is only a complex of modules and not of algebras, in general.
\end{rmk}

\begin{thm} \label{thm:descendable_presheaf}
The structure pre-sheaf on $\SpecNC(A)$ is a descendable pre-sheaf. Furthermore, under Assumption~\ref{ass:conjecture} the same holds on $\SpecNC_\fine(A)$.
\end{thm}
\begin{proof}
    Theorem~\ref{thm:comonadic_descent} and Proposition~\ref{prop:faithful_covers} directly imply the claim for $\SpecNC(A)$. Conjecture~\ref{conj:faithfulness_conjecture} allows one to apply Theorem~\ref{thm:comonadic_descent} to $\SpecNC_\fine(A)$, too. 
\end{proof}

In all examples described so far it is possible to check that the structure pre-sheaf is, in fact, a sheaf on $\SpecNC(A)$. This often happens for trivial reasons: in many examples the homotopical Zariski topology is the trivial topology on $\bOuv(A)$, and thus all pre-sheaves are sheaves. We conjecture the following.

\begin{conj} \label{conj:structure_sheaf}
The structure pre-sheaf on $\SpecNC(A)$ is a sheaf.
\end{conj}

Proving Conjecture~\ref{conj:structure_sheaf} amounts to applying Lurie-Barr-Beck to the adjunction
\begin{equation} \label{eq:adj_cover_rings}
\bH \bRings_A \leftrightarrows \prod_{i \in I} \bH \bRings_{B_i} 
\end{equation}
for a cover $\{ A \to B_i \}_{i \in I}$ for the homotopical Zariski topology. The main obstacle in doing so is that it requires a good understanding of the homological properties of the free algebras $A\{ x \} = \Z \lt x \gt \ast_\Z A$ introduced earlier --- an understanding which we are lacking at the moment.

\begin{rmk} \label{rmk:final_descendable}
    Let us make a final remark about descendable pre-sheaves. 
    A reason for introducing this notion is our analysis of the example of the quiver path algebra $k A_2$. For the latter, we have described the canonical morphism
    \[ \iota_{k A_2}: \SpecNC_\fine(k A_2) \to \SpecNC(k A_2) \]
    as the inclusion of the discrete topological space with three points as the three generic points of $\SpecNC(k A_2)$. In this case the structure pre-sheaf is a sheaf, because the homotopical Zariski topology is trivial on $\bOuv(k A_2)$ and hence all pre-sheaves on $\bOuv(k A_2)$ are sheaves. Pulling back via $\iota_{k A_2}$ just gives the same pre-sheaf on $\bOuv(k A_2)$ (because $\iota_{k A_2}$ is induced by the identity functor on $\bOuv(k A_2)$). Then, Example~\ref{exa:structure_not_sheaf} tells us that the resulting pre-sheaf is not a sheaf on $\SpecNC_\fine(k A_2)$ and its sheafification leads to a sheaf whose global sections is not $k A_2$. This latter feature is not desirable if we want to study the geometry of $k A_2$. But Example \ref{exa:A_2_descent} shows that the structure pre-sheaf is a descendable pre-sheaf on $\SpecNC_\fine(k A_2)$ and thus it satisfies a local-to-global gluing process that is different from the one commonly used in sheaf theory. In view of the above, we would argue that sheaf theory is not adequate for studying the geometry of $\SpecNC_\fine(k A_2)$, and an extension along the lines described in this section is required.
\end{rmk}

In this subsection we have briefly outlined what a noncommutative version of sheaf theory may look like. In the next subsection we will show that a suitable subcategory of the category of pre-ringed spaces is equivalent to the opposite category of $\bH \bRings_\Z$, thus realising a weak form of Gelfand's duality for noncommutative dga's --- and, in particular, for noncommutative rings.

\subsection{Towards noncommutative Gelfand's duality}

In this section we provide a first version of a noncommutative extension of Gelfand's duality to the category of rings. This extension is carried out in the algebraic setting of this paper; Gelfand duality for $C^*$-algebras will be discussed in a future work, where the machinery from \cite{bambi2}, \cite{BKK} and related works will be employed. We will also elaborate on why the current version of our result is only a first step towards a more refined theory which will be carefully developed elsewhere.

Let us start by briefly recalling some key aspects of the duality between affine schemes and commutative rings. The algebraic version of Gelfand's duality, \ie the duality between affine schemes and commutative rings developed by Grothendieck, requires one to endow the topological spaces resulting from the construction of the spectrum with extra-structure. Indeed, this step is necessary to obtain a perfect duality. The extra-structure in question for $X = \SpecG(A)$ to be equipped with is the structure sheaf of rings $\cO_X$ which allows one to discriminate between affine schemes that are indistinguishable as mere topological spaces. For example, any field extension $k \to l$ induces a homeomorphism of the underlying spectra $\SpecG(l) \to \SpecG(k)$. Therefore, the setting of ringed spaces is what is needed for the theory of affine schemes. But it is still not enough to obtain a perfect duality, because if we consider two spectra of commutative rings $(X, \cO_X)$ and $(Y, \cO_Y)$ regarded as ringed spaces, then there exist morphisms $(X, \cO_X) \to (Y, \cO_Y)$ that do not come from a morphism of commutative rings. For the benefit of the reader, we recall the following well known example of a morphism with this property. 

\begin{exa} \label{exa:morphism_ringed_spaces_not_form_crings}
Let $X = (\SpecG(\Q), \cO_{\SpecG (\Q)})$ and $Y = (\SpecG(\Z), \cO_{\SpecG (\Z)})$ be the classical affine schemes associated to $\Q$ and $\Z$, and consider the continuous map \[f_p: \star = \SpecG(\Q) \to \{ p \} \subset \SpecG(\Z)\] 
sending the point of $X$ to a closed point of $Y$ identified by a prime number $p$. We have that 
\[ f_p^{-1}(\cO_{\SpecG(\Z)}) = \limind_{U \supset \{p\}} \cO_Y(U) \cong \Z_{(p)}. \]
Therefore, the canonical inclusion map $\Z_{(p)} \to \Q$ defines a morphism of ringed spaces, but this morphism does not come from a map of rings because the only morphism of rings $\Z \to \Q$ sends $\SpecG(\Q)$ to the generic point. 
\end{exa}

The key observation here is that the map $\Z_{(p)} \to \Q$ is not a local map of local rings, which naturally leads one to consider the category of locally ringed spaces. This latter is a non-full subcategory of the category of ringed spaces that realises a perfect duality with commutative rings.

Another way to rephrase the above observation is that it is possible to construct a functor
\[ \SpecG: \bC\bRings^\op = \bAff \to \bTop \]
given by $A \mapsto \SpecG(A)$ which is not faithful, because it is not conservative: the functor $\SpecG$ loses some information about the category $\bAff$. Therefore, it must be enhanced to a functor 
\begin{equation} \label{eq:duality_ring_spaces}
(\SpecG, \cO): \bAff \to \mathbf{RingSp}
\end{equation} 
to the category of ringed spaces. This latter is faithful, so does not lose information, but it is not fully faithful. 

At this point, the task is to describe the essential image of this functor. This has led to the introduction of the category of locally ringed spaces for which Grothendieck proved that the functor 
\begin{equation} \label{eq:duality_lo_ring_spaces}
(\SpecG, \cO): \bAff \to \mathbf{LocRingSp} 
\end{equation} 
is fully faithful, so that the latter duality is 
a duality, \ie a fully faithful inclusion. 
In this work we do not achieve the goal of describing a perfect duality between the category of rings and a suitable category of (homotopically) ringed spaces, like in the case of \eqref{eq:duality_lo_ring_spaces}. What we can achieve is the less ambitious task of constructing a faithful functor like in the case of \eqref{eq:duality_ring_spaces}, leaving the task of describing the essential image of the functor we construct for future work. 

\

Let us consider a variation of the definition of ringed space. The issue we faced at the beginning of this section is that in the noncommutative case we do not know whether the association $U \mapsto \cO_U$ is a (homotopical) sheaf (although we have Conjecture \ref{conj:structure_sheaf} in this direction), but only a pre-sheaf (see Corollary~\ref{cor:NCpre-sheaf}).

Therefore, we can generalise the notion of ringed space to that of pre-ringed space.

\begin{defn} \label{defn:pre-ringed_space}
    A \emph{pre-ringed space} (or \emph{site}) is a topological space (or site) $X$ equipped with a structure pre-sheaf $\cO_X : \bOuv(X)^\op \to \bH \bRings_\Z$. A morphism of a pre-ringed spaces $(X, \cO_X) \to (Y, \cO_Y)$ is a pair given by a continuous map $f: X \to Y$ and morphism of pre-sheaves $\cO_Y \to f_*(\cO_X)$. We denote the category of pre-ringed spaces (resp. sites) by $\mathbf{PreRingSp}$ (resp. $\mathbf{PreRingSites}$).
\end{defn}

The construction given above defines a ringed site $(\bZar_A, \cO_A)$ for any $A \in \bH \bRings_\Z$. In the next remark we explain how to extend this construction to the topological space $\SpecNC(A)$.

\begin{rmk} \label{rmk:pre-sheaf_extension}
    Given an open subset $U \subset \SpecNC(A)$, one can always write $U = \bigcup_{i \in I} U_i$ where $U_i = \SpecNC(B_i)$ are open subsets associated to homotopical epimorphisms $A \to B_i$. Then we define
    \[  \cO_{\SpecNC(A)}(U) = \R \lim_{i \in I} N^\bullet(\cO_X(U_i)) \cong \R \lim_{i \in I} N^\bullet(B_i)\,, \]
    where $N^\bullet(B_i)$ is the Amistur--\u{C}ech complex introduced earlier when discussing the comonadic descent. Since for any affine cover $\bigcup_{j \in J} V_j = U_i$ we have that
    \[ \cO_{\SpecNC(A)}(U_i) = \R \lim_{j \in J} N^\bullet(\cO_X(V_i)) \,, \]
    \ie the value of $\cO_{\SpecNC(A)}(U_i)$ does not depend on the cover used to compute it (because by Theorem \ref{thm:descendable_presheaf} the structure pre-sheaf is descendable), a standard computation shows that also the value of $\cO_{\SpecNC(A)}(U)$ is independent of the chosen cover. The problem with this procedure is that the resulting pre-sheaf on $\SpecNC(A)$ is only a pre-sheaf of $A$-modules, so we cannot endow sections with the structure of a ring.
    Assuming the validity of Conjecture~\ref{conj:structure_sheaf} we can also extend $\cO_X$ to a sheaf of homotopical rings on all $\SpecNC(A)$ using the comonadicity of the adjunction \eqref{eq:adj_cover_rings}. We keep using the pre-ringed site $(\bZar_A, \cO_A)$, although conjecturally its use can be replaced by the pre-ringed space $(\SpecNC(A), \cO_A)$.
\end{rmk}

Observe that, unlike in the classical notion of ringed space (or site), in Definition~\ref{defn:pre-ringed_space} we allow the values of $\cO_X$ to be homotopical rings; if these are concentrated in degree $0$ we obtain a classical pre-sheaf of rings. The important feature brought about by Definition~\ref{defn:pre-ringed_space}, when put in the context of the discussion of duality made thus far, is expressed by the following lemma.

\begin{lemma} \label{lemma:ringed_full_into_pre_ringed}
The category of ringed spaces (resp. sites) is a full subcategory of the category of pre-ringed spaces (resp. sites)\footnote{We prove the lemma only for classical ringed spaces and pre-ringed spaces, but the same result holds for the category of homotopically ringed spaces replacing hom-sets with mapping spaces. We refrain to enter in the technical details of the homotopical version of the lemma because the best way to state and prove it is using the language of $\infty$-categories. We also emphasise that in the noncommutative situation the homotopical version of the lemma is the more relevant one.}.
\end{lemma}
\begin{proof}
We prove the lemma for ringed spaces; the same proof works for ringed sites as well.
Clearly, every ringed space is a pre-ringed space in a canonical way. We only need to check that the notions of morphism in the two categories agree. Let us consider two ringed spaces $(X, \cO_X) \to (Y, \cO_Y)$. Then, given a morphism of topological spaces $f: X \to Y$, the subset of 
\[ \Hom_{\mathbf{RingSp}}((X, \cO_X), (Y, \cO_Y)) \]
corresponding to $f$ is given by the set
\[ \Hom_{\bSh(Y)}(\cO_Y, f_*(\cO_X))\,. \]
But since the direct image of a sheaf is always a sheaf and the category of sheaves is a full subcategory of the category of pre-sheaves, we get a canonical bijection
\[ \Hom_{\bSh(Y)}(\cO_Y, f_*(\cO_X)) \cong \Hom_{\bPsh(Y)}(\cO_Y, f_*(\cO_X))\,. \]
But the latter is the subset of 
\[ \Hom_{\mathbf{PreRingSp}}((X, \cO_X), (Y, \cO_Y)) \]
corresponding to $f$, and we have the claim.
\end{proof}

In view of Lemma \ref{lemma:ringed_full_into_pre_ringed}, the 
classical duality issue of scheme theory of describing the essential image of the functor
\[ A \mapsto (\SpecG(A), \cO_{\SpecG(A)}) \]
can be formulated in an equivalent way in $\mathbf{RingSp}$ or $\mathbf{PreRingSp}$. The advantage is that the latter formulation makes sense also in the noncommutative case, where we do not know if the structure pre-sheaf is a sheaf. What remains to be checked in the noncommutative setting is that the resulting functor is faithful, \ie no information about the original category is lost. Our current version of Gelfand duality for dg-algebras precisely proves this claim.

\begin{thm} \label{thm:Gelfand_duality}
The (contravariant) functor 
\begin{equation} \label{eq:Gelfand_duality}
A \mapsto (\bZar_A, \cO_A)
\end{equation} 
from the category $\bH \bRings_\Z$ to the category $\mathbf{PreRingSites}$ is faithful.
\end{thm}
\begin{proof}
    Let $A, B \in \bH \bRings_\Z$. We need to check that the map
    \[ \Hom_{\bH \bRings_\Z}(A, B) \to \Hom_{\mathbf{PreRingSites}}((\bZar_B, \cO_{B}), (\bZar_A, \cO_{A})) \]
    is injective. Let $f, g: A \to B$ be two different maps. We have defined the maps of topological spaces $\tilde{f}, \tilde{g}: \SpecNC(B) \to \SpecNC(A)$ induced by $f$ and $g$ in Corollary \ref{cor::funct}, as the maps induced by the pseudofunctor of Proposition \ref{prop:funct}. It is immediate, by definition, that in both cases we have
    \[ f_*(\cO_{B})(X_A) = B, \qquad g_*(\cO_{B})(X_A) = B, \]
    where $X_A$ is the top element of the lattice of $\bZar_A$.
    Once again by definition, we have that the morphism
    \[ \cO_{A}(X_A) = A \to f_*(\cO_{B})(X_A) = B  \]
    is $f$, and the morphism 
    \[ \cO_{A}(X_A) = A \to g_*(\cO_{B})(X_A) = B \]
    is $g$. Thus the two morphisms of pre-ringed sites are different.
\end{proof}

Theorem \ref{thm:Gelfand_duality} has an immediate consequence, which can be stated in classical terms.

\begin{cor} \label{cor:Gelfand_duality}
    The category of rings is equivalent to a subcategory of the category $\mathbf{PreRingSites}$.
\end{cor}
\begin{proof}
    The inclusion $\bRings_\Z \to \bH \bRings_\Z$ is fully faithful, therefore the composition with the Gelfand duality of Theorem~\ref{thm:Gelfand_duality} is a faithful functor.
\end{proof}

\begin{rmk} \label{rmk:Gelfand_duality_for_rings}
Let us observe that, although the statement of Corollary \ref{cor:Gelfand_duality} is of classical nature, to make sense of it it is not enough to remain in the framework of rings but it is necessary to use the (homotopy) category of dg-algebras, because the natural structural pre-sheaf on $\SpecNC(A)$ is a homotopical pre-sheaf of dg-algebras not necessarily concentrated in degree $0$ (if $A$ is not commutative) --- see, \eg, Example~\ref{exa:homotopy_epimorphism}.
\end{rmk}

Theorem~\ref{thm:Gelfand_duality} can be rightfully viewed as the noncommutative analogue of Gelfand's duality, in that it describes $\bH \bRings_\Z^\op$ (and hence its full subcategory $\bRings_\Z^\op$) as a (non-full) subcategory of $\mathbf{PreRingSites}$. This is formally the same as the duality between commutative rings and affine schemes described as a (non-full) subcategory of $\mathbf{PreRingSp}$. In order for this to be useful in doing geometry, one should also have a description of the essential image of the functor \eqref{eq:Gelfand_duality}. This consists of a very specific family of pre-ringed sites, because we have shown that modules on $A$ (and hence the associated pre-sheaf of $\cO_{A}$-modules they determine) satisfy comonadic descent, \ie they have a pre-sheafy local-to-global behaviour. We postpone this study in a future work.

\

We conclude this subsection with an improvement on the description of the essential image of the functor \eqref{eq:Gelfand_duality}. This is achieved by using the following variation of the notion of ringed space.

\begin{defn} \label{defn:descendable_pre_ringed_space}
    A pre-ringed space $(X, \cO_X)$ (resp. site) is called \emph{descendable} if the structure pre-sheaf is a descendable pre-sheaf in the sense of Definition~\ref{defn:descendable_pre_sheaf}. We denote the full subcategory of $\mathbf{PreRingSp}$ (resp. sites) of descendable pre-ringed spaces (resp. sites) by $\mathbf{DescPreRingSp}$ (resp. $\mathbf{DescPreRingSites}$).
\end{defn}

\begin{thm}
\label{thm:descendable_Gelfand_duality}
The essential image of the Gelfand duality functor \eqref{eq:Gelfand_duality} lies in the subcategory $\mathbf{DescPreRingSites}$. 
\end{thm}
\begin{proof}
Theorem~\ref{thm:descendable_presheaf} immediately implies that $\cO_{A}$ is always a descendable pre-sheaf.
\end{proof}

\begin{rmk}
Under Conjecture~\ref{conj:structure_sheaf} the statement of Theorem~\ref{thm:descendable_Gelfand_duality} can be upgraded to a faithful embedding into the category $\mathbf{PreRingSp}$.
\end{rmk}

The notion of descendable pre-sheaf is a possible noncommutative generalisation of the notion of quasi-coherent sheaf, as the upcoming definition and proposition indicate.

\begin{defn} \label{defn:descendable_sheaf_O_X-modules}
    Let $(X, \cO_X)$ be a descendable pre-ringed space (resp. site). A pre-sheaf of right (or left) $\cO_X$-modules $\sF$ is said to be \emph{descendable} if for all covers $U = \bigcup_{i \in I} U_i$ we have
    \[ \sF(U) \cong \R \lim_{i \in I}  \left (\prod_{i \in I} \sF({U_i}) \rightrightarrows \prod_{i, j \in I} \sF({U_i}) \otimes_{\cO_X({U})}^\L \cO_X({U_j})  \mathrel{\substack{\textstyle\rightarrow\\[-0.6ex]
    \textstyle\rightarrow \\[-0.6ex]
\textstyle\rightarrow}}\cdots \right ). \]
We denote the full subcategory of the category of pre-sheaves of $\cO_X$-modules consisting of descendable pre-sheaves of $\cO_X$-modules by $\mathbf{Desc}$-$\cO_X$-$\bMod$. 
\end{defn}

\begin{prop} \label{thm:quasicohrent_nc}
    Consider the site $(\bZar_A, \cO_{A})$ of an object $A \in \bH \bRings_\Z$. There is an equivalence of categories $\mathbf{Desc}$-$\cO_X$-$\bMod \cong \bH \bMod_A$.
\end{prop}
\begin{proof}
Proposition~\ref{thm:comonadic_descent} immediately implies the claim.
\end{proof}

% Theorem~\ref{thm:quasicohrent_nc} suggests that the notion of descendable pre-sheaves of $\cO_X$-modules could be a good noncommutative generalisation of the notion of quasi-cohrent sheaf of algebraic geometry.

\subsection{Concluding remarks}

We would like to make some further remarks on the constructions described in this last section. So far we have been focussing on understanding the noncommutative aspects of our definitions, but they apply as well to commutative rings. It is clear that by definition our notions are compatible with the classical notions of algebraic geometry, because they have been designed to have this property. Nevertheless, we would like to make such observations more explicit.

\begin{prop} \label{prop:sheafy}
    Let $A \in \bH \bRings_\Z$ be commutative. The pre-ringed space $(\SpecNC_\fine(A), \cO_{A})$ is a (homotopically) ringed space.
\end{prop}
\begin{proof}
Theorem~\ref{cor:commutative_faithful_conjecture} implies that the Amistur-\u{C}ech complex of comonadic descent considered in Definition~\ref{defn:descendable_pre_ringed_space} agrees with the sheaf-theoretic \u{C}ech complex of $\cO_{A}$. Therefore, in this case being a descendable pre-sheaf is equivalent to being a (homotopical) sheaf\footnote{As already underlined, when $A$ is commutative all its localizations are commutative as well and concentrated in degree $0$. Therefore, $\cO_{\SpecNC_\fine(A)}(U)$ is a discrete commutative simplicial ring for all open subsets $U \subset \SpecNC(A)$ and hence the condition of homotopical sheafyness reduces to usual sheafyness.}.
\end{proof}

\begin{rmk} \label{rmk:sheafy_deg_0}
In many situations where we have that $A$ is discrete, \ie concentrated in degree $0$, and all its localizations are discrete as well (for example in the case of commutative Noetherian rings, see Example \ref{exa:homotopy_epimorphism}) Proposition \ref{prop:sheafy} implies that the space $(\SpecNC_\fine(A), \cO_{A})$ is an actual ringed space, not a homotopical one.
\end{rmk}

When $A$ is a discrete commutative ring, the structure of pre-sheaf of $\SpecNC_\fine(A)$ is compatible with that of $\SpecG(A)$.

\begin{prop} \label{prop:projection_spectra_prepringed}
Let $A$ be a discrete commutative ring.
The canonical map of topological spaces
\[ \pi_A: \SpecNC_\fine(A) \to \SpecG(A) \]
obtained in Proposition \ref{prop:compatision with Grothendieck} induces a canonical morphism of pre-ringed spaces
\[ \pi_A: (\SpecNC(A), \cO_{A}) \to (\SpecG(A), \cO_{A}) \]
that is a map of ringed spaces in the case when $A$ is Noetherian.
\end{prop}
\begin{proof}
We remark that since $\SpecG(A)$ is a spectral space, it is enough to specify the values of a sheaf on a base of compact open subsets. Therefore, for any affine open subset $U \subset \SpecG(A)$ we have that the map $\pi_A$ induces the identity morphism
\[ \cO_{A}(U) \stackrel{\cong}{\longrightarrow} (\pi_A)_*(\cO_{A})(U) = \cO_{A}(\pi_A^{-1}(U)) \]
that is obviously a morphism of pre-sheaves. The claim in the Noetherian case is a consequence of Proposition \ref{prop:sheafy}.
\end{proof}

We can finally also sketch how a full theory of derived noncommutative schemes based on the results above looks like. The notion of descendable pre-ringed space permits to give a reasonable definition of derived noncommutative scheme.

\begin{defn} \label{defn:derived_noncommutative_scheme}
    A \emph{derived noncommutative scheme} is a descendable pre-ringed site $(X, \cO_X)$ such that there is an open cover of open sub-pre-ringed sites $X = \bigcup_{i \in I} U_i$ such that $(U_i, \cO_{U_i}) \cong (\bZar_{A_i}, \cO_{A_i})$ for some $A_i \in \bH \bRings_\Z$.
\end{defn}

\begin{rmk}
Under Conjecture~\ref{conj:structure_sheaf}, Defnition~\ref{defn:derived_noncommutative_scheme} can be upgraded using the ringed spaces $(\SpecNC(A), \cO_{A})$ associated with affine noncommutative schemes.  
\end{rmk}

We emphasise that Definition~\ref{defn:derived_noncommutative_scheme} possibly gives a good notion of noncommutative scheme, but without a better understanding of the essential image of the functor \eqref{eq:Gelfand_duality} it is not clear what the correct notion of morphism between such spaces should be. Therefore, the first step to proceed in the study of derived noncommutative scheme theory is to complete the study of the noncommutative Gelfand duality obtained above.

%\section{An application to algebraic $K$-theory}

%\input{k-theory.tex}

%\section{Future outlook?}

%\begin{itemize}
%\item menzionare che ci sono molte altre possibili applicazioni.
%\item menzionare spazi glocabi, stack e altre topologie
%\item menzionare il caso analitico.
%\item menzionare che abbiamo 
%\item menzionare il caso degli $A^\infty$-rings e chromatic homotopy theory.
%\item menzionare il lavoro di Scholze in complex analytic geometry e anche i D-moduli.
%\item menzionare che il calcolo degli spettri di anelli noncommumtativi è interessante. Forse specialmente nel caso di finia presentazione relativamente ad un campo.
%\item ...
%\end{itemize}


\begin{thebibliography}{2}

\bibitem{ARV} 
{J. Ad\'amek, J. Rosick\'y, E. M. Vitale},  \textit{Algebraic theories: a categorical introduction to general algebra}, Cambridge Tracts in Mathematics \textbf{184}, Cambridge University Press, (2010).

\bibitem{AMSTV}  L. Angeleri H\"{u}gel, F. Marks, J. \u{S}tovi\u{c}ek, R. Takahashi, J. Vit\'oria, {\it Flat ring epimorphisms and universal localizations of commutative rings}, Q. J. Math. \textbf{71}(4), 1489--1520, (2020).

\bibitem{AS} M. Artin, W. Schelter, {\it Integral ring homomorphisms}, Adv.  Math. \textbf{39}(3), 289--329, (1981).

\bibitem{waldmann}
D. Bahns, S. Waldmann, \textit{Locally noncommutative Space-Times}, Rev. Math. Phys. \textbf{19}, 273--305, (2007). 

\bibitem{Balmer} P. Balmer, \textit{Descent in triangulated categories}, Math. Ann. \textbf{353}(1), 109--125, (2012).

\bibitem{bambi3} F. Bambozzi, O. Ben-Bassat, \textit{Dagger geometry as Banach algebraic geometry}, J.~Number Theory \textbf{162}, 391--462, (2016).

\bibitem{bambi4} F. Bambozzi, O. Ben-Bassat, K. Kremnizer, \textit{Stein domains in Banach algebraic geometry}, J. Funct. Anal. \textbf{274}(7), 1865--1927, (2018).

\bibitem{bambi2}
F. Bambozzi, O. Ben-Bassat, K. Kremnizer, \textit{Analytic geometry over $\mathbb{F}_1$ and the Fargues-Fontaine curve}, Adv. Math. \textbf{356}, 106815, (2019).


\bibitem{bambi}
F. Bambozzi, K. Kremnizer, \textit{On the sheafyness property of spectra of Banach rings}, J. London Math. Soc. \textbf{109}(1), e12855, (2024).

\bibitem{BamMi}
F. Bambozzi, T. Mihara, \textit{Homotopy epimorphisms and derived Tate’s acyclicity for commutative C*-algebras}, Quart. J. Math. \textbf{74}(2), 421--458, (2023).

\bibitem{BamMi2}
F. Bambozzi, T. Mihara, \textit{Derived Analytic Geometry for Z-Valued Functions Part I: Topological Properties}, Bull. Iranian Math. Soc. \textbf{50}(4), 58, (2024).

\bibitem{NCtori}
F. Bambozzi, S. Murro, \textit{On the uniqueness of invariant states}, Adv. Math. \textbf{376}, 107445, (2021).

\bibitem{NCtorus}
 F. Bambozzi, S. Murro, N. Pinamonti, \textit{Invariant states on noncommutative tori}, Int. Math. Res. Not. \textbf{2021}(5), 3299--3313, (2021).

\bibitem{Baz}
 S. Bazzoni, \textit{Cotilting modules and homological ring epimorphisms}, J.~Algebra \textbf{441}, 552--581, (2015).

\bibitem{BazSto}
S. Bazzoni, and J. Stovicek, \textit{Smashing localizations of rings of weak global dimension at most one}, Adv. Math. \textbf{305}, 351--401, (2017).

\bibitem{Bellissard1}
J. Bellissard, {\it The noncommutative geometry of aperiodic solids}, In: A. Cardona, S. Paycha, H. Ocampo (eds.), \textit{ Geometric and topological methods for quantum field theory }, World Sci. Publ. 86--156 (2003).

\bibitem{Bellissard2}
J. Bellissard, A. van Elst, H. Schulz-Baldes, \textit{The noncommutative geometry of the quantum Hall effect}, J. Math. Phys. \textbf{35}(10), 5373--5451, (1994)

% \bibitem{beck}
% J. M. Beck, \textit{Triples, algebras and cohomology}, Repr. Theory Appl. Categ. \textbf{2}, 1–-59, (2003).
% (reprinted)

\bibitem{Bei} 
A.A. Beilinson, \textit{Coherent sheaves on $\P^n$ and problems of linear algebra}, Funct. Anal. Appl. \textbf{12}(3), 214--216, (1978).

\bibitem{BKK} O. Ben-Bassat, J. Kelly, K. Kremnizer, \textit{A perspective on the foundations of derived analytic geometry}, preprint arXiv:2405.07936, (2024).

\bibitem{Laz} C. Braun, J. Chuang, A. Lazarev, \textit{Derived localisation of algebras and modules},  Adv. Math. \textbf{328}, 555--622, (2018).

\bibitem{Buchholz}
D. Buchholz, G. Lechner, S. J. Summers, 
\textit{Warped convolutions, Rieffel deformations and the construction of quantum field theories}, Comm. Math. Phys. \textbf{304}, 95--123,
(2011).

\bibitem{Cohn}
P. M. Cohn, \textit{The affine scheme of a general ring, Applications of sheaves}, Lecture Notes in Mathematics \textbf{753}, Springer Berlin, 197--211, (1979).

\bibitem{Laz2} 
J. Chuang, A. Lazarev, \textit{Homological epimorphisms, homotopy epimorphisms and acyclic maps}, Forum Math. \textbf{32}(6), 1395--1406, (2020).

\bibitem{connes1}
A. Connes, \textit{Noncommutative geometry and reality}, J. Math. Phys. \textbf{36}(11),  6194--6231, (1995).

\bibitem{connes2}
A. Connes,\textit{ On the spectral characterization of manifolds}, J. Noncomm. Geom. \textbf{7}, 1--82, (2013).

\bibitem{DeNittis}
G. De Nittis, H.  Schulz-Baldes,
\textit{The non-commutative topology of two-dimensional dirty superconductors}, J. Geom. Phys. \textbf{124}, 100--123, (2018).

\bibitem{DFR}
 S. Doplicher, K. Fredenhagen, and J.E. Roberts, \textit{The quantum structure of spacetime at the Planck scale and quantum fields}, Comm. Math. Phys. \textbf{172}, 187--220,  (1995). 

\bibitem{Dug} 
D. Dugger, \textit{A primer on homotopy colimits}, preprint available at \url{https://pages.uoregon.edu/ddugger/hocolim.pdf}.

\bibitem{Franco}
N. Franco, {\it Temporal Lorentzian spectral triples}, Rev. Math. Phys. \textbf{26}, 1430007, (2014).
  
 \bibitem{Froeb}
M.B. Fr\"ob, A. Much, K. Papadopoulos, \textit{Noncommutative geometry from perturbative quantum gravity},
Phys. Rev. D \textbf{107}, 064041, (2023).

\bibitem{gelfand}
I. Gelfand, \textit{Normierte Ringe.} Mat. Sb. Nov. Ser. \textbf{9}, 3--23, (1951).

\bibitem{Goldman}
O. Goldman, \textit{Rings and modules of quotients}, J. Algebra \textbf{13}, 10--47, (1969). 


\bibitem{GS} 
S. Gratz, G. Stevenson \textit{Approximating triangulated categories by spaces}, Adv. Math. \textbf{425}, 109073, (2023).

\bibitem{Grosse}
 H. Grosse, G. Lechner, \textit{Wedge-local quantum fields and noncommutative Minkowski space}, J. High Energy Phys. \textbf{11}, (2007).

\bibitem{Hov} M. Hovey, \textit{Model categories}, Mathematical Surveys and Monographs \textbf{63}, American Mathematical Society, (2007).


\bibitem{Stone} 
P.T. Johnstone, \textit{Stone spaces}, Cambridge University Press, (1982).

%\bibitem{khan}
%A. A. Khan, \textit{Descent in algebraic K-theory},  available at \url{https://www.preschema.com/lecture-notes/kdescent}.

% \bibitem{defquant}
% M. Kontsevich, \textit{Deformation quantization of Poisson manifolds I.},
% Lett. Math. Phys. \textbf{66}, 157--216, (2003).

\bibitem{Kra} H. Krause, \textit{Smashing subcategories and the telescope conjecture–an algebraic approach}, Invent. Math. \textbf{139}, 99--133, (2000).

\bibitem{KS} H. Krause, G.  Stevenson, \textit{The derived category of the projective line}. In: M. Baake, F. G\"otze, and W. Hoffmann (eds.),    \textit{Spectral Structures and Topological Methods in Mathematics}, 275--297, (2019).

\bibitem{Lawvere}
W. Lawvere, \textit{Functorial Semantics of Algebraic Theories}, Repr. Theory
Appl. Categ. \textbf{5}, 1--121, (2004).

\bibitem{HTT} 
J. Lurie, \textit{Higher topos theory}, Princeton University Press, (2009).

\bibitem{HA} 
 J. Lurie \textit{Higher algebra}, Available at \url{https://people.math.harvard.edu/~lurie/papers/HA.pdf}

\bibitem{Mac} S. Mac Lane, \textit{Categories for the working mathematician},  Springer Science \& Business Media \textbf{5}, (2013).

\bibitem{MM} S. Mac Lane, I. Moerdijk, \textit{Sheaves in geometry and logic: A first introduction to topos theory}, Springer Science \& Business Media, (2012).

\bibitem{Mathew} A. Mathew, \textit{The Galois group of a stable homotopy theory}, Adv. Math. \textbf{291}, 403--541, (2016).

\bibitem{smoothGelfand}
J.A. Navarro Gonz\'alez, J.B. Sancho de Salas, 
\textit{$C^\infty$-Differentiable Spaces},  Lecture Notes in Mathematics \textbf{1824}, Springer, (2003).

% \bibitem{Nee} A. Neeman, \textit{The Grothendieck duality theorem via Bousfield’s techniques and Brown representability}, J. Amer. Math. Soc. \textbf{9}(1), 205--236, (1996).

\bibitem{NB} A. Neeman, M. B\"okstedt, \textit{The chromatic tower for $D(R)$}, Topology \textbf{31}(3), 519--532, (1992).

\bibitem{NS} P. Nicolas, M. Saorin. \textit{Parametrizing recollement data for triangulated categories}, J. Algebra \textbf{322}(4), 1220--1250, (2009).

\bibitem{verch}
M. Paschke, R. Verch, \textit{Local covariant quantum field theory over spectral geometries}, Classical Quantum Gravity \textbf{21}, 5299--5316, (2004).

% \bibitem{Quillen} 
% D.G. Quillen, \textit{Homotopical algebra}, Lecture Notes in Mathematics \textbf{43}, Springer-Verlag, Berlin, (1967).

\bibitem{Rak} A.
Raksit, \textit{Hochschild homology and the derived de Rham complex revisited}, preprint arXiv:2007.02576, (2020).
 
\bibitem{nogo}
M. L. Reyes, \textit{Obstructing extensions of the functor spec to noncommutative rings}, 
Israel J. Math. \textbf{192}, 667--698, (2012).

\bibitem{RSS} C. Rezk, S. Schwede, B. Shipley, \textit{Simplicial structures on model categories and functors}, Amer. J. Math. \textbf{123}(3), 551--575, (2001).

\bibitem{rosenberg}
A.L. Rosenberg,\textit{ The left spectrum, the Levitzki radical, and noncommutative schemes}, Proc. Natl. Acad. Sci. USA \textbf{87}, 8583--8586, (1990).

\bibitem{Scholze} P. Scholze, \textit{Condensed mathematics}, Lecture notes based on joint work with D. Clausen. Available at \url{https://people.mpim-bonn.mpg.de/scholze/papers.html#Notes} (2019).

\bibitem{SS} S. Schwede, B. Shipley, \textit{Equivalences of monoidal model categories}, Algebr. Geom. Topol. \textbf{3}(1),  287--334, (2003).

\bibitem{SS2} S. Schwede, B. Shipley, \textit{Stable model categories are categories of modules}, Topology \textbf{42}(1), 103--153, (2003).

\bibitem{Sil} L. Silver,  \textit{Noncommutative localizations and applications}, J. Algebra \textbf{7}(1), 44--76, (1967).

\bibitem{Stack} \textit{The stacks project}, \url{https://stacks.math.columbia.edu/}.


\bibitem{toen} B.~To{\"e}n, \textit{Champs affines}, Selecta Math. (N.S.) \textbf{12}, 39--134, (2006).

%\bibitem{toen1}
%B.~To{\"e}n, G.~Vezzosi,
%{\it Homotopical algebraic geometry I: topos theory},
%Adv. Math.  193 (2), 257--372 (2005).

\bibitem{TV2}
B.~To{\"e}n, G.~Vezzosi,
{\it Homotopical Algebraic Geometry II: Geometric Stacks and Applications}, Mem. Amer. Math. Soc. \textbf{193}, (2008).


\bibitem{TV3}
B.~To{\"e}n, G.~Vezzosi, {\it ``Brave New" Algebraic Geometry and global derived moduli spaces of ring spectra}, in: H. R. Miller, D. C. Ravenel (eds.), \textit{Elliptic Cohomology: Geometry, Applications, and Higher Chromatic Analogues}, London Mathematical Society Lecture Note Series \textbf{342}, Cambridge University Press, (2007).

\bibitem{Wei}
C. A. Weibel, \textit{An introduction to homological algebra}, Cambridge studies in advanced mathematics  \textbf{38}, Cambridge University Press, (1994).

%\bibitem{kbook}
%C. A. Weibel, \textit{The $K$-book: An Introduction to Algebraic $K$-theory},  Vol. 1American Mathematical Soc. 45, (2013).

\bibitem{Whi} D. White, \textit{Model structures on commutative monoids in general model categories}, J. Pure Appl. Algebra \textbf{221}(12), 3124--3168, (2017).

\end{thebibliography}
\end{document}